\documentclass{siamart220329}
\usepackage[utf8]{inputenc}

\usepackage{latexsym,amsmath,amssymb,stmaryrd,graphicx,hyperref}
\usepackage{wasysym,geometry}
\usepackage{mathtools}
\usepackage{amsmath,amssymb,mathrsfs}
\usepackage{algpseudocode}
\usepackage{todonotes}
\usepackage{tikz}

\newsiamremark{rem}{Remark}
\newsiamremark{ex}{Example}

\newcommand{\blockcomment}[1]{\State \(\triangleright\) \textit{#1}}
   
  \DeclareMathOperator{\Pol}{Pol}
  \DeclareMathOperator{\mi}{mi}
  \DeclareMathOperator{\mx}{mx}
  \DeclareMathOperator{\pp}{pp}
  \DeclareMathOperator{\lex}{lex}
    \DeclareMathOperator{\lele}{ll}
  \DeclareMathOperator{\End}{End}
  \DeclareMathOperator{\Aut}{Aut}
  
  \DeclareMathOperator{\Sym}{Sym}
  \DeclareMathOperator{\id}{id}
  \DeclareMathOperator{\Csp}{CSP}

  \DeclareMathOperator{\ve}{ve}
  \DeclareMathOperator{\cm}{cm}
  \DeclareMathOperator{\Cr}{cr}

  \newcommand{\Q}{\mathbb Q}
  \newcommand{\bA}{\mathfrak A}
  \newcommand{\bB}{\mathfrak B}
  \newcommand{\bC}{\mathfrak C}
  \newcommand{\bD}{\mathfrak D}
  
  \newcommand{\bL}{\mathfrak L}
  \newcommand{\bI}{\mathfrak I}
  \newcommand{\bU}{\mathfrak U}
  
  \newcommand{\bX}{\mathfrak X}
  
  \def\pa{\hbox{\sf p}}
  \def\pu{\hbox{\sf p${}^{\smallsmile}$}}
  \def\ma{\hbox{\sf m}}
  \def\allen-mi{\hbox{\sf m${}^{\smallsmile}$}}
  \def\oa{\hbox{\sf o}}
  \def\oi{\hbox{\sf o${}^{\smallsmile}$}}
  \def\da{\hbox{\sf d}}
  \def\di{\hbox{\sf d${}^{\smallsmile}$}}
  \def\sa{\hbox{\sf s}}
  \def\si{\hbox{\sf s${}^{\smallsmile}$}}
  \def\fa{\hbox{\sf f}}
  \def\fu{\hbox{\sf f${}^{\smallsmile}$}}
  
  \newcommand{\xp}{\ensuremath{X^+}}
  \newcommand{\xm}{\ensuremath{X^-}}
  \newcommand{\yp}{\ensuremath{Y^+}}
  \newcommand{\ym}{\ensuremath{Y^-}}

  \newcommand{\spii}{\ \ }
  \newcommand{\spiii}{\ \ \ }
  \newcommand{\spiv}{\ \ \ \ }
  \newcommand{\spv}{\ \ \ \ \ }

\title{Complexity Classification Transfer for CSPs\\ via Algebraic Products
\thanks{Manuel Bodirsky and \v{Z}aneta Semani\v{s}inov\'{a} have received funding from the ERC (Grant Agreement no. 101071674, POCOCOP) and from the DFG (Project FinHom, Grant 467967530). Views and opinions expressed are however
those of the authors only and do not necessarily reflect those of the European Union or the European Research
Council Executive Agency.
Peter Jonsson is partially supported by the Swedish Research Council (VR) under grants 2017-04112 and 2021-04371. Barnaby Martin is supported by EPSRC grant EP/X03190X/1.}}

\author{
Manuel Bodirsky\thanks{TU Dresden, Germany, (\email{manuel.bodirsky@tu-dresden.de}, \email{zaneta.semanisinova@tu-dresden.de})} \and
Peter Jonsson\thanks{Link\"{o}ping University, Sweden, (\email{peter.jonsson@liu.se})} \and
Barnaby Martin\thanks{Durham University, UK, (\email{barnaby.d.martin@durham.ac.uk})} \and
Antoine Mottet\thanks{Hamburg University of Technology, Research Group for Theoretical Computer Science, Germany, (\email{antoine.mottet@tuhh.de})} \and
\v{Z}aneta Semani\v{s}inov\'a\footnotemark[2]
}

\headers{Complexity Classification Transfer for CSPs via Algebraic Products}{M. Bodirsky, P. Jonsson, B. Martin, A. Mottet, \v{Z}. Semani\v{s}inov\'{a}}

\begin{document}
\maketitle

\begin{abstract}
We study the complexity of infinite-domain constraint satisfaction problems:
our basic setting is that a complexity classification for the CSPs of first-order expansions of a structure $\bA$ 
can be transferred to a classification of the CSPs of first-order expansions of another structure $\bB$. 
 We exploit a product of structures (the {\em algebraic product}) that corresponds to the product of the 
 respective polymorphism clones and present a complete complexity classification of the CSPs for first-order expansions of
the $n$-fold algebraic power of $({\mathbb Q};<)$.
This is proved by various algebraic and logical methods in combination with knowledge of the polymorphisms of the tractable first-order expansions of $({\mathbb Q};<)$ and explicit descriptions of the expressible relations in terms of syntactically restricted first-order formulas.
By combining our classification result with general classification transfer techniques, we obtain surprisingly strong new classification results for highly relevant formalisms such as 
Allen’s Interval Algebra, the $n$-dimensional Block Algebra, and the Cardinal Direction Calculus, even if higher-arity 
relations are allowed. 
Our results confirm the infinite-domain tractability conjecture for classes of structures that 
have been difficult to analyse with older methods.
For the special case of structures with binary signatures, the results 
can be substantially
strengthened and tightly connected to Ord-Horn formulas; this solves several longstanding open problems from the AI literature.
\end{abstract}

\begin{keywords}
constraint satisfaction, temporal reasoning, computational complexity, polymorphisms, universal algebra, polynomial-time tractability
\end{keywords}

\begin{MSCcodes}
06A05, 68Q25, 08A70
\end{MSCcodes}

\section{Introduction}

This introductory section is divided into three parts where we describe the background, present our contributions, and provide
an outline of the article, respectively.

\subsection*{Background}

Constraint satisfaction problems (CSPs) are computational problems that appear in many areas of computer science, for example in temporal and spatial reasoning in artificial intelligence~\cite{Qualitative-Survey} or in database theory~\cite{KolaitisVardi,OBDA}.
The computational complexity of CSPs is of central interest in these areas, and a general research goal is to obtain systematic complexity classification results, in particular about CSPs that are in P and CSPs that are NP-hard. 
CSPs can be described elegantly by fixing a structure with a finite relational signature, the \emph{template}; the computational task is to determine whether a given finite input structure has a homomorphism to the template. 
A breakthrough result was obtained independently by Bulatov~\cite{BulatovFVConjecture} and by Zhuk~\cite{ZhukFVConjecture,Zhuk20}, which confirmed the famous Feder-Vardi conjecture~\cite{FederVardi}: every CSP over a finite template (i.e., a structure with a finite domain) is in P, or it is NP-complete.
Moreover, given the template it is possible to decide algorithmically whether its CSP is in P or whether it is NP-complete. 

Most of the CSPs in temporal and spatial reasoning can \emph{not} be formulated as CSPs with a finite template. The same is true for many of the CSPs that appear in database theory (e.g., most of the CSPs in the logic MMSNP, which is a fragment of existential second-order logic introduced by Feder and Vardi~\cite{FederVardi}, and which is important for database theory~\cite{OBDA}, cannot be formulated as CSPs with a finite template~\cite{MadelaineStewartSicomp}). For CSPs with infinite templates we may not hope for general classification results~\cite{BodirskyGrohe}; however, 
we may hope for general classification results if we restrict our attention to classes of templates that are model-theoretically well behaved.
An example of such a class is the class of all structures with domain ${\mathbb Q}$ where all relations are definable with a first-order formula over the structure $({\mathbb Q};<)$. This class of structures is of fundamental interest in model theory, and,  by a result of Cameron~\cite{Cameron5}, also in the theory of infinite permutation groups (it is precisely the class of all countable structures with a highly set-transitive automorphism group).
The CSPs for such structures have been called 
  \emph{temporal CSPs} because they include many CSPs that are of relevance in temporal reasoning, such as the Betweenness problem~\cite{Opatrny}, the And/Or scheduling problem~\cite{and-or-scheduling}, or the satisfiability problem for Ord-Horn constraints~\cite{Nebel}. The complexity of temporal CSPs has been classified by Bodirsky and Kára~\cite{tcsps-journal}, and a temporal CSP is either in P or it is NP-complete. 

Over the past 10 years, many classes of infinite structures have been classified with respect to the complexity of their CSP. We may divide these results into \emph{first-} and \emph{second-generation} classifications. First-generation classifications, such as the classification of temporal CSPs mentioned above, typically use concepts from universal algebra and Ramsey theory, and essentially proceed by a combinatorial case distinction~\cite{BodPin-Schaefer-both,BMPP16,Kompatscher:VanPham:flap2018,Phylo-Complexity}. 
Second-generation classifications also use universal algebra and Ramsey theory, but they eliminate large parts of the combinatorial analysis by using arguments for finite structures and the Bulatov-Zhuk theorem (the first example following this approach is~\cite{BodMot-Unary}; other examples include~\cite{MMSNP-Journal,MottetPinskerSmooth,BodirskyBodorUIP,Collapses,BodirskyKnaeAAAI}). 
So the idea of second-generation classifications is to \emph{transfer} the finite-domain classification to certain tame classes of infinite structures. 

Complexity transfer is also the topic of the present article; however, we transfer classification results not from finite structures to classes of infinite structures, but between classes of infinite  structures. The key to systematically relating  many classes of infinite structures are various \emph{product constructions}, and \emph{logical interpretations}. Examples are \emph{Allen's Interval Algebra}~\cite{Allen} from temporal reasoning, which has a first-order interpretation in $({\mathbb Q};<)$, or the \emph{rectangle algebra}~\cite{Guesgen,MukerjeeJ90} and the $r$-dimensional \emph{block algebra}~\cite{BlockAlgebra}.
These links extend to links between fragments of the respective formalisms. In order to also establish links between the complexity of the respective CSPs, the logical interpretations must use \emph{primitive positive} formulas, rather than full first-order logic. 
There are various notions of products of
constraint formalisms that have been studied in the literature; see~\cite{WestphalWoelfl,BulatovAmalgams}. 
In this article we use a product of structures known as the \emph{algebraic product}; it corresponds to the product of the respective polymorphism clones, which is essential for the universal algebraic approach.

\subsection*{Contributions}
This article contains both theoretical results and applications of these results to well-studied formalisms and open problems in the area. 
Our first main
contribution is a complete complexity classification for the CSPs of first-order expansions of
$({\mathbb Q};<) \boxtimes ({\mathbb Q};<)$, i.e., 
the algebraic 
product of $({\mathbb Q};<)$ with itself. We then generalise this result to
first-order expansions of 
finite algebraic powers of $(\Q; <)$, denoted by $({\mathbb Q};<)^{(n)}$. In the proof we use known results about first-order expansions of $({\mathbb Q};<)$
combined with a mix of algebraic and of logical arguments. On the algebraic side, we use the fact that the
first-order expansions of $({\mathbb Q};<)$ with a tractable CSP have certain polymorphisms. 
On the logic side, we use highly informative descriptions of the relations of the templates using syntactically restricted forms of first-order logic, so-called (weakly) $i$-{\em determined clauses}. These syntactic forms partially allow us to separate the relations coming from different factors of the algebraic product. The combination of algebraic and syntactic methods turned out to be very powerful in our setting and we believe that it will be fruitful for analysing first-order expansions of products of other structures. 

Together with a general classification transfer result from~\cite{Book}, we then obtain a sequence 
of new complexity classification results for classes of CSPs that have been studied in temporal and spatial reasoning. We derive our applications in two steps: we first derive classification results for structures with relations of arbitrary arity. With little extra effort, we then obtain stronger results for the special case that all relations are binary.
The restriction to binary relations is very common in the AI and qualitative reasoning
literature: the influential survey by Dylla et al.~\cite{TemporalSpatialSurvey} lists 50+ formalisms based on binary relations but just a handful
of formalisms that use relations with higher arity.
One reason for this is that binary relations in the infinite-domain regime are very powerful and have attracted much attention. Another reason is pointed out in~\cite[Section 3]{TemporalSpatialSurvey}: generalizations to higher-arity relations are indeed useful in practice, but progress has been hampered by a lack of algebraic understanding and algorithmic methods. The results in this article provide a step towards a better understanding of AI-relevant formalisms and their generalisation to non-binary relations.

\medskip

\noindent
{\bf Templates with relations of unrestricted arity.}
We determine the complexity of the CSP for first-order expansions of the basic relations in
three influential formalisms for spatio-temporal reasoning: Allen's Interval Algebra~\cite{Allen}, the Block Algebra (BA)~\cite{BlockAlgebra}, and
the Cardinal Direction Calculus (CDC)~\cite{LigozatCDC}.
Allen's interval algebra is a fundamental formalism within AI and qualitative reasoning that has a myriad of applications, e.g. in automated
planning~\cite{Allen:Koomen:ijcai83,Mudrova:Hawes:icra2015,Pelavin:Allen:aaai87}, natural language processing~\cite{Denis:Muller:ijcai2011,Song:Cohen:aaai88} and molecular biology~\cite{GolumbicShamir}. Both BA and CDC can be viewed as variations of Allen's algebra originally aimed at expanding the range of applicability. For instance, the BA can handle directions in a spatial reasoning setting (something that is
difficult in standard formalisms such as RCC~\cite{RandellCuiCohn}) with diverse applications such as computer vision~\cite{Cohn:etal:kr2012}, architecture~\cite{Regateiro:etal:aei2012},
and physics simulation in computer games~\cite{Zhang:Renz:kr2014}.
CDC have found applications in geographic information systems and image interpretation: see \cite{Chang:Jungert:image,Frank:ijgis96,LigozatCDC} and the references therein.
These three formalisms have been very important in the evolution of 
calculi for qualitative reasoning, 
for the development of methods
for complexity classifications, and for the algebraic theory of infinite-domain CSPs.
For instance, two milestones in complexity classification of infinite-domain CSPs
are concerned with Allen's algebra: Nebel and Bürckert's~\cite{Nebel} result for
subsets of binary Allen relations containing all basic relations,
and Krokhin et al's~\cite{KrokhinAllen} generalization to arbitrary subsets of binary Allen relations. BA and CDC, on the other hand, have presented significant challenges and resisted full complexity classifications for at least
twenty years.
We conclude that Allen's Interval Algebra, BA and CDC were important test cases for CSP complexity classification projects
long before any complexity classification conjectures for infinite-domain CSPs had been formulated.

In these particular cases, our results show that
the so-called \emph{infinite-domain tractability conjecture},
which is formulated for all reducts of finitely bounded homogeneous structures~\cite{BPP-projective-homomorphisms}, holds.
The conjecture states that 
such a structure has a polynomial-time tractable CSP unless the structure admits a primitive positive interpretation of a structure which is homomorphically equivalent to $K_3$, the clique with three vertices (note that the CSP for the template $K_3$ is the 3-colourability problem, which is a well-known NP-complete problem). 
This hardness condition is known to be equivalent to the structure admitting a primitive positive construction of $K_3$ \cite{wonderland}.
All the classes of infinite structures discussed so far 
are first-order interpretable over $({\mathbb Q};<)$ and it can be shown that they
fall into the scope of this conjecture (see, for instance, Theorem 4 from \cite{MottetPinskerCores}, Lemma~3.5.4 and Proposition~4.2.19 from~\cite{Book}, and Lemma 3.8~from \cite{wonderland}).
Interestingly, the structures we treat here are notoriously difficult for the methods underpinning second-generation classification results: e.g., the unique interpolation property usually fails in this 
context~\cite{BodirskyBodorUIP}. 

To make progress with proving the  infinite-domain tractability conjecture, one strategy is to verify it on larger and larger classes of structures. Highly useful restrictions on classes of interesting structures come from model theory. The concept of \emph{stability} and, more generally, \emph{NIP} (i.e., not having Shelah's {\em independence property}~\cite{Shelah:aml71}) are central concepts in model theory (see, e.g.,~\cite{Simon,Chernikov}). While the consequences of these concepts for the complexity of constraint satisfaction are unclear, stability and NIP are still relevant here, because in combination with homogeneity in a finite relational signature they have strong consequences and allow for model-theoretic classification results, which in turn can be the basis for CSP classification results.

We would like to stress the particular role of structures with a first-order interpretation over $({\mathbb Q};<)$ in this context. All of these structures are NIP. Moreover, it is known that every homogeneous structure with a finite relational signature which is stable has a first-order interpretation in $({\mathbb Q};<)$~\cite{Lachlan-Tree-Decomp}. 
These structures are not only important in model theory, but also significant algorithmically, because the cases with a polynomial-time tractable CSP usually cannot be characterised by canonical polymorphisms, and new algorithms rather than polynomial-time reductions to tractable finite-domain CSPs are needed~\cite{BodirskyKara,RydvalDescr,MottetPinskerSmooth}. 
Therefore, a complexity classification for CSPs of structures with a first-order interpretation in $({\mathbb Q};<)$ would be an important milestone for resolving the tractability conjecture. 
Since our syntactic approach to complexity classification overcomes the mentioned challenges in important cases, it represents a step towards this goal.

\medskip

\noindent
{\bf Templates with binary relations.}
Our results concerning first-order expansions of $({\mathbb Q};<)^{(n)}$
can be specialised to the case when only binary relations are allowed.
If $\mathfrak{D}$ is such a structure, then our results imply
that CSP$(\mathfrak{D})$ is in P if and only if every relation in $\mathfrak{D}$ can be defined by
an Ord-Horn formula~\cite{Nebel}.
This allows us to answer
several open questions from the AI literature.
In particular, we solve 
an open problem from 2002 about the $n$-dimensional cardinal direction calculus~\cite{Balbiani:Condotta:appint2002} (Section~\ref{sec:cdc}),
an open problem from 1999 about fragments of the rectangle algebra~\cite{Balbiani:etal:ijcai99} (Theorem~\ref{thm:proof-of-conjecture}) 
and an open problem from 2002 about the $n$-dimensional block algebra~\cite{BlockAlgebra} (Corollary~\ref{cor:block-binary}). We can also answer another question in \cite{BlockAlgebra} about integration of the tractable cases into tractable formalisms that can also handle metric constraints; see the discussion at the end of Section~\ref{sect:rect}. 
Finally, we obtain short new proofs of known results about reducts of Allen's Interval Algebra (Section~\ref{sect:allen}). 
Our results typically answer more general questions than those asked in the publications above, in particular, they yield results also in the case when the relations are of arity higher than two.

\subsection*{Outline}
The structure of the article is as follows. Section~\ref{sect:csp}
contains the basic concepts that are needed for a formal definition of the CSPs, and some facts about constraint satisfaction problems and their computational complexity. Section~\ref{sect:algprod} contains the definition of the algebraic product together with
some related results.
In Section~\ref{sect:products}, we study $({\mathbb Q};<)^{(n)}$, and this ultimately provides us with
a complexity classification of the CSP for first-order expansions of $({\mathbb Q};<)^{(n)}$.
We additionally study the restriction to binary signatures in this section, i.e., signatures where all relations have arity at most two.
The next section is devoted to a condensed introduction to complexity classification transfer.
Thereafter, we combine the complexity results for $({\mathbb Q};<)^{(n)}$ with complexity classification
transfer in order to analyse various spatio-temporal formalisms
in Section~\ref{sect:applications}.
We conclude the article with a brief discussion of the results together with some
possible future research directions (Section~\ref{sect:conc}).

Some of the results in this article have been announced in a conference paper~\cite{ClassificationTransfer}.
However, one of the central proofs there 
(Lemma 2) is not correct. The proof in the present article avoids proving Lemma 2 from \cite{ClassificationTransfer} and
is entirely new; in particular, the syntactic approach to analysing first-order expansions of products in Section~\ref{sect:decomp} did not appear in the old approach.

\section{Constraint Satisfaction Problems} \label{sect:csp}

In this section we introduce basic concepts that are needed for a formal definition of the class of constraint satisfaction problems (CSP) together with some basic facts about CSPs and their computational complexity.

\subsection{Basic Definitions}

Let $\tau$ be a \emph{relational signature}, i.e., a set of \emph{relation symbols} $R$, each equipped with an \emph{arity} $k \in {\mathbb N}$. A \emph{$\tau$-structure} $\bA$ consists of a set $A$, called the \emph{domain} of $\bA$, 
and a relation $R^{\bA} \subseteq A^k$ for each relation symbol $R \in \tau$ of arity $k$.
A structure is called \emph{finite} if its domain is finite. Relational structures are often written like $(A;R_1^{\bA},R_2^{\bA},\dots)$, with the obvious interpretation; for example, $({\mathbb Q};<)$ denotes the structure whose domain is the set of rational numbers ${\mathbb Q}$ and which carries a single binary relation $<$ which denotes the usual strict order of the rationals. Sometimes, we do not distinguish between the symbol $R$ for a relation and the relation $R^{\bA}$ itself. 
Let $\bA$ be a $\tau$-structure and let $\bA'$ be a $\tau'$-structure with
$\tau \subseteq \tau'$. If $\bA$ and $\bA'$ have the same domain and $R^{\bA} = R^{\bA'}$
for all $R \in \tau$, then $\bA$ is
called a $\tau$-\emph{reduct} (or simply \emph{reduct}) of $\bA'$, and $\bA'$ is called a $\tau'$-\emph{expansion} (or simply
\emph{expansion}) of $\bA$. If $R$ is a relation over the domain of $\bB$, then we let $(\bA;R)$
denote the expansion of $\bA$ by $R$.

We continue by introducing some logical terminology and machinery.
We refer the reader
to~\cite{Hodges} for an introduction to first-order logic. 
An \emph{atomic $\tau$-formula} is a formula of the form $x=y$, $R(x_1,\dots,x_n)$, or the form $\bot$, where $x_1,\dots,x_n,x,y$ are variables, $R$ is a symbol from $\tau$, and $\bot$ is a symbol that stands for `false'.
Let $\bB$ denote a $\tau$-structure.
If $\psi$ is a sentence (i.e. a first-order formula without free variables), then we write $\bB \models \psi$ 
to denote that $\bB$ is a model of (or satisfies) $\psi$.
One can use first-order formulas over the signature $\tau$ to define relations over $\bB$: if $\phi(x_1,\dots,x_n)$ is a first-order $\tau$-formula with free variables $x_1,\dots,x_n$, 
then the relation \emph{defined} by $\phi$ over $\bB$ is the relation $\{ (b_1,\dots,b_n) \in B^n \mid \bB \models \phi(b_1,\dots,b_n) \}$.
We say that $\tau$-formulas $\phi(x_1, \dots, x_n)$ and $\psi(x_1, \dots, x_n)$ are equivalent over $\bB$, if $\bB \models \forall x_1, \dots, x_n (\phi \Leftrightarrow \psi)$. We often omit the specification of the structure if it is clear from the context.
We say that a structure $\bB$ has \emph{quantifier elimination} if every first-order formula is equivalent to a quantifier-free formula over $\bB$. 
Every quantifier-free formula can be written in \emph{conjunctive normal form} (CNF), i.e., as a conjunction of disjunctions of
{\em literals}, i.e., atomic formulas or their negations. A disjunction of literals is also called a {\em clause}. 

If $\bA$ and $\bB$ are $\tau$-structures,
then a \emph{homomorphism} from $\bA$ to $\bB$ is a function $h \colon A \to B$ that \emph{preserves} all the relations, that is,
if $(a_1,\dots,a_k) \in R^{\bA}$, then 
$(h(a_1),\dots,h(a_k)) \in R^{\bB}$. 
The structures $\bA$ and $\bB$ are called \emph{homomorphically equivalent}
if there exists a homomorphism from $\bA$ to $\bB$ and a homomorphism from $\bB$ to $\bA$.
A first-order $\tau$-formula is {\em preserved} by a map between two $\tau$-structures $\bA$  and $\bB$ if it preserves the relation defined by the formula in these structures. 

A \emph{first-order expansion} of $\bA$ is a structure $\bA'$ augmented by relations
that are first-order definable in $\bA$. 
A \emph{first-order reduct} of $\bA$ is a reduct of a first-order expansion of $\bB$.
Relational structures might have an infinite signature; however, to avoid representational issues and for simplicity we restrict ourselves to finite signatures in the following definition.

\begin{definition}[CSPs]
Let $\tau$ be a finite relational signature 
and let $\bB$ be a $\tau$-structure. 
The \emph{constraint satisfaction problem} for $\bB$, denoted by $\Csp(\bB)$,  is the computational problem of deciding for a given finite $\tau$-structure $\bA$ whether $\bA$ has a homomorphism to $\bB$ or not. 
\end{definition}

Note that this definition of constraint satisfaction problems can be used even if $\bB$ is an infinite structure over a finite relational signature. Also note that homomorphically equivalent structures have the same CSP. 

\begin{ex}
The structure $(\{0,1,2\};\neq)$ is denoted by $K_3$. The problem $\Csp(K_3)$ is the three-colorability problem for graphs. The input is a structure with a single binary relation, representing edges in a graph (ignoring the orientation); homomorphisms from this graph to $K_3$ correspond precisely to the proper 3-colorings of the graph. 
\end{ex}

\subsection{Primitive Positive Constructions}

Three central concepts in the complexity analysis of CSPs are \emph{primitive positive definitions}, \emph{primitive positive interpretations}, and \emph{primitive positive constructions}. The three concepts are increasingly powerful. Their definitions build on each other and will be recalled here for the convenience of the reader. 

A \emph{primitive positive $\tau$-formula}
is a formula 
$\phi(x_1,\dots,x_n)$ with free variables $x_1,\dots,x_n$ of the form
$$ \exists y_1,\dots,y_l (\psi_1 \wedge \cdots \wedge  \psi_m)$$
where $\psi_1,\dots,\psi_k$ are atomic $\tau$-formulas over the variables $x_1,\dots,x_n,y_1,\dots,y_l$. 
Two relational structures $\bA$ and $\bB$ are called 
\begin{itemize}
    \item \emph{(primitively positive) interdefinable} if they have the same domain $A=B$, and if every relation of $\bA$ is
(primitively positively) definable in $\bB$ and vice versa. 
\item \emph{(primitively positively) bi-definable} if $\bB$ is isomorphic to 
a structure that is (primitively positively) interdefinable with $\bA$. 
\end{itemize}

We will now turn our attention towards methods for complexity analysis. 

\begin{lemma}[\cite{Jeavons}] \label{lem:pp-red}
Let $\bA$ and $\bB$ be structures with finite relational signatures and the same domain. 
If every relation of $\bA$ has a primitive positive definition in $\bB$, then there is a polynomial-time reduction from $\Csp(\bA)$ to $\Csp(\bB)$.
\end{lemma}

Primitive positive definability can be generalised as follows.

\begin{definition}[Interpretations]
A \emph{(primitive positive) interpretation} of a structure $\bC$ in a structure $\bB$ is a partial surjection $I$ from $B^d$ to $C$, for some finite $d \in {\mathbb N}$ called the \emph{dimension} of the interpretation, such that the preimage 
of a relation of arity $k$ defined by an atomic formula in $\bC$, considered as a relation of arity $dk$ over $B$, is (primitively positively) definable in $\bB$; in this case, we say that $\bC$ is (primitively positively) interpretable in $\bB$.

Two structures $\bB$ and $\bC$ such that $\bB$ has a primitive positive interpretation in $\bC$ and $\bC$ has a primitive positive interpretation in $\bB$ are called \emph{mutually primitively positively interpretable}.
\end{definition}

Note that in particular $x=x$ and $x=y$ are atomic formulas and hence the domain and the kernel of $I$ are primitively positively definable. Primitive positive interpretations preserve the complexity of CSPs in the following way.

\begin{proposition}[see, e.g., Theorem 3.1.4 in~\cite{Book}]
\label{prop:pp-int-reduce}
Let $\bB$ and $\bC$ be structures with finite relational signatures. 
If $\bC$ has a primitive positive interpretation in $\bB$, then there is a polynomial-time reduction from $\Csp(\bC)$ to $\Csp(\bB)$. 
\end{proposition}

\begin{ex}\label{ex:I}
Let ${\mathbb I}$ be the set of all pairs $(x,y) \in {\mathbb Q}^2$ such that $x<y$; i.e., ${\mathbb I}$ might be viewed as the set of all closed intervals $[a,b]$ of 
rational numbers. Let $\sf m$ be the binary relation over ${\mathbb I}$ that contains all pairs $((a_1,a_2),(b_1,b_2))$ such that $a_2 = b_1$. 
Then the structure $({\mathbb I};{\sf m})$
has a primitive positive interpretation (of dimension 2) in $({\mathbb Q};<)$:
\begin{itemize}
    \item The interpretation map is the identity map on $\mathbb{I} \subseteq \Q^2$.
    \item The preimage of the relation defined 
    by the atomic formula $a=b$ is defined by the formula $a_1 = b_1 \wedge a_2 = b_2 \wedge a_1 < a_2$.  
    \item The preimage of the relation defined by the atomic formula ${\sf m}(a,b)$ is 
    defined by the formula $a_1 < a_2 \wedge a_2 = b_1 \wedge b_1 < b_2$;
\end{itemize}
It is straightforward to adapt the construction above to atomic formulas that are obtained by variable identification by
using the equality relation.
Proposition~\ref{prop:pp-int-reduce} implies that CSP$({\mathbb I};{\sf m})$ is in P since
CSP$({\mathbb Q};<)$ is in P.
\end{ex}

\begin{ex}\label{ex:allen}
The {\em interval algebra}~\cite{Allen} is a formalism that is 
both well-known and well-studied in AI. 
It can be viewed as a relational structure
with the domain ${\mathbb I}$ introduced in Example~\ref{ex:I} and a binary relation
symbol for each binary relation $R \subseteq {\mathbb I}^2$ such that the relation $\{(a_1,a_2,b_1,b_2) \mid ((a_1,a_2),(b_1,b_2)) \in R\}$
is first-order definable in $({\mathbb Q};<)$. We let $\mathfrak{IA}$ denote this structure
and we let $\top$ denote the relation which holds for all pairs of intervals.
Clearly, Allen's Interval Algebra has a 2-dimensional interpretation in $({\mathbb Q};<)$, but not a primitive positive interpretation.

The \emph{basic} relations of Allen's Interval Algebra are
the 13 relations defined in Table~\ref{tb:allen-basic-defs}: we let $\mathfrak{IA}^b$ be the corresponding structure.
If $I=[a,b] \in \mathbb I$, then we write $I^-$ for $a$ and $I^+$ for $b$. 
It is well-known that all the basic relations of Allen's Interval Algebra have a primitive positive definition over $({\mathbb I}; {\sf m})$~\cite{Allen:Hayes:ijcai85}. We conclude that $\Csp(\mathfrak{IA}^b)$ is in P since $\Csp({\mathbb I}; {\sf m})$ is in P by
the previous example.
\end{ex}

\begin{table}[t]
\caption{Basic relations in the interval algebra.} \label{tb:allen-basic-defs}
\begin{center}
     \begin{tabular}{|ll|l|l|}\hline
       Basic relation \hspace*{3mm} & \hspace*{8mm} & Example & Endpoints \hspace*{2mm}\\ 
\hline\hline
      $X$ precedes       $Y$ & $\pa$  & \texttt{XXX\spv}    & $\xp < \ym$ 
\\ \cline{1-2}
      $Y$ preceded by         $X$ & $\pu$  & \texttt{\spv YYY}   & \\ 
\hline
      $X$ meets $Y$ & $\ma$   & \texttt{XXXX\spiv}  & $\xp=\ym$ 
\\ \cline{1-2}
      $Y$ is met by        $X$ & $\allen-mi$  & \texttt{\spiv YYYY} & \\ 
\hline
      $X$ overlaps      $Y$ & $\oa$   & \texttt{XXXX\spii} & $\xm<\ym \: \wedge$\\ 
& & \texttt{\spii YYYY} & $\ym <\xp \: \wedge$ \\ \cline{1-2}
      $Y$ overlapped by $X$ & $\oi$  & & $\xp <\yp \: \wedge$
 \\ \hline
      $X$ during        $Y$ & $\da$   & \texttt{\spiii XX \spii} & $\xm>\ym \: \wedge$ \\ \cline{1-2}
      $Y$ includes      $X$ & $\di$  & \texttt{YYYYYY}   & $\xp<\yp$ \\ 
\hline
      $X$ starts        $Y$ & $\sa$   & \texttt{XXX\spiv}  & $\xm=\ym \: \wedge$
\\ \cline{1-2}
      $Y$ started by $X$ & $\si$  & \texttt{YYYYYY}   & $\xp<\yp$ \\ 
\hline
      $X$ finishes      $Y$ & $\fa$   & \texttt{\spiii XXX} & $\xp=\yp \: \wedge$
\\ \cline{1-2}
      $Y$ finished by $X$ & $\fu$  & \texttt{YYYYYY}   & $\xm>\ym$ \\ 
\hline
      $X$ equals        $Y$ & $\equiv$ & \texttt{XXXX}      & $\xm=\ym \: \wedge$ \\ 
                                   &          & \texttt{YYYY}      & $\xp=\yp$ \\ 
\hline
    \end{tabular}
\end{center}
\end{table}

We finally define primitive positive constructions. We will not use such constructions in our proofs but it
is used in the statement of Theorem~\ref{thm:wonderland} and, thus, in the formulation of the infinite-domain tractability conjecture.

\begin{definition}[Primitive positive constructions]
A structure $\bC$ has a \emph{primitive positive construction} in $\bB$ if $\bC$ is homomorphically equivalent to a structure $\bC'$ with a primitive positive interpretation in $\bB$. \end{definition}

\begin{lemma}
\label{lem:pp-construct-reduce}
Let $\bB$ and $\bC$ be structures with finite relational signature. If $\bC$ has a primitive positive construction in $\bB$, then there is a polynomial-time reduction from $\Csp(\bC)$ to $\Csp(\bB)$. 
\end{lemma}
\begin{proof}
An immediate consequence of Proposition~\ref{prop:pp-int-reduce} since homomorphically equivalent structures have the same CSP. 
\end{proof}

\subsection{Model Theory and Algebra}
\label{sect:cat}

This section collects some basic terminology and facts from model theory and algebra.
The set of all first-order $\tau$-sentences that are true in a given $\tau$-structure $\bA$ is
called the {\em first-order theory} of $\bA$.
A countable structure $\bA$ is  \emph{$\omega$-categorical} if all countable models of the first-order theory of $\bA$ are isomorphic.
The structure $({\mathbb Q};<)$, and all structures with an interpretation in $({\mathbb Q};<)$, are $\omega$-categorical: for the structure $({\mathbb Q};<)$, this was shown by Cantor who proved that all countable dense and unbounded linear orders are isomorphic. 
All structures that we encounter in the later parts of
this article are $\omega$-categorical.

An \emph{automorphism} of a structure $\bA$ is a permutation $\alpha$ of $A$ such that both $\alpha$ and its inverse are homomorphisms. The set of all automorphisms of a structure $\bA$ is denoted by $\Aut(\bA)$, and forms a group with respect to composition;
the neutral element of the group is the identity map which is denoted by $\id_A$.
The set of all permutations of a set $A$ is called the \emph{symmetric group} and denoted by $\Sym(A)$. Clearly, $\Sym(A)$ equals $\Aut(A;=)$.

The \emph{orbit} of $(a_1,\dots,a_n) \in A^n$ in $\Aut(\bA)$ is the set $\{(\alpha(a_1),\dots,\alpha(a_n)) \mid \alpha \in \Aut(\bA) \}$. 
It was proved independently by Engeler, Ryll-Nardzewski, and Svenonius that a countable structure $\bA$ is $\omega$-categorical if and only if $\Aut(\bA)$ is \emph{oligomorphic}, i.e., has only finitely many orbits of $n$-tuples, for all $n \geq 1$ (see, e.g., \cite[Theorem 6.3.1]{Hodges}). 
This implies that structures with a first-order interpretation in an $\omega$-categorical structure are $\omega$-categorical~\cite[Theorem 6.3.6]{Hodges}. In particular, first-order reducts of $\omega$-categorical structures are again $\omega$-categorical. It also follows that first-order expansions of $\omega$-categorical structures $\bB$ are $\omega$-categorical themselves since such relations are preserved by all automorphisms of $\bB$. 
In an $\omega$-categorical structure $\bB$, a relation
is preserved by all automorphisms of $\bB$ if and only if it is first-order definable in $\bB$ (see, e.g.,~\cite[Theorem 4.2.9]{Book}).

Another important property of $({\mathbb Q};<)$ is called \emph{homogeneity}.
An \emph{embedding} from $\bA$ to $\bB$ is an injective homomorphism $e$ from $\bA$ to $\bB$ which also preserves the complement of each relation, i.e., $(a_1,\dots,a_k) \in R^{\bA}$ if and only if 
$(f(a_1),\dots,f(a_k)) \in R^{\bB}$. 
A $\tau$-structure $\bA$ is called a \emph{substructure} of $\bB$ if $A \subseteq B$ and the identity map $\id_A$ is an embedding from $\bA$ to $\bB$.
Now, a relational structure is called \emph{homogeneous} (or sometimes \emph{ultrahomogeneous}) if every isomorphism 
between finite substructures
can be extended to an automorphism of the structure~\cite[p. 160]{Hodges}.  
An $\omega$-categorical structure $\bB$ is homogeneous if and only if $\bB$ has quantifier elimination~\cite[Theorem 6.4.1]{Hodges}. 

As mentioned above, a relation $R$ is first-order
definable in
an $\omega$-categorical structure $\bB$
if and only if it
is preserved by all automorphisms of $\bB$.
Similarly, the question whether a given relation is primitively positively definable in $\bB$ can be studied using \emph{polymorphisms}. 
A polymorphism of a structure $\bB$ is a homomorphism from $\bB^k$ to $\bB$. Here, the structure $\bB^k$ denotes the $k$-fold direct product structure $\bB \times \cdots \times \bB$; 
more generally, if $\bB_1,\dots,\bB_k$ are $\tau$-structures, then $\bC = \bB_1 \times \cdots \times 
\bB_k$ is defined to be the $\tau$-structure 
with domain $B_1 \times \cdots \times B_k$
and for every $R \in \tau$ of arity $m$ we have
\begin{align*}
R^{\bC} = \{ & ((a_{1,1},\dots,a_{1,k}),\dots,(a_{m,1},\dots,a_{m,k})) \in C^m \mid 
\\ & (a_{1,1},\dots,a_{m,1}) \in R^{\bB_1},\dots,(a_{1,k},\dots,a_{m,k}) \in R^{\bB_k} \} .
\end{align*}
The set of all polymorphisms of a structure $\bB$ is denoted by $\Pol(\bB)$.
\emph{Endomorphisms} are a special case  of polymorphisms with $k=1$: an endomorphism
of a structure $\bB$ is thus a homomorphism from
$\bB$ to itself. The set of all endomorphisms
of $\bB$ is denoted by $\End(\bB)$.
For every $i \leq n$, the 
$i$-th projection of arity $n$ 
is the operation $\pi^n_i$ 
defined by $\pi^n_i(x_1,\dots,x_n) \coloneqq x_i$. 
The set of all polymorphisms of a structure $\bB$ forms an \emph{(operation) clone}: it is closed under composition and contains all  projections. 
Moreover, an operation clone $\mathscr C$ on a set $B$ is a polymorphism clone of a relational structure if and only if the operation clone is \emph{closed}, i.e., for each $k \geq 1$ the set of $k$-ary operations in $\mathscr C$ is closed with respect to the product topology on $B^{B^k}$ where $B$ is taken to be discrete (see, e.g., Corollary 6.1.6 in~\cite{Book}). 

An operation clone ${\mathscr C}$ is called \emph{oligomorphic} if the permutation group
${\mathscr G}$ of invertible unary maps in ${\mathscr C}$ is oligomorphic. 
A relation $R \subseteq B^n$ is preserved by all polymorphisms of an $\omega$-categorical structure $\bB$ if and only if $R$ has a primitive positive definition in $\bB$~\cite{BodirskyNesetrilJLC}.  
It follows that two $\omega$-categorical relational structures with the same domain have the same polymorphisms if and only if they are primitively positively interdefinable. If ${\mathscr F} \subseteq {\mathscr C}$, we write 
\begin{itemize}
\item $\langle {\mathscr F} \rangle$ for the smallest subclone of 
${\mathscr C}$ which contains ${\mathscr F}$, and
\item $\overline{{\mathscr F}}$ for the smallest \emph{(locally, or topologically) closed} 
subset of ${\mathscr C}$ that contains ${\mathscr F}$: that is, $\overline{{\mathscr F}}$ consists of all operations $f$ such that for every finite subset $S$ of the domain there exists an operation $g \in {\mathscr F}$ which agrees with $f$ on $S$. 
\end{itemize}

If $\sigma \colon \{1,\dots,n\} \to \{1,\dots,k\}$ and $f \in {\mathscr C}$ has arity $n$, then we write
$f^{\sigma}$ for the operation
$f(\pi^{k}_{\sigma(1)},\dots,\pi^{k}_{\sigma(n)})$ which maps $(x_1,\ldots,x_{k})$ to $f(x_{\sigma(1)},\dots,x_{\sigma(n)})$.
Let ${\mathscr C}_1$ and ${\mathscr C}_2$ be two clones and let $\xi$ be a function from ${\mathscr C}_1$ to ${\mathscr C}_2$ that preserves the arities of the operations. 
Then
$\xi$ is 
\begin{itemize}
    \item a \emph{clone homomorphism} if for all $f \in {\mathscr C}_1$ of arity $n$ and $g_1,\dots,g_n \in {\mathscr C}_1$ of arity $k$ we have 
$$ \xi(f(g_1,\dots,g_n)) = \xi(f)(\xi(g_1),\dots,\xi(g_n))$$
and $\xi(\pi^n_i) = \pi^n_i$ for all $i \leq n$. 
    \item \emph{minor-preserving} if for all $f \in {\mathscr C}_1$ of arity $n$ and $\sigma \colon \{1,\dots,n\} \to \{1,\dots,k\}$ 
$$ \xi(f^{\sigma}) = \xi(f)^{\sigma}.$$
\item  \emph{uniformly continuous} if for every finite subset $F$ of the domain of ${\mathscr C}_2$  there exists a finite subset $G$ of the domain of ${\mathscr C}_1$ such that 
and for every $n \geq 1$ and all $f,g \in {\mathscr C}_1$ of arity $n$, if $f|_{G^n} = g|_{G^n}$, then $\xi(f)|_{F^n} = \xi(g)|_{F^n}$.
\end{itemize}

Note that trivially, every clone homomorphism is minor-preserving.
Uniformly continuous clone homomorphims naturally arise from interpretations.

\begin{lemma}[\cite{Topo-Birk}]\label{lem:clone-homo}
If $\bC$ has a primitive positive interpretation in $\bB$, then $\Pol(\bB)$ has a uniformly continuous clone homomorphism to $\Pol(\bC)$.  
\end{lemma}

The relevance of minor-preserving maps for the complexity of constraint satisfaction is witnessed by the following theorem. 

\begin{theorem}[\cite{wonderland}]
\label{thm:wonderland}
If $\bB$ is an $\omega$-categorical structure and $\bC$ is a finite structure,
then $\bC$ has a primitive positive construction in $\bB$ if and only if $\Pol(\bB)$ has a uniformly continuous minor-preserving map to $\Pol(\bC)$. 
\end{theorem}

Note that Theorem~\ref{thm:wonderland} in combination with Lemma~\ref{lem:pp-construct-reduce}
implies the following.

\begin{corollary}\label{cor:hard}
Let $\bB$ be an $\omega$-categorical structure and suppose that $\Pol(\bB)$ has a uniformly continuous minor-preserving map to $\Pol(K_3)$.  Then
$\bB$ has a finite-signature reduct whose $\Csp$ is NP-hard.
\end{corollary}

The \emph{infinite-domain tractability conjecture}
implies that for reducts of finitely bounded homogeneous structures, and if P $\neq$ NP, then the condition given in Corollary~\ref{cor:hard} is not only sufficient, but also necessary for NP-hardness~\cite{BPP-projective-homomorphisms}.  
Note that $\Csp$s of reducts of finitely bounded homogeneous structures are always in NP (see, e.g., Proposition 2.3.15 in \cite{Book}).
The conjecture has also  interesting algebraic interpretations in line with the dichotomy for $\Csp$s of finite structures~\cite{BartoPinskerDichotomy,BKOPP-equations,wonderland}.
An operation $f \colon B^k \to B$ is called a \emph{weak near unanimity operation} if for all $x,y \in B$ the operation $f$ satisfies 
$$f(y,x,\dots,x) = f(x,y,\dots,x) = \cdots = f(x,\dots,x,y).$$
If $\mathscr C$ is a clone on a finite domain $B$ without a minor-preserving map to $\Pol(K_3)$, then $\mathscr C$ contains a weak near unanimity operation~\cite{MarotiMcKenzie}. A potential generalisation of this fact to polymorphism clones of $\omega$-categorical structures $\bB$ involves the concept of a  \emph{pseudo weak near unanimity (pwnu) polymorphism}, i.e., a polymorphism $f$ of arity at least two such that there are endomorphisms $e_1,\dots,e_k$ of $\bB$ such that for all $x,y \in B$ 
\begin{align}
    e_1(f(y,x,\dots,x)) = e_2(f(x,y,\dots,x)) = \cdots = e_k(f(x,\dots,x,y)). \label{eq:pwnu}
\end{align}

We note that all of the polynomial-time tractability conditions that we prove in this article can be phrased in terms of the existence of pwnu polymorphisms.
 It is not known whether every polymorphism clone of
 a reduct of a finitely bounded homogeneous structure that does not have a uniformly continuous minor-preserving map to $\Pol(K_3)$ contains a pwnu (see Question 21 in~\cite{Book}); however, it is known to be false for polymorphism clones of $\omega$-categorical structures in general~\cite{BBKMP}.
 Note that clone homomorphisms preserve identities such as \eqref{eq:pwnu}, and
 it follows from Lemma~\ref{lem:clone-homo} that primitive positive interpretations 
 preserve the existence of pwnu polymorphisms. The same is not true for minor-preserving maps instead of clone homomorphisms.
 However, we have the following; it uses the assumption that $\overline{\Aut(\bB)} = \End(\bB)$
 which is equivalent to $\bB$ being a \emph{model-complete core} (see, e.g.,~\cite[Section 4.5]{Book}). 
 
 \begin{lemma}\label{lem:excl}
 Let $\bC$ be a homogeneous structure with finite relational signature  and let 
$\bB$ be a first-order reduct of $\bC$ with a pwnu polymorphism. If
 $\overline{\Aut(\bB)} = \End(\bB)$, 
 then $\Pol(\bB)$ does not have a uniformly continuous minor-preserving map to $\Pol(K_3)$.
 \end{lemma}
 \begin{proof}
  The assumptions imply that we may apply Corollary 3.6 in~\cite{BKOPP-equations}. Hence,
  $\bB$ has a uniformly continuous minor-preserving map to $\Pol(K_3)$ if and only if there exist $n \in {\mathbb N}$ and $c_1,\dots,c_n \in B$ such that the  clone $\Pol(\bB,\{c_1\},\dots,\{c_n\})$ 
  has a continuous clone homomorphism to $\Pol(K_3)$. 
  But if  $\bB$ has a pwnu polymorphism, then so does every expansion of $\bB$ by finitely many unary singleton relations (Proposition 10.1.13 in~\cite{Book}).
  Since clone homomorphisms preserve the existence of pwnu polymorphisms and $K_3$ does not have such a polymorphism (see, e.g.,~\cite[Proposition 6.1.43]{Book}), the statement follows.
 \end{proof}

\section{Algebraic Products} 
\label{sect:algprod}
We devote this section to presenting the algebraic product and studying some of its properties:
Section~\ref{sec:algprodbasics} contains the definition together with some elementary facts while
Section~\ref{sec:determinedclauses} describes connections
with $i$-determined clauses (that
we introduce in Section~\ref{sec:algprodbasics}).
The algebraic product has been studied in the past. For instance, Greiner~\cite{Greiner:PhD} uses it for studying CSPs
of combinations of structures or background theories (a topic we will touch upon in Section~\ref{sect:conc}),
Baader and Rydval~\cite{RydvalBaader} use it for analysing the complexity of description logics, 
and Bodirsky~\cite{Book}
uses it in connection with Ramsey structures.

\subsection{Basic Properties}
\label{sec:algprodbasics}

The algebraic product is defined
as follows.

\begin{definition}
Let ${\mathfrak A}_1$ and ${\mathfrak A}_2$ be structures with signature $\tau_1$ and $\tau_2$, respectively. 
Then the \emph{algebraic product} 
${\mathfrak A}_1 \boxtimes {\mathfrak A}_2$
is the structure with domain $A_1 \times A_2$ which has for every atomic 
$\tau_1$-formula $\phi(x_1,\dots,x_k)$ the relation 
$$\{((u_1,v_1),\dots,(u_k,v_k)) \mid {\mathfrak A}_1 \models \phi(u_1,\dots,u_k) \}$$
and analogously for every atomic $\tau_2$-formula $\phi$ the relation 
$$\{((u_1,v_1),\dots,(u_k,v_k)) \mid {\mathfrak A}_2 \models \phi(v_1,\dots,v_k) \}.$$
\end{definition}

The relation symbol for the atomic $\tau_i$-formula $x=y$ will be denoted by $=_i$. Clauses over the signature of $\bA_1 \boxtimes \bA_2$ where all atomic formulas are built from symbols that have been introduced for atomic $\tau_i$-formulas are called \emph{$i$-determined}.

\begin{ex} \label{ex:algprod}
The structure $({\mathbb Q};<) \boxtimes ({\mathbb Q};<)$
has the domain $D={\mathbb Q}^2$ and
contains four relations 
\begin{itemize}
\item
$=_1$ equal to $\{((u_1,v_1),(u_2,v_2)) \in D^2 \; | \; u_1=u_2\}$,
\item
$=_2$ equal to $\{((u_1,v_1),(u_2,v_2)) \in D^2 \; | \; v_1=v_2\}$,
\item
$<_1$ equal to $\{((u_1,v_1),(u_2,v_2)) \in D^2\; | \; u_1 < u_2\}$, and
\item
$<_2$ equal to $\{((u_1,v_1),(u_2,v_2)) \in D^2 \; | \; v_1 < v_2\}$.
\end{itemize}
The clauses  $(x =_1 y \vee y <_1 z)$
and $(x <_2 y \vee x =_2 z \vee x <_2 u)$ are
$1$- and $2$-determined, respectively. The clause
$(x <_1 y \vee x =_2 y)$ is not $i$-determined for any
$i \in \{1,2\}$.
\end{ex}

\begin{rem}\label{rem:hom}
We note that the algebraic product preserves some of the important fundamental properties of structures. For example,
if $\bA_1$ and $\bA_2$ are homogeneous, then 
so is ${\mathfrak A}_1 \boxtimes {\mathfrak A}_2$ (Proposition 4.2.19 in~\cite{Book}),
and if  $\bA_1$ and $\bA_2$ are $\omega$-categorical, then so is ${\mathfrak A}_1 \boxtimes {\mathfrak A}_2$. 
\end{rem}

We define the $n$-{\em fold algebraic product} ${\mathfrak A}_1 \boxtimes \dots \boxtimes {\mathfrak A}_n$ in the natural way
together with the 
 $n$-{\em fold algebraic power}
$$\mathfrak{A}^{(n)} \coloneqq 
\underbrace{\mathfrak{A} \boxtimes \cdots \boxtimes \mathfrak{A}}_{\mbox{$n$ times}}.$$
The forthcoming definitions and statements in this section are presented for the binary product, but they can easily be generalised to $n$-fold algebraic products. 
We continue by studying the polymorphism clone of ${\mathfrak A}_1 \boxtimes {\mathfrak A}_2$.
For $i \in \{1,2\}$, let ${\mathscr C}_i$ be a clone of operations on a set $A_i$. If $f_1 \in {\mathscr C}_1$
and $f_2 \in {\mathscr C}_2$ both have arity $k$, then we write $(f_1,f_2)$ for the operation on $A_1 \times A_2$
given by $$((a_1,b_1),\dots,(a_k,b_k)) \mapsto (f_1(a_1,\dots,a_k),f_2(b_1,\dots,b_k)).$$

The \emph{direct product} ${\mathscr C}_1 \times {\mathscr C}_2$ of ${\mathscr C}_1$ and ${\mathscr C}_2$ is the clone ${\mathscr D}$ on the set $A_1 \times A_2$ whose operations of arity $k$ consist of the set of all operations $(f_1,f_2)$ where
$f_i \in {\mathscr C}_i$ is of arity $k$.
Note that this generalises the usual definition of direct products of permutation groups. 
If ${\mathscr C}_1 = {\mathscr C}_2 = {\mathscr C}$ then ${\mathscr D}$ is called a \emph{direct power}
and we write ${\mathscr C}^2$ instead of ${\mathscr C} \times \mathscr C$. 
Note that the function $\theta_i \colon {\mathscr C}_1 \times {\mathscr C}_2 \to {\mathscr C}_i$ given by 
$(f_1,f_2) \mapsto f_i$ is a uniformly continuous minor-preserving map. 
Also note that if ${\mathscr C}_1$ and ${\mathscr C}_2$ are oligomorphic, then so is ${\mathscr C}_1 \times {\mathscr C}_2$.

The following proposition is one of the important features of the algebraic product. 
We present it for two-fold algebraic products but it can 
obviously be generalised to the $n$-fold case.

\begin{proposition}
\label{prop:product-hom}
For all  
structures $\bA_1$ and $\bA_2$ we have 
$$\Pol({\mathfrak A}_1 \boxtimes {\mathfrak A}_2) = \Pol({\mathfrak A}_1) \times \Pol({\mathfrak A}_2).$$ 
Likewise, we have $\End(\bA_1 \boxtimes \bA_2) = \End(\bA_1) \times \End(\bA_1)$ and $\Aut(\bA_1 \boxtimes \bA_2) = \Aut(\bA_1) \times \Aut(\bA_2)$. 
\end{proposition}
\begin{proof}
If $f_i \in \Pol(\bA_i)$, for $i \in \{1,2\}$, then clearly $(f_1,f_2)$ preserves all relations of ${\mathfrak A}_1 \boxtimes {\mathfrak A}_2$. Conversely, let $f \in \Pol({\mathfrak A}_1 \boxtimes {\mathfrak A}_2)$.
Pick $a \in A_2$ and 
define $f_1(x_1,\dots,x_n) \coloneqq f((x_1,a),\dots,(x_n,a))_1$. 
Since 
$f$ preserves $=_1$, 
this definition does not depend on the choice of $a \in A_2$.  Note that $f_1 \in \Pol(\bA_1)$. The function $f_2 \in \Pol(\bA_2)$ is defined analogously, with the roles of $1$ and $2$ swapped, using a fixed element $b \in A_1$. 
Finally,  we observe that $f = (f_1,f_2)$. Indeed, if $(x_1,y_1), \dots, (x_n,y_n) \in A_1 \times A_2$, then
\begin{align*}
(f_1, f_2)((x_1, y_1),\dots (x_n, y_n)) &= (f((x_1, a), \dots, (x_n,a))_1, f((b, y_1), \dots, (b, y_n))_2) \\
&= (f((x_1, y_1), \dots, (x_n,y_n))_1, f((x_1, y_1), \dots, (x_1, y_n))_2) \\
& = f((x_1, y_1),\dots (x_n, y_n)),
\end{align*} 
where the middle equality uses that $f$ preserves $=_1$ and $=_2$.
The statement for endomorphisms and automorphisms of algebraic product structures can be proved analogously. 
\end{proof}

Under fairly general assumptions on ${\mathscr C}$ it holds that $\theta_i({\mathscr C})$, for $i \in \{1,2\}$, is closed.

\begin{proposition}\label{prop:closed}
Let $\bA_1$ and $\bA_2$ be countable $\omega$-categorical structures and let ${\mathscr C} \subseteq \Pol(\bA_1 \boxtimes \bA_2)$ be a closed set of 
operations on $A_1 \times A_2$. If ${\mathscr C}$ contains $\alpha f$ for every $f \in {\mathscr C}$ and every $\alpha \in \Aut(\bA_1 \boxtimes \bA_2)$, then $\theta_1({\mathscr C})$ and $\theta_2({\mathscr C})$ are closed.
\end{proposition}
\begin{proof} 
It suffices to show the statement for $i = 1$.
Let $f \in \overline{\theta_1({\mathscr C})}$ be of arity $k$. 
Fix an enumeration $p_0,p_1,\dots$
of $A_1$ and an enumeration $q_0,q_1,\dots$ of $A_2$. 
Then for every $n \in {\mathbb N}$ there exists $g_n \in \Pol({\mathscr C})$ such that $\theta_1(g_n)|_{\{p_0,\dots,p_n\}^k} = f|_{\{p_0,\dots,p_n\}^k}$. 
In the proof we use {\em K\H{o}nig's tree lemma}: if $T$ is a rooted tree with an infinite number of nodes and
each node has a finite number of children, then $T$ contains a branch of infinite length. To define the tree $T$, let $S_n := \{p_0,\dots,p_n\} \times \{q_0,\dots,q_n\}$ and consider the functions ${\mathcal F} \coloneqq \bigcup_{n \in {\mathbb N}} {\mathcal F}_{n}$ where 
$${\mathcal F}_{n} \coloneqq \{ g_m|_{S^k_n} \mid m \geq n \}.$$
For $h_1,h_2 \in {\mathcal F}$ define 
$h_1 \sim h_2$ if there exists $\alpha \in \Aut(\bA_1 \boxtimes \bA_2)$ such that $h_1 = \alpha h_2$. 
The vertices of the tree $T$ that we consider here are the equivalence classes of the equivalence relation $\sim$. The edges of $T$ are defined as follows: 
if $h_1 \in {\mathcal F}_{\ell}$ 
is the restriction of $h_2 \in {\mathcal F}_{\ell+1}$, then 
the equivalence class of $h_1$ and the equivalence class of $h_2$ are linked by an edge. 
Clearly the tree thus defined is infinite, and by the oligomorphicity of $\Aut(\bA_1 \boxtimes \bA_2)$ there are finitely many $\sim$-classes on ${\mathcal F}_{\ell}$, which implies that each vertex in $T$ has finitely many neighbours. 
K\H{o}nig's tree lemma implies the existence of an infinite path in the tree.

We construct an operation $g$ such that $g|_{S_n^k}$ is in a $\sim$-class that lies on this infinite path, for every $n \in \mathbb{N}$. This is trivial to achieve for $n = 0$. Now suppose inductively that $g$ is already defined on $S^k_n$, for $n \geq 0$. By the construction of the infinite path, we find representatives $h_n$ and $h_{n+1}$ of the $n$-th and the $(n+1)$-st element on the path such that $h_n$ is a restriction of $h_{n+1}$. The inductive assumption gives us an automorphism $\alpha$ of $\bA$ such that $\alpha h_n(x) = g(x)$ for all $x \in S_n^k$. We set $g(x)$ to be $\alpha h_{n+1}(x)$ for all $x \in S_{n+1}^k$. Since $\mathscr C$ is closed, the operation $g$ thus defined is in $\mathscr C$. By the choice of $g$, for $n \in {\mathbb N}$, we have that $\theta_1(g) = f$.
Therefore, $f \in \theta_1(\mathscr C)$ as we wanted to prove.
\end{proof}

Another important feature of the algebraic product is that it preserves certain computational properties.

\begin{proposition}\label{prop:prod-alg}
For $i \in \{1,2\}$, let $\bA_i$ be a countable $\omega$-categorical structure with finite relational signature $\tau_i$. If both
$\Csp(\bA_1)$ and $\Csp(\bA_2)$ are in P,
 then $\Csp(\bA_1 \boxtimes \bA_2)$ is in P. 
\end{proposition}

\begin{proof}
Without loss of generality, we may assume that the signatures $\tau_1$ and $\tau_2$ are disjoint and that $\bA_1 \boxtimes \bA_2$ is a $(\tau_1 \cup \tau_2 \cup \{=_1, =_2\})$-structure. Let $\bC$ be an instance of $\Csp(\bA_1 \boxtimes \bA_2)$. For $i\in \{1,2\}$, let $\bC_i$ be the
$(\tau_i \cup \{=_i\})$-reduct of $\bC$. 
For each $i \in \{1,2\}$, run an algorithm for $\Csp(\bA_i)$ on the input $\bC_i$. Accept the instance $\bC$ if and only if both algorithms accept.

The correctness of the algorithm follows from the fact that the map $h_i \colon C \to A_i$ is a homomorphism from $\bC_i$ to $\bA_i$ for both $i \in \{1,2\}$ if and only if the map $h \colon C \to A_1 \times A_2$ such that $h(c)=(h_1(c),h_2(c))$ for every $c \in C$ is a homomorphism from $\bC$ to $\bA_1 \times \bA_2$.
\end{proof}

\begin{ex}
CSP$({\mathbb Q}; <)$ is solvable in polynomial time by using any cycle detection algorithm for directed graphs. Hence, Proposition~\ref{prop:prod-alg} implies that
CSP$(({\mathbb Q}; <)^{(n)})$ is polynomial-time solvable for arbitrary $n \geq 1$, too.
\end{ex}

\begin{corollary}\label{cor:prod-alg}
Let $\bA_1, \bA_2, \bB$ be countable $\omega$-categorical structures with finite relational signature such that
$\Pol(\bB)$ contains $\Pol(\bA_1) \times \Pol(\bA_2)$. 
If both
$\Csp(\bA_1)$ and $\Csp(\bA_2)$ are in P,
 then $\Csp(\bB)$ is in P, too. 
\end{corollary}
\begin{proof}
Since $\Pol(\bB)$ contains $\Pol(\bA_1) \times \Pol(\bA_2) = \Pol(\bA_1 \boxtimes \bA_2)$, all relations of $\bB$ are primitively positively definable in $\bA_1 \boxtimes \bA_2$ \cite{BodirskyNesetril}. Hence, there is a polynomial-time reduction from 
$\Csp(\bB)$ to $\Csp(\bA_1 \boxtimes \bA_2)$ by Lemma~\ref{lem:pp-red} so $\Csp(\bB)$ is in P
by Proposition~\ref{prop:prod-alg}.
\end{proof}

The following lemma enables us to use Lemma~\ref{lem:excl} for first-order expansions of algebraic products.

\begin{lemma}\label{lem:prod-mc-cores}
For $i \in \{1,\dots, n\}$, let $\bA_i$ be a
structure such that $\overline{\Aut(\bA_i)}=\End(\bA_i)$. If $\bB$ is a first-order expansion of $\bA_1 \boxtimes \dots \boxtimes \bA_n$, then $\overline{\Aut(\bB)}=\End(\bB)$.
\end{lemma}

\begin{proof}
Clearly, $\overline{\Aut(\bB)} \subseteq \End(\bB)$. The converse inclusion also holds since 
\begin{align*}
\End(\bB) \subseteq \End(\bA_1 \boxtimes \dots \boxtimes \bA_n) 
& = \End(\bA_1) \times \dots \times \End(\bA_n) 
&& \text{(Proposition~\ref{prop:product-hom})} \\
& = \overline{\Aut(\bA_1)} \times \dots \times \overline{\Aut(\bA_n)} \\
& =
\overline{\Aut(\bA_1) \times \dots \times \Aut(\bA_n)} \\
& =  
\overline{\Aut(\bA_1 \boxtimes \dots \boxtimes \bA_n)} &&  \text{(Proposition~\ref{prop:product-hom})} \\
&
=  \overline{\Aut(\bB)}. && 
\end{align*}
\end{proof}

Lemma~\ref{lem:prod-mc-cores} holds, for instance, for the structure
$(\Q;<) \boxtimes (\Q;<)$ since
$\overline{\Aut(\Q;<)}$ is equal to $\End(\Q;<)$: every endomorphism of $({\mathbb Q};<)$ is injective and preserves the complement of $<$, and by the  homogeneity of $(\Q;<)$ every restriction of an endomorphism to a finite subset of $\Q$ can be extended to an automorphism of $(\Q;<)$.

We finish this section by a corollary of Lemmas~\ref{lem:clone-homo}, \ref{lem:excl}, and \ref{lem:prod-mc-cores}, which will be useful in Section~\ref{sect:applications} to prove disjointness of cases in dichotomy results.

\begin{corollary}\label{cor:excl}
Let $\bC$ be a first-order expansion of $(\Q; <)^{(n)}$ and let $\bB$ be a structure with a pwnu polymorphism that is mutually primitively positively interpretable with $\bC$. Then $\Pol(\bB)$ does not have a uniformly continuous minor-preserving map to $\Pol(K_3)$.
\end{corollary}

\begin{proof}
By Lemma~\ref{lem:clone-homo}, there is a uniformly continuous clone homomorphism from $\Pol(\bB)$ to $\Pol(\bC)$. Therefore, $\Pol(\bC)$ has a pwnu polymorphism.
By Lemma~\ref{lem:prod-mc-cores}, $\overline{\Aut(\bC)}=\End(\bC)$. Since $(\Q;<)^{(n)}$ is homogeneous, $\Pol(\bC)$ does not have a uniformly continuous minor-preserving map to $\Pol(K_3)$ by Lemma~\ref{lem:excl}. Finally, by Lemma~\ref{lem:clone-homo}, there is a uniformly continuous clone homomorphism from $\Pol(\bC)$ to $\Pol(\bB)$ and therefore $\Pol(\bB)$ does not have a uniformly continuous minor-preserving map to $\Pol(K_3)$ either.
 \end{proof}

\subsection{$i$-Determined Clauses} \label{sec:determinedclauses}

In the complexity analysis of first-order expansions of algebraic products $\bA_1 \boxtimes \bA_2$, we would like to use as much  information about first-order expansions of $\bA_1$ and of $\bA_2$ as possible; in this context, $i$-determined clauses are of particular relevance. In this section we collect several general observations about definability by (conjunctions of) $i$-determined clauses. 
Throughout this section, let
$\bA_i$ be a $\tau_i$-structure for $i \in \{1,2\}$.

We begin by making a definition.
If $\phi$ is a conjunction of  $i$-determined clauses over $\bA_1 \boxtimes \bA_2$, then  we let $\hat \phi$  denote the $\tau_i$-formula obtained from
replacing each atomic formula
$R(x_1,\dots,x_k)$ by $\psi(x_1,\dots,x_k)$ where 
$\psi$ is the atomic $\tau_i$-formula for which $R$ has been introduced in $\bA_1 \boxtimes \bA_2$. 
Let us reconsider Example~\ref{ex:algprod} and a formula
$\phi=(x =_1 y \vee x =_1 z) \wedge (x <_1 y \vee x <_1 z \vee x <_1 u)$. Every clause
in $\phi$ is $1$-determined and $\hat \phi$ equals
$(x = y \vee x = z) \wedge (x < y \vee x < z \vee x < u)$

\begin{lemma}\label{lem:pres}
Let
$f \in \Pol(\bA_1 \boxtimes \bA_2)$.
A conjunction of $i$-determined clauses $\phi$ is preserved by 
$f$ if and only if
$\hat \phi$ is preserved by $\theta_i(f)$.
\end{lemma}
\begin{proof}
Let $x_1,\dots,x_m$ be the free variables of $\phi$. 
Let $((a^{1,1}_1,a^{1,1}_2),\dots,(a^{k,m}_1,a^{k,m}_2)) \in (A_1 \times A_2)^{k \times m}$ and let $f$
be of arity $k$. 
For $j \in \{1,\dots,k\}$, the tuple  $((a^{j,1}_1,a^{j,1}_2),\dots,(a^{j,m}_1,a^{j,m}_2))$ satisfies $\phi$
in $\bA_1 \boxtimes \bA_2$
if and only if 
$(a^{j,1}_i,\dots,a^{j,m}_i)$ satisfies $\hat \phi$ in 
$\bA_i$. Likewise, 
$$\big(f((a^{1,1}_1,a^{1,1}_2),\dots,(a^{k,1}_1,a^{k,1}_2)),\dots, f((a^{1,m}_1,a^{1,m}_2),\dots,(a^{k,m}_1,a^{k,m}_2))\big)$$ satisfies $\phi$ in $\bA_1 \boxtimes \bA_2$ if and only if $\big (\theta_i(f)(a^{1,1}_i,\dots,a^{k,1}_i),\dots, \theta_i(f)(a^{1,m}_i,\dots,a^{k,m}_i) \big )$ satisfies $\hat \phi$ in $\bA_i$. This implies the statement.
\end{proof}

We continue by analysing the polymorphisms of first-order expansion of $\bA_1 \boxtimes \bA_2$ and definability via $i$-determined clauses.
The proof is based on {\em reduced} formulas (see, e.g., \cite[Definition 42]{BodChenPinsker}):
a quantifier-free formula $\phi$ (over a structure $\bB$) in CNF is called reduced if all formulas obtained from $\phi$ by removing one of the literals from one of the clauses in the formula are not equivalent to $\phi$ (over $\bB$). Clearly, every quantifier-free formula is equivalent to a reduced formula (in particular, if a formula is unsatisfiable, the reduced formula contains an empty clause). Satisfiable reduced formulas $\phi$ have the property that for every literal in $\phi$ there exists a satisfying assignment for $\phi$ that satisfies the literal, but satisfies no other literal of the same clause. 
For example, the formula $(x<y \vee x=y) \wedge (y<x \vee z<x)$ over $(\Q; <)$ is not reduced, but it is equivalent to the reduced formula $(x<y \vee x=y) \wedge (z<x)$.

\begin{lemma}\label{lem:factors}
Suppose that $\bA_1$ and $\bA_2$ are structures with quantifier elimination and 
let $\bB$ be a first-order expansion of $\bA_1 \boxtimes \bA_2$. 
Then, the following are equivalent.
\begin{enumerate}
    \item Every relation of $\bB$ has a definition by a conjunction of clauses each of which is either $1$-determined or $2$-determined. 
    \item 
$\Pol(\bB) = \theta_1(\Pol(\bB)) \times \theta_2(\Pol(\bB))$. 
\item $\Pol(\bB)$ contains $(\pi^2_1,\pi^2_2)$.
\end{enumerate}
The implication from 1 to 2 and from 2~to 3~also hold without the assumption that $\bA_1$ and $\bA_2$ have quantifier elimination. 
\end{lemma} 
\begin{proof}
$1 \Rightarrow 2$.
Clearly, $\Pol(\bB) \subseteq \theta_1(\Pol(\bB)) \times \theta_2(\Pol(\bB))$. To prove the converse inclusion, 
let $g_1,g_2 \in \Pol(\bB)$ be of arity $k$ and let $\phi$ be a formula that defines an $m$-ary relation of $\bB$. 
We claim that $(\theta_1(g_1),\theta_2(g_2))$ preserves $\phi$.
Let $t^1 = (t^1_1,\dots,t^1_m),\dots, t^k = (t^k_1,\dots,t^k_m) \in B^m$ be tuples that satisfy $\phi$ in $\bB$.
Let $\psi$ be a conjunct of $\phi$; then $\psi$ is $i$-determined, for some $i \in \{1,2\}$. 
Since $g_i \in \Pol(\bB)$ we have that $g_i(t^1,\dots,t^k)$ 
satisfies $\psi$. 
Since for every $j \in \{1,\dots,m\}$ we have  $$g_i(t^1_j,\dots,t^k_j)_i = (\theta_1(g_1),\theta_2(g_2))(t^1_j,\dots,t^k_j)_i$$
and $\psi$ is $i$-determined, 
we have that $(\theta_1(g_1),\theta_2(g_2))(t^1,\dots,t^k)$ satisfies $\psi$ as well. 
This implies that $(\theta_1(g_1),\theta_2(g_2))$ preserves $\phi$ and shows that every operation in 
$\theta_1(\Pol(\bB)) \times \theta_2(\Pol(\bB))$
is a polymorphism of $\bB$. 

\smallskip

$2 \Rightarrow 3$. Trivial. 

\smallskip

$3 \Rightarrow 1$.
We show the contrapositive. Arbitrarily choose a relation $R$ in $\bB$. By assumption, $R$ has a quantifier-free first-order definition $\phi$ in $\bA_1 \boxtimes \bA_2$ and we may additionally assume that $\phi$ is written in reduced CNF. Suppose for contradiction that $\phi$ contains a clause which is neither $1$- nor $2$-determined, i.e., a clause $\psi$ that contains a $\tau_1$-literal $\psi_1$ and a $\tau_2$-literal $\psi_2$. By the assumption that $\phi$ is reduced, $\phi$ has for every $i \in \{1,2\}$ a satisfying assignment $\alpha_i$  such that $\alpha_i$ satisfies $\psi_i$ and does not satisfy all other literals of $\psi$.
But then $x \mapsto (\pi^2_1,\pi^2_2)(\alpha_2(x),\alpha_1(x))$ satisfies none of the literals of $\psi$, and hence $(\pi^2_1,\pi^2_2)$ is not in $\Pol(\bB)$. 
\end{proof}

The final result in this section
connects primitive positive definability and $i$-determined clauses.
We will first (in Lemmas~\ref{lem:wreath}--\ref{lem:extract}) prove a restricted result for 
$\bA_1 \boxtimes \bA_2$ and then extend it to $n$-fold products $\bA_1 \boxtimes \dots \boxtimes \bA_n$ in
Corollary~\ref{cor:extract-n}.
We use it
for defining {\em conjunction replacement} and verify its properties; this concept
is important in our
algorithmic results (see Propositions~\ref{prop:tract2} and \ref{prop:tract-n}). 
The proof exploits
the so-called \emph{wreath product}, which is a central 
group-theoretic construction.
Let us denote $\bA_1 \boxtimes \bA_2$ by $\bA$.
The wreath product will be used for
concisely describing the automorphism group of $\bA^{*j}$, where
$\bA^{*j}$, $j \in \{1,2\}$, denotes the structure with domain $A_1 \times A_2$ that contains all relations that are defined by a $j$-determined clause.
If $G$ is a permutation group on $A_1$ and $H$ is a permutation group on $A_2$, then the \emph{action of the (unrestricted) wreath product
$G \ltimes H^{A_1}$ on $A_1 \times A_2$} (also denoted in the literature by $H \, \text{Wr}_{A_1} \, G$) is the
permutation group 
$$\{ (a_1,a_2) \mapsto (\alpha(a_1),\beta_{a_1}(a_2))  \mid 
\alpha \in G, \beta_{a} \in H \text{ for every } a \in A_1 \}.$$

\begin{lemma}\label{lem:wreath} 
For any structures $\bA_1$ and $\bA_2$, let $\bA$ denote $\bA_1 \boxtimes \bA_2$. Then it holds that
\[\Aut(\bA^{*1}) = \Aut(\bA_1) \ltimes \Sym(A_2)^{A_1}\]
and if $\bA_1$ is homogeneous then so is $\bA^{*1}$. 
 The analogous statements hold for $\bA^{*2}$. 
\end{lemma}
\begin{proof}
To show that $\Aut(\bA_1) \ltimes \Sym(A_2)^{A_1} \subseteq \Aut(\bA^{*1})$, let $\psi(x_1,\dots,x_m)$ be a 1-determined clause and let $((s_1,t_1),\dots,(s_m,t_m))$ be a tuple that satisfies $\psi$;
i.e., there exists an atomic $\tau_1$-formula $\phi$
such that $(s_1,\dots,s_m)$ satisfies $\phi$. 
For $\alpha \in \Aut(\bA_1)$ and $\beta_{s_1},\dots,\beta_{s_m} \in \Sym(A_2)$,
note that $$((\alpha(s_1),\beta_{s_1}(t_1)),\dots,(\alpha(s_m),\beta_{s_m}(t_m)))$$ satisfies $\psi$ since
$(\alpha(s_1),\dots,\alpha(s_m))$ satisfies $\phi$. 

To show that $\Aut(\bA^{*1}) \subseteq \Aut(\bA_1) \ltimes \Sym(A_2)^{A_1}$, let $\gamma \in \Aut(\bA^{*1})$. 
Arbitrarily fix $t \in A_2$. 
The operation $\gamma$ preserves $=_1$ so the operation $\alpha$ defined by  $s \mapsto \gamma((s,t))_1$ is well-defined
and it is an automorphism of $\bA_1$. The operation $\gamma$ is bijective
so for every $s \in A_1$, the map 
$\beta_s$ defined by $t \mapsto \gamma(s,t)_2$ is a member of $\Sym(A_2)$. Since $\gamma$ equals the map
that sends $(s,t)$ to $(\alpha(s),\beta_s(t))$, this shows that $\gamma \in \Aut(\bA_1) \ltimes \Sym(A_2)^{A_1}$. 

Now suppose that $\bA_1$ is homogeneous. 
Let $\alpha$ be an isomorphism between $m$-element 
substructures of $\bA_1^{*1}$ that maps $(s_j,t_j)$ to $(s_j',t'_j)$ for $j \in \{1,\dots,m\}$ and $m \in {\mathbb N}$. Note that if $s_j = s_k$, then $s_j' = s_k'$ because $\alpha$ must preserve the relation $=_1$. Hence, the map $\alpha_1$ that sends $s_j$ to $s_j'$ is a well-defined map between finite subsets of $A_1$. Moreover, since $\alpha$ preserves all $i$-determined clauses, the map $\alpha_1$ is in fact an isomorphism between finite substructures of $\bA_1$, and hence can be extended to an automorphism $\beta$ of $\bA_1$ by the homogeneity of $\bA_1$. 
Note that if $p,q$ are distinct
and $s_p \neq s_q$, then
$\alpha(s_p,t_p) \neq \alpha(s_q,t_q)$, because $\alpha$ is injective. Hence, for each $s$ in the domain of $\alpha_1$ 
we may fix a bijection
$\gamma_s$ of $A_2$ such that
$\alpha(s,t) = (\beta(s),\gamma_s(t))$ for all
$t$ such that $(s,t)$ lies in the domain of $\alpha$. For all other $s \in A_1$ we may define $\gamma_s$ to be the identity. 
Then the map that sends $(a,b)$ to $(\beta(a),\gamma_a(b)) \in \Aut(\bA^{*1})$ extends $\alpha$. We conclude that $\bA^{*1}$ is homogeneous.
\end{proof}

\begin{lemma}\label{lem:wreath-pres} 
Assume that
$\bA_1$ and $\bA_2$ are countable, homogeneous, and $\omega$-categorical. 
 A relation $R$ can be defined by a conjunction of $1$-determined clauses 
over $\bA_1 \boxtimes \bA_2$ if and only if it is preserved by the wreath product 
 $$\Aut(\bA_1) \ltimes \Sym(A_2)^{A_1} $$
 in its action on $A_1 \times A_2$; the analogous characterisation holds for clauses that are $2$-determined. 
\end{lemma}

\begin{proof}
The forward implication is an immediate consequence of Lemma~\ref{lem:wreath}. Conversely, suppose that $R$ is preserved by $\Aut(\bA_1) \ltimes \Sym(A_2)^{A_1}$. 
Recall that $\Aut(\bA^{*1}) = \Aut(\bA_1) \ltimes \Sym(A_2)^{A_1}$ by Lemma~\ref{lem:wreath}.
The structure $\bA_1$ is homogeneous by assumption so $\bA^{*1}$ is homogeneous, too. 
We have assumed that $\bA_1$ and $\bA_2$ are $\omega$-categorical so $\bA_1 \boxtimes \bA_2$
is $\omega$-categorical by Remark~\ref{rem:hom}.
Consequently, $\bA^{*1}$ is $\omega$-categorical since it is a first-order reduct of $\bA_1 \boxtimes \bA_2$. It follows that $R$ is first-order definable in 
$\bA^{*1}$, and even has a quantifier-free definition because $\bA^{*1}$ is homogeneous. This implies that $R$ can be defined by a conjunction of $1$-determined clauses over $\bA_1 \boxtimes \bA_2$. 
\end{proof}

\begin{lemma}\label{lem:extract}

Assume the following:

\begin{enumerate}
\item
$\bA_1$ and $\bA_2$ are countable, homogeneous, and $\omega$-categorical,

\item
$\bB$ is a first-order expansion of $\bA_1 \boxtimes \bA_2$, and

\item
$\phi_1 \wedge \phi_2$ is a formula that defines a relation $R$
over $\bA_1 \boxtimes \bA_2$ 
such that
$\phi_1$ is a conjunction of  $1$-determined clauses 
and such that $R$ is primitively positively definable over $\bB$. 
\end{enumerate}

Then there exists 
a conjunction $\psi_1$ of $1$-determined clauses 
which is equivalent to a  primitive positive formula over $\bB$ such that $\psi_1 \wedge \phi_2$ still defines $R$ over $\bA_1 \boxtimes \bA_2$.
An analogous statement holds when $\phi_2$ is a conjunction of $2$-determined clauses.
\end{lemma}
\begin{proof} 
Suppose that items 1--3 hold true.
Let $\psi(x_1,\dots,x_m)$ be the formula $$\exists y_1,\dots,y_m  \left (R(y_1,\dots,y_m) \wedge \bigwedge_{j \in \{1,\dots,m\}} x_j =_1 y_j \right).$$
Note that $\psi$ is equivalent to a primitive positive formula over $\bB$.

We first show that the relation $S$ defined by $\psi$ over $\bB$ can be defined by a conjunction of $i$-determined clauses over $\bA_1 \boxtimes \bA_2$. We use Lemma~\ref{lem:wreath-pres}. 
Let $((a^1_1,a^1_2),\dots,(a^m_1,a^m_2)) \in S$, let $\alpha \in \Aut(\bA_1)$, and 
let $\pi_1, \dots, \pi_m \in \Sym(A_2)$ be such that
  $\pi_p = \pi_q$ whenever $a^p_1 = a^q_1$ for $p,q \in \{1,\dots,m\}$. 
 We have to show that 
 $$t := \big ((\alpha(a^1_1),\pi_1(a^1_2)),\dots,(\alpha(a^m_1),\pi_m(a^m_2)) \big)$$ satisfies $\psi$ as well. 
 Let $(b^1_1,b^1_2)$,
 \dots,
 $(b^m_1,b^m_2) \in B$ be the witnesses from $B$ for the existentially quantified variables of $\psi$ that show that $\psi$ holds for $((a^1_1,a^1_2),\dots,(a^m_1,a^m_2))$. Then the tuple
 $((\alpha(b^1_1),b^1_2),\dots,(\alpha(b^m_1),b^m_2))$ 
 provides witnesses that show that the same formula holds for $t$:
 \begin{itemize}
     \item 
     $R((\alpha(b^1_1),b^1_2),\dots,(\alpha(b^m_1),b^m_2))$ holds in $\bB$ 
     because 
     $R((b^1_1,b^1_2),
 \dots,
 (b^m_1,b^m_2))$ holds in $\bB$ 
     and 
     $(\alpha,\id) \in \Aut(\bB)$, and 
     \item  $(\alpha(a^j_1),\pi(a^j_2)) =_1 (\alpha(b^j_1),
     b^j_2)$ holds for every $j \in \{1,\dots,m\}$ because 
      $(a^j_1,a^j_2) =_1 (b^j_1,b^j_2)$. 
 \end{itemize}
 Lemma~\ref{lem:wreath-pres} shows
 that there exists a conjunction $\psi_1$ of $1$-determined clauses that is equivalent to $\psi$.
 Clearly, $\psi_1$ implies $\phi_1$ (since $\phi_1$ is $1$-determined),
 so $\psi_1 \wedge \phi_2$ defines $R$ in $\bB$ and this concludes the proof.  

The case when $\phi_2$ is a conjunction of $2$-determined clauses can be treated analogously.
\end{proof}

We continue by generalising the previous lemma to algebraic products involving more than two structures. To this end, we need a particular notion that generalizes $i$-determined clauses. Let $S \subseteq \{1,\dots,n\}$. A clause over $\bA_1 \boxtimes \dots \boxtimes \bA_n$ is called \emph{$S$-determined} if all atomic formulas in the clause are built from symbols that have been introduced for atomic $\tau_i$-formulas for some $i\in S$.

\begin{corollary}\label{cor:extract-n}
Assume the following:
\begin{enumerate}
\item $\bA_1, \dots, \bA_n$ are countable, homogeneous, and $\omega$-categorical,
\item $\bB$ is a first-order expansion of $\bA_1 \boxtimes \dots \boxtimes \bA_n$, and
\item $\phi_{S} \wedge \phi$ is a formula that defines a relation $R$ over $\bA_1 \boxtimes \dots \boxtimes \bA_n$
such that for some $S \subseteq \{1, \dots, n\}$, the formula
$\phi_{S}$ is a conjunction of  $S$-determined clauses 
and such that $R$ is primitively positively definable over $\bB$. 
\end{enumerate}
Then there exists 
a conjunction $\psi_{S}$ of $S$-determined clauses
which is equivalent to a  primitive positive formula over $\bB$ such that $\psi_{S} \wedge \phi$ still defines $R$ over $\bA_1 \boxtimes \dots \boxtimes \bA_n$.
\end{corollary}

\begin{proof}
Assume without loss of generality that $S = \{1, \dots, p\}$ for some $p \geq 1$. We can view the $n$-fold product $\bA_1\boxtimes \dots \boxtimes \bA_n$ as $\bB_1 \boxtimes \bB_2$, where $\bB_1 = \bA_1 \boxtimes \dots \boxtimes \bA_{p}$ and $\bB_2 = \bA_{p+1} \boxtimes \dots \boxtimes \bA_n$. Note that $\phi_1$ is a conjunction of $1$-determined clauses when considered as a formula over $\bB_1 \boxtimes \bB_2$. By Lemma~\ref{lem:extract}, there exists a conjunction $\psi_{S}$ of $1$-determined clauses which is equivalent to a primitive positive formula over $\bB$ such that $\psi_{S} \wedge \phi$ still defines $R$. Since $\psi_{S}$ is $S$-determined when viewed as a formula over $\bA_1\boxtimes \dots \boxtimes \bA_n$, the claim follows.
\end{proof}

Let $\bA_1, \dots \bA_n$, $\bB$, $S$, and $\phi_{S} \wedge \phi$ be as in the statement of Corollary~\ref{cor:extract-n}. Arbitrarily choose a conjunction $\psi_S$ of $S$-determined clauses equivalent to a primitive positive formula over $\bB$ such that $\psi_S \wedge \phi$ is equivalent to $\phi_S \wedge \phi$. Note that the existence of $\psi_{S}$  follows from the corollary. We denote the formula $\psi_{S}$ by $\Cr(\phi_S \wedge \phi, S, \phi_{S})$, where $\Cr$ stands for {\em conjunction replacement}.
Note that since $\psi_S$ is equivalent to a primitive positive formula, it is preserved by $\Pol(\bB)$, which will be relevant for proving syntactic restrictions on $\psi_S$ and use of known algorithms for temporal CSPs.

\begin{rem}
In a typical example of a formula $\phi_S \wedge \phi$, where $\phi_S$ is a conjunction of all $S$-determined clauses of $\phi_S \wedge \phi$, conjunction replacement is not needed, because $\phi_S$ is already preserved by $\Pol(\bB)$ and therefore equivalent to a primitive positive formula over $\bB$. This is because syntactic normal forms are often defined as conjunctions of clauses of some specific shape. Nevertheless, one cannot guarantee that a general defining first-order formula of $R$ will be of this shape and therefore $\phi_S$ might not be preserved by $\Pol(\bB)$. Since we want to exploit the syntactic normal forms proved for temporal CSPs (see Section~\ref{sec:tempcons}), we ensure this property using the operator $\Cr$.
\end{rem}

\section{Algebraic Powers of $({\mathbb Q};<)$}
\label{sect:products}

In this section we classify the complexity of the CSP for every first-order expansion 
of the structure $({\mathbb Q};<)^{(n)}$.
The main outline of the section is as follows: we consider the case when
$n=2$ in Section~\ref{sect:2D-class}, generalize to arbitrary $n \geq 2$ in
Section~\ref{sect:n-dim}, and finally specialize
our results to binary signatures in Section~\ref{sect:binary}. 
We begin the section by recapitulating some
known results concerning first-order
expansions of $({\mathbb Q};<)$ (Section~\ref{sec:tempcons}),
studying the polymorphisms of $({\mathbb Q};<) \boxtimes ({\mathbb Q};<)$ (Section~\ref{sec:polymorphisms_prod_temp_constr}) and
presenting syntactic normal forms of certain relations that are first-order definable in 
$({\mathbb Q};<) \boxtimes ({\mathbb Q};<)$ (Section~\ref{sect:decomp}).

Recall that $({\mathbb Q};<) \boxtimes ({\mathbb Q};<) = ({\mathbb Q}^2;<_1,=_1,<_2,=_2)$ 
is 
$\omega$-categorical and homogeneous (Remark~\ref{rem:hom}),
and therefore has quantifier elimination.
From here until Section~\ref{sect:2D-class}, we let the symbol $\bD$ denote a first-order expansion of $({\mathbb Q}^2;<_1,=_1,<_2,=_2)$.

\subsection{First-order Expansions of $({\mathbb Q};<)$}
\label{sec:tempcons}
Let $\bB$ be a first-order reduct of $({\mathbb Q};<)$
with a finite relational signature. 
The complexity of CSP$(\bB)$
for all choices of $\bB$ has been determined by Bodirsky and Kára~\cite{tcsps-journal}.
For our purposes, it is sufficient to understand the complexity of all
first-order expansions of $({\mathbb Q};<)$ with a finite relational structure.
We next present first-order expansions of $({\mathbb Q};<)$ with a polynomial-time solvable CSP, we describe some of their polymorphisms,
and how the relations can be described with syntactically restricted definitions.
We make use of relational and functional dualities to simplify the presentation.

\begin{definition}
The \emph{dual} of a 
relation $R \subseteq \Q^k$ is the relation $$R^* \coloneqq \{(-x_1,\dots,-x_k) \mid (x_1,\dots,x_k) \in R\}.$$ If $\bB$ is a relational  structure with domain $\Q$, then the \emph{dual} 
of $\bB$ is the structure with domain $\Q$ and the same signature $\tau$ as $\bB$ where
$R \in \tau$ denotes $(R^{\bB})^*$. 
Similarly, if $f \colon {\mathbb Q}^n \to {\mathbb Q}$ is an operation, then the \emph{dual} of $f$ is the operation $f^*$ defined as follows.
$$(x_1,\dots,x_n) \mapsto -f(-x_1,\dots,-x_n)$$ 
If ${\mathscr C}$ is a operation clone on ${\mathbb Q}$, then the \emph{dual} of ${\mathscr C}$ is the operation clone ${\mathscr C}^* \coloneqq \{f^* \mid f \in {\mathscr C} \}.$
\end{definition}

\noindent
Now, consider the following first-order
expansions of $({\mathbb Q};<)$.

\begin{itemize}

        \item $\bU \coloneqq (\Q;<,R^{\min}_{\leq})$ where
    $R^{\min}_{\leq} \coloneqq \{(x,y,z) \in {\mathbb Q}^3 \mid y \leq x \; {\rm or} \; z \leq x \}$.

    \item $\bX \coloneqq (\Q;<,X)$ where  $X \coloneqq \{(x,y,z) \in {\mathbb Q}^3 \mid x=y < z \; {\rm or} \; y=z < x \; {\rm or} \; z=x < y \}$
\item $\bI \coloneqq (\Q;<,R^{\mi},S^{\mi})$ where
\begin{align*}
    R^{\mi} & \coloneqq \{(x,y,z) \in {\mathbb Q}^3 \mid y \leq x \; {\rm or} \; z < x \} \; {\rm and} \\
    S^{\mi} & \coloneqq \{(x,y,z) \in {\mathbb Q}^3 \mid y \neq x \; {\rm or} \; z \leq x \}. 
\end{align*}
   
 \item $\bL \coloneqq (\Q;<,L,I_4)$ where
    \begin{align*}
        L & \coloneqq \{(x,y,z) \in {\mathbb Q}^3 \mid y < x \; {\rm or} \; z < x \; {\rm or} \; x = y = z\} \; {\rm and} \\
        I_4 & \coloneqq \{(x,y,u,v) \in \Q^4 \mid x=y \; {\rm implies} \; u=v\}. 
    \end{align*}   
\end{itemize}

We need the following characterisation of primitive positive definability in these structures in terms of certain polymorphisms. The precise definition of these operations can be found in~\cite{RydvalFP} but it is not needed in this article; the properties stated in the next proposition suffice for our purposes.

    \begin{theorem}[Proposition 7.27 in~\cite{RydvalFP}]
\label{thm:pwnu}
For every $k \geq 3$ there are pseudo weak near unanimity polymorphisms $\min_k,\mx_k,\mi_k,\lele_k$ of arity $k$ such that a relation 
 $R \subseteq {\mathbb Q}^m$ is preserved by $\min_k$ ($\mx_k$, $\mi_k$, $\lele_k$) and $\Aut({\mathbb Q};<)$ if and only if
$R$ has a primitive positive definition in $\bU$ ($\bX$, $\bI$, $\bL$). The operation $\lele_k$ is injective.
\end{theorem}

The following result is essentially taken from~\cite{tcsps-journal}
but we formulate it differently with the aid of polymorphisms.

\begin{theorem}[Theorem 12.0.1 in \cite{Book}]\label{thm:tcsps}
Let $\bB$ be a first-order expansion of $(\Q;<)$. 
Then exactly one of the following two cases applies. 
\begin{enumerate}
    \item 
    $\bB$ is preserved by the operation $\min_3$, $\mx_3$, $\mi_3$, or $\lele_3$ from Theorem~\ref{thm:pwnu}, or the dual of one of these operations. In this case, the CSP of every finite-signature reduct of $\bB$ is in P. 
\item $\Pol(\bB)$ has a uniformly continuous minor-preserving map to $\Pol(K_3)$. In this case, $\bB$ has a finite-signature reduct whose CSP is NP-complete. 
\end{enumerate}
\end{theorem}

Theorem~\ref{thm:tcsps} immediately connects first-order expansions of $(\Q;<)$ with the infinite-domain
tractability conjecture from Section~\ref{sect:cat}.

We now describe polymorphisms and syntactic normal forms of the structures that were described 
earlier.
Clearly, an operation $f$ preserves a relation $R \subseteq \Q^k$ if and only if $f^*$ preserves $R^*$ so we can concentrate on 
the structures $\bU$, $\bX$, $\bI$, and $\bL$.
We start by considering polymorphisms of the structures $\bU$, $\bX$, and $\bI$.

\begin{definition} \label{def:pp}
A binary operation $f$ on ${\mathbb Q}$ is called a
\emph{$\pp$-operation} if $f(a,b) \leq f(a',b')$ if and only if
\begin{enumerate}
        \item $a\leq 0$ and $a \leq a'$, or
        \item $0 < a$, $0<a'$, and $b \leq b'$.
 \end{enumerate}
 \end{definition}
  
 \begin{rem}\label{rem:easy}
     Note that if $f,g \colon {\mathbb Q}^k \to {\mathbb Q}$ are such that for all $a,b \in {\mathbb Q}^k$ we have 
     $f(a) \leq f(b) \Leftrightarrow g(a) \leq g(b)$, then a relation $R$ which is first-order definable in $({\mathbb Q};<)$
     is preserved by $f$ if and only if it is preserved by $g$. To see this, suppose that $R$ is of arity $k$ and $a^1,\dots,a^k \in R$. Then 
     $s \coloneqq f(a^1,\dots,a^k)$ and $t \coloneqq g(a^1,\dots,a^k)$ satisfy
     the same atomic formulas over $({\mathbb Q};<)$, 
     and hence by the homogeneity of 
     $({\mathbb Q};<)$ there exists 
     $\alpha \in \Aut({\mathbb Q};<)$ which maps $s$ to $t$, and since $R$ is first-order definable over $({\mathbb Q};<)$ either both or neither of $s$ and $t$ lie in $R$.
     In particular,
     $R$ is preserved by a 
     $\pp$-operation 
     if and only if it is preserved by all $\pp$-operations. 
     \end{rem}

\begin{proposition}[\cite{tcsps-journal}]\label{prop:pp}
Each of the structures $\bU$, $\bX$, and $\bI$ is preserved by a $\pp$-operation. 
Equivalently, if a relation $R\subseteq \Q^k$ with a first-order definition in $(\Q; <)$ is preserved by $\min_3$, $\mx_3$ or $\mi_3$, then it is preserved by a $\pp$-operation.
\end{proposition}

One should note that if a structure $\bB$ is preserved by a $\pp$-operation, then this does not imply
that $\Csp(\bB)$ is polynomial-time solvable. It does, however, imply that the relations in $\bB$ can be
defined via a restricted form of definitions.

\begin{theorem}[Theorem 4 in~\cite{ToTheMax}]
\label{thm:pp}
Let $R \subseteq {\mathbb Q}^m$ be a relation with a first-order definition in $({\mathbb Q};<)$. Then the following are equivalent. 
\begin{itemize}
    \item $R$ is preserved by a (equivalently: every) $\pp$-operation. 
    \item $R$ has a definition by a conjunction of clauses of the form
    $$ y_1 \neq x \vee \cdots \vee y_k \neq x \vee z_1 \leq x \vee \cdots \vee z_l \leq x$$
    where it is permitted that $l = 0$ or $k=0$. 
\end{itemize}
\end{theorem}

We conclude this section by another characterisation of $\bL$ via polymorphisms and presenting a syntactic normal form.

\begin{definition}\label{def:lele}
A binary operation $f$ on $\mathbb{Q}$ is called an \emph{$\lele$-operation} if $f(a,b)<f(a',b')$ if and only if 
 \begin{enumerate}
 	\item $a\leq 0$ and $a<a'$, or
 	\item $a\leq 0$ and $a=a'$ and $b<b'$, or
 	\item $a,a'>0$ and $b<b'$, or
 	\item $a>0$ and $b=b'$ and $a<a'$.
\end{enumerate}
\end{definition}

Note that every $\lele$-operation is injective; this fact will be important in some of the forthcoming proofs.
A visualisation of a $\pp$-operation and an $\lele$-operation can be found in Figure~\ref{fig:ppll}.

 \begin{figure}
     \centering
     \includegraphics{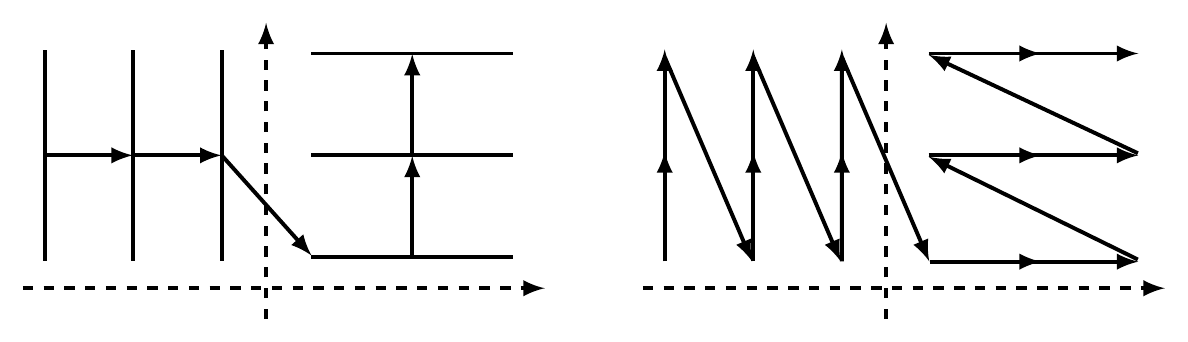}
     \caption{A visualisation of a pp-operation (left) and an $\lele$-operation (right) \cite{tcsps}. Arrows depict the growth of values.}
     \label{fig:ppll}
 \end{figure}

\begin{definition} \label{def:ll-hornclause}
A formula is an \emph{ll-Horn clause}
if it is of the form $$x_1 \neq y_1 \vee \cdots \vee x_m \neq y_m \vee z_1 < z_0 \vee \cdots \vee z_{\ell} < z_0$$
or
$$x_1 \neq y_1 \vee \cdots \vee x_m \neq y_m \vee z_1 < z_0 \vee \cdots \vee z_{\ell} < z_0 \vee (z_0 = z_1 = \cdots = z_{\ell})$$
where it is permitted that $l = 0$ or $m=0$.
\end{definition}

We also need \emph{lexicographic operations} in order
to formulate the final theorem of this section.

\begin{definition} \label{def:lex}
A binary operation $f$ on $\mathbb{Q}$ is called a \emph{$\lex$-operation} if $f(a,b)<f(a',b')$ if and only if 
 \begin{itemize}
 	\item $a < a'$, or
 	\item $a = a'$ and $b < b'$.
\end{itemize}
It is called a \emph{twisted $\lex$-operation} 
if $f(a,-b)$ is a $\lex$-operation. 
\end{definition}

\begin{rem}\label{rem:ll-lex}
Every relation $R \subseteq {\mathbb Q}^k$ with a first-order definition in $({\mathbb Q};<)$ that is preserved by an $\lele$-operation is also preserved by all $\lex$-operations.
\end{rem}

\begin{theorem}[\cite{ll} and~\cite{MottetMPRI}; also see Theorem 12.7.3 and Lemma 12.4.4 in~\cite{Book}]\label{thm:lele}
Let $R \subseteq \Q^k$ be a relation with a first-order definition in $(\Q;<)$. Then the following are equivalent. 
\begin{itemize}
    \item $R$ has a primitive positive definition in $\bL$. 
    \item $R$ is preserved by an (equivalently: every) $\lele$-operation.
    \item $R$ is preserved by $\lele_k$ (from Theorem~\ref{thm:pwnu}) for some (equivalently: for all) $k \geq 3$.
    \item $R$ has a definition by a conjunction of ll-Horn clauses. 
    \end{itemize}
Moreover, if $R$ is preserved by a $\pp$-operation and by a $\lex$-operation, then $R$ is preserved by an $\lele$-operation.
\end{theorem}

The operations discussed in this section will be used frequently in the sequel.
A concise summary can be found in Table~\ref{tb:op-defs}.

\begin{table} 
\caption{Summary of Definitions~\ref{def:pp}, \ref{def:lele}, and~\ref{def:lex}.}
\begin{tabular}{|lrcl|}\hline
\rule{0pt}{12pt}
$f$ is pp if & $f(a,b) \leq f(a',b')$ & $\Longleftrightarrow$ & $(a \leq 0 \wedge a \leq a') \vee (0 < a \wedge 0 < a' \wedge b \leq b')$\\[5pt] \hline

\rule{0pt}{12pt}
$f$ is ll if & $f(a,b) < f(a',b')$ & $\Longleftrightarrow$ & $(a \leq 0 \wedge a<a') \vee (a \leq 0 \wedge a=a' \wedge b<b') \; \vee$ \\ & & & $(a>0 \wedge a'>0 \wedge b<b') \vee (a>0 \wedge b=b' \wedge a < a')$\\[5pt] \hline

\rule{0pt}{12pt}
$f$ is lex if & $f(a,b) < f(a',b')$ & $\Longleftrightarrow$ & $(a < a') \vee (a=a' \wedge b < b')$ \\[5pt] \hline

\end{tabular}
\label{tb:op-defs}
\end{table}

\subsection{Polymorphisms} \label{sec:polymorphisms_prod_temp_constr}

We will now analyse the polymorphism clones of first-order expansions of $({\mathbb Q};<) \boxtimes ({\mathbb Q};<)$. This involves a study of \emph{canonical functions} (see, e.g.,~\cite{BP-canonical}) in the product setting. 

We formulate the results for the polymorphisms of first-order expansions of $({\mathbb Q};<) \boxtimes ({\mathbb Q};<)$, since they can be described more explicitly than in first-order expansions of $(\Q;<)^{(n)}$.
Nevertheless, generalised versions of Lemma~\ref{lem:canoni}, Lemma \ref{lem:canonise}, Lemma \ref{lem:dom}, Corollary \ref{cor:dom} and Lemma \ref{lem:flip} for first-order expansions of $(\mathbb Q;<)^{(n)}$ can be proved in a similar fashion.
The basic idea is to choose one or two dimensions from $\{1,\dots, n\}$ that are referred to in the statements. For a
concrete example, see the proof of Proposition~\ref{prop:syntax-n} that is a generalization of Corollary~\ref{cor:dom}. Generalizations of this kind will be 
important in Section \ref{sect:n-dim}.

Let $G$ be a permutation group on a set $A$ and let $H$ be a permutation group on a set $B$. 
A function $f \colon A \to B$ is called \emph{canonical with respect to $(G,H)$} if for every $m \in {\mathbb N}$, $t \in A^m$, and $\alpha \in G$ there exists a $\beta \in H$ such that $f \alpha(t) = \beta f(t)$ (where functions are applied to tuples componentwise). 
If $f$ is canonical with respect to $(\Aut(\bA)^n,\Aut(\bA))$ for some $n \in {\mathbb N}$,  then we say that $f$ is \emph{canonical over $\Aut(\bA)$}. 
In other words, $f$ is canonical over $\Aut(\bA)$ 
if and only if for every $m \in {\mathbb N}$ and all $t_1,\dots,t_n \in A^m$ the orbit of $f(t_1,\dots,t_n)$ in $\Aut(\bA)$ only depends on the orbits of $t_1,\dots,t_n$ in $\Aut(\bA)$. 
Note that if $\bB = \bA_1 \boxtimes \bA_2$, then an operation $f$ is canonical over $\Aut(\bB)$ if 
$\theta_1(f)$ is canonical over $\bA_1$ and $\theta_2(f)$ is canonical over $\bA_2$.

\begin{ex}
Let $f$ be a $\lex$-operation and $g$ a twisted $\lex$-operation (see Definition~\ref{def:lex}). Then $f$ and $g$ are canonical over $\Aut(\Q;<)$ and $(f,g)$ is canonical over $\Aut((\Q;<) \boxtimes (\Q; <))$.
\end{ex}

A permutation group $G$ is called \emph{extremely amenable} if every continuous action of $G$ on a compact Hausdorff space has a fixed point. The reader need not be familiar with this notion since it will only be used in a black-box fashion via Theorem~\ref{thm:canon} below; we refer the interested reader to~\cite{Topo-Dynamics}. 
A fundamental example of a structure with an extremely amenable automorphism group is $({\mathbb Q};<)$. Moreover, direct products of extremely amenable groups are extremely amenable~\cite{Topo-Dynamics}.

\begin{theorem}[see, e.g.,~\cite{BPT-decidability-of-definability,BP-canonical}]\label{thm:canon}
Let $G$ be an extremely amenable permutation group on a set $A$, let $H$ be an oligomorphic permutation group on a set $B$, and let $f \colon A \to B$ be a function. Then 
$$ \overline{ \{ \beta f  \alpha \mid \alpha \in G, \beta \in H \}}$$ 
contains a canonical function with respect to $(G,H)$. 
\end{theorem}

The following result will be useful later on when we analyse the polymorphisms of first-order expansions of powers of $({\mathbb Q};<)$. If $f$ is an operation of arity $k$ and
$\alpha_1,\dots,\alpha_n$ are unary operations, then we write $f(\alpha_1,\dots,\alpha_n)$ to denote the function $(x_1,\dots,x_n) \mapsto f(\alpha_1(x_1),\dots,\alpha_n(x_n))$.

\begin{lemma}\label{lem:canoni}
Let $\bA_1,\bA_2$ be $\omega$-categorical structures such that $\Aut(\bA_1)$ is extremely amenable and 
assume $f \in \Pol(\bA_1 \boxtimes \bA_2)$ has arity $n$. 
Then, the set
$$ \mathscr C \coloneqq
\overline{ \{ \alpha_0 f (\alpha_1,\dots,\alpha_n) \mid \alpha_0 \in \Aut(\bA_1\boxtimes \bA_2),\; \alpha_j \in \Aut(\bA_1)\times \{\id_{A_2}\} \text{ for all }
j = 1, 2,\dots,n \} }
$$
contains an operation $g$ 
such that $\theta_1(g)$ is canonical over $\Aut(\bA_1)$, and
$\theta_2(g) = \theta_2(f)$. 
The symmetric statement holds if the roles of $\bA_1$ and $\bA_2$ are exchanged.

\end{lemma}
\begin{proof}
By Theorem~\ref{thm:canon} there exists 
an operation $$g'' \in  \overline{\{\alpha_0 \theta_{1}(f) (\alpha_1,\dots,\alpha_n) \mid \alpha_0,\alpha_1,\dots,\alpha_n \in \Aut(\bA_1)\}}$$ 
which is canonical over $\Aut(\bA_1)$. Note that $g'' \in \overline{\theta_1(\mathscr C)}$.
By Proposition~\ref{prop:closed}, $\theta_1(\mathscr{C})$ is closed and therefore there exists 
$g' \in \mathscr C$ such that $\theta_{1}(g') = g''$. 
Arbitrarily
choose $a^1, \dots, a^k$ in $A_2^n$. The definition of $\mathscr C$ implies that
there is an automorphism $\alpha \in \Aut(\bA_1 \boxtimes \bA_2)$ such that $\theta_2(g')(a^i)=\theta_2(\alpha) \theta_2(f)(a^i)$ for $i=1,\dots, k$, 
which we can rewrite
as 
$\theta_2(\alpha^{-1}) \theta_2(g') (a^i) = \theta_2(f)(a^i)$.
This shows that $\theta_2(f)  \in \overline{ \{ \beta \theta_2(g') \mid \beta \in \Aut(\bA_2) \} }$. 
Let $S :=  \overline{ \{ \gamma g' \mid \gamma \in \Aut(\bA_1\boxtimes \bA_2)\} }$. 
Note that $\theta_2(S)$ contains 
$\{\beta \theta_2(g') \mid \beta \in \Aut(\bA_2) \}$ so
$\overline{\theta_2(S)}$ contains $\theta_2(f)$. 
Applying 
Proposition~\ref{prop:closed} to the set 
$S$ implies that 
$\theta_2(S)$ is closed. Hence, 
there exists $g \in S$ such that $\theta_2(g) = \theta_2(f)$.
Note that for every finite subset $F$ of $A_1$, $\theta_1(g)|_F=\gamma'g''|_F$ for some $\gamma' \in \Aut(\bA_1)$. Therefore, $\theta_1(g)$ is canonical over $\Aut(\bA_1)$, because $g''$ is.
\end{proof}

We continue to introduce terminology.
An operation $f \colon A^k \to A$ is called \emph{essentially unary} if there exists $i \in \{1,\dots,k\}$ and a unary operation $g \colon A \to A$ such that $f(x_1,\dots,x_k) = g(x_i)$ for all $x_1,\dots,x_k \in A$. 
Let $S \subseteq {\mathbb Q}$. 
An operation $f \colon {\mathbb Q}^2 \to {\mathbb Q}$ is called \emph{dominated by the first argument on $S$} if $f(x,y) < f(x',y')$ for all $x,x' \in S$ such that  $x<x'$. 
If an operation $f \colon {\mathbb Q}^2 \to {\mathbb Q}$ 
is dominated by the first argument on all of ${\mathbb Q}$,
we say that is  \emph{dominated by the first argument}. 
Examples of operations that are dominated by their first argument
are $\lex$-operations, twisted $\lex$-operations, and order-preserving operations that 
only depend on the first argument.

Recall that $\bD$ is an arbitrary fixed first-order expansion of $(\Q; <) \boxtimes (\Q, <) = (\Q; <_1, =_1, <_2, =_2)$. Our aim is now to show that $\Pol(\bD)$ contains an operation with suitable domination properties (Lemma~\ref{lem:dom}).
This lemma will be a cornerstone in the proof of our first result on syntactic normal forms (Proposition~\ref{prop:decomp}).
We first note that
the binary polymorphisms of $({\mathbb Q};<)$ that are canonical over $\Aut({\mathbb Q};<)$ can
be given a succinct characterisation.

\begin{lemma}[see, e.g.,~Example 11.4.13 in~\cite{Book}]
\label{lem:lex}
Assume that $f \in \Pol({\mathbb Q};<)$ is a binary operation that is canonical over $\Aut({\mathbb Q};<)$. Then, either $f$ is 
essentially unary, or $f$ 
is a $\lex$-operation or a twisted $\lex$-operation, or the operation $(x,y) \mapsto f(y,x)$ is a $\lex$-operation or a twisted $\lex$-operation. 
\end{lemma}

We now turn our attention to the structure ${\mathfrak D}$ 
and obtain the following intermediate result
by analysing operations $g$ in $\Pol(\bD)$
that are canonical in a particular dimension. Note that the following statements of Lemma \ref{lem:canonise}, Lemma \ref{lem:dom} and Corollary \ref{cor:dom} remain true also if the duals of $\lele$- or $\pp$-operations are used. 

\begin{lemma}\label{lem:canonise}
If $\Pol(\bD)$ contains an operation $f$ such that 
$\theta_1(f)$ is an $\lele$-operation, 
then $\Pol(\bD)$ also contains an operation $g$
such that $\theta_1(g)$ is an $\lele$-operation
and $\theta_2(g)$ or $(x,y) \mapsto \theta_2(g)(y,x)$ is  either a $\lex$-operation
or essentially unary (and in particular preserves $\leq_2$ and $\neq_2$). 
The analogous statement holds if $\theta_1(f)$  is a $\pp$-operation.
\end{lemma}
\begin{proof}
Apply Lemma~\ref{lem:canoni} to the operation $f$ for dimension $i=2$ and let 
$g\in \Pol(\bD)$
be the resulting operation such that $\theta_2(g)$ is canonical and $\theta_1(g) = \theta_1(f)$.
By Lemma~\ref{lem:lex}, either $\theta_2(g)$ is essentially unary, or a $\lex$-operation, or a twisted $\lex$-operation, or $(x,y) \mapsto \theta_2(g)(y,x)$ is a $\lex$-operation, or a twisted $\lex$-operation. 
If $\theta_2(g)$ is a twisted $\lex$-operation, then 
we consider $g'$ defined by $g'(x,y) \coloneqq g(x,g(x,y))$ which is a $\lex$-operation. The argument if $(x,y) \mapsto \theta_2(g)(y,x)$ is a twisted $\lex$-operation is similar.
We finally note that $\theta_1(g')$ is an $\lele$-operation. 
The same proof works
if $\theta_1(f)$ is a  $\pp$-operation.
\end{proof}

In our final step, we show that if $\Pol(\bD)$ contains an operation that satisfies the preconditions of 
Lemma~\ref{lem:canonise}, then the expanded 
structure $(\bD;\leq_1,\neq_1,\leq_2,\neq_2)$ admits
a polymorphism with a certain domination property.

\begin{lemma}\label{lem:dom}
Let $f \in \Pol(\bD)$ be such that  $\theta_1(f)$ is an $\lele$-operation or a $\pp$-operation
and let $i \in \{1,2\}$. Then $\Pol(\bD;\leq_1,\neq_1)$ contains an operation $g$ such that $\theta_2(g) = \theta_2(f)$ and
$\theta_1(g)$ is dominated by the $i$-th argument. If $\theta_1(f)$ is 
a $\pp$-operation, 
then $g$ can be chosen
such that $\theta_1(g)$  equals $\pi^2_i$.
\end{lemma}

\begin{proof}
We begin with the case $i=1$. Define
\begin{align*}
    U = \overline{ \{ \beta f (\alpha,\id_{{\mathbb Q}^2}) \mid \alpha,\beta \in \Aut({\mathbb Q};<)\times \{\id_\mathbb{Q}\} \} } 
\end{align*}
and note that
$U \subseteq \Pol(\bD;\leq_1,\neq_1)$
since $f\in \Pol(\bD)$ and $\theta_1(f)$
preserves $\leq$ and $\neq$. We claim that $U$
contains 
an operation $g$ such that $\theta_1(g)$ 
is dominated by the first argument. 
To see this, let $S \subseteq {\mathbb Q}$ be finite. If $f$ is an $\lele$-operation,
then choose $\alpha_S \in \Aut({\mathbb Q};<)$ so that $\alpha_S(x) < 0$ for every $x \in S$. Note that  $\theta_1(f((\alpha_S,\id_{\mathbb Q}),\id_{{\mathbb Q}^2}))$ is then dominated by the first argument on $S$.
If $T \subseteq S$, then the homogeneity of $({\mathbb Q};<)$ implies that we can choose $\beta_S,\beta_T \in \Aut({\mathbb Q};<)$ such that  
\[\left( (\beta_S,\id_{\mathbb Q}) f((\alpha_S,\id_{\mathbb Q}),\id_{{\mathbb Q}^2}) \right)|_{S^2}\] 
is an extension of 
\[\left( (\beta_T,\id_{\mathbb Q}) f((\alpha_T,\id_
{\mathbb Q}),\id_{{\mathbb Q}^2}) \right)|_{T^2}.\] 
Hence, $U$ contains an operation $g$ such that $\theta_1(g)$ is dominated by the first argument.  
Moreover, $\theta_2(g)=\theta_2(f)$ since we have only applied automorphisms that fix the second dimension, so $g$ satisfies the statement of the lemma. 

If $\theta_1(f)$ is a $\pp$-operation, then we proceed in the same way but in this case
the operation $\theta_1(g)$
is essentially unary, and by applying automorphisms and using the fact that $\Pol(\bD)$ is closed we may  suppose that $\theta_1(g)$
equals $\pi^2_1$.

The proof when $i=2$ only requires flipping the inequalities in the definition of the automorphisms $\alpha_S$.
\end{proof}

\begin{corollary}\label{cor:dom}
Let $f \in \Pol(\bD)$ be such that $\theta_1(f)$ is an $\lele$- or a $\pp$-operation.
Then there is an operation $g \in \Pol(\bD; \leq_1, \neq_1, \leq_2, \neq_2)$ such that $\theta_i(g)$ is dominated by the $i$-th argument for both $i=1,2$. If $\theta_1(f)$ is a $\pp$-operation,
then we can choose $g$ such that $\theta_1(g)=\pi_1^2$.
\end{corollary}

\begin{proof}
Lemma~\ref{lem:canonise} implies that there is an operation $f'\in \Pol(\bD; \leq_2, \neq_2)$ and an index $j \in \{1,2\}$ such that $\theta_1(f')$ is an $\lele$-operation or a $\pp$-operation, 
and $\theta_2(f')$ is dominated by the $j$-th argument. Hence, by Lemma~\ref{lem:dom} applied on $f'$, there is $g \in \Pol(\bD; \leq_1, \neq_1, \leq_2, \neq_2)$ such that $\theta_i(g)$ is dominated by the $i$-th argument for both $i$ or by the $(3-i)$-th argument for both $i$. Without loss of generality, we can assume that $g$ satisfies the former, because otherwise we can replace $g$ by the operation obtained from $g$ by flipping arguments. Similarly, if $\theta_1(f)$ is a $\pp$-operation,
then we can choose $g$ such that $\theta_1(g)$ is the projection $\pi_1^2$.
\end{proof}

We conclude this section with a duality result that reduces the number of cases that we have to consider in some of the forthcoming proofs.

\begin{lemma}\label{lem:flip}
Let $\bD$ be a first-order expansion of $({\mathbb Q}^2;<_1,=_1,<_2,=_2)$. 
The map given by 
$(x,y) \mapsto (x,-y)$ is an isomorphism  between $\bD$ and a
structure which is primitively positively interdefinable with a 
first-order expansion $\bC$ of $({\mathbb Q}^2;<_1,=_1,<_2,=_2)$.
For every $f \in \Pol(\bD)$ 
there exists $f' \in \Pol(\bC)$ such that 
\begin{itemize}
    \item $\theta_1(f') = \theta_1(f)$ and
       \item $\theta_2(f') = \theta_2(f)^*$. 
\end{itemize} 
\end{lemma}
\begin{proof}
For each relation $R$ of $\bD$, let $\phi$ be the defining formula of $R$ over $({\mathbb Q}^2;<_1,=_1,<_2,=_2)$. Replace each atomic formula of the form $x <_2 y$ in $\phi$ by the formula $y <_2 x$. The relation defined by the formula will be denoted by $R'$. The structure $\bC$ is then the first-order expansion of $({\mathbb Q}^2;<_1,=_1,<_2,=_2)$ by all relations of the form $R'$ where $R$ is a relation of $\bD$.
Then, $(x,y) \mapsto (x,-y)$ is an isomorphism between $\bD$ and a structure $\bC'$ which is the first-order expansion of $({\mathbb Q}^2;<_1,=_1,>_2,=_2)$ by the relations $R'$. Note that $\Pol(\bC')=\Pol(\bC)$.
Let $f \in \Pol(\bD)$ be of arity $k$. Then, the operation
$f'$ defined as 
$$f'((x_1,y_1),\dots,(x_k,y_k)) \coloneqq (\theta_1(f)(x_1,\dots,x_k),\theta_2(f)^*(y_1,\dots,y_k))$$ 
is a polymorphism of $\bC$ that satisfies the requirements of the lemma. 
\end{proof}

\subsection{Syntactic Normal Forms} 
\label{sect:decomp}
In this section we prove that if $\theta_1(\Pol(\bD))$ and 
$\theta_2(\Pol(\bD))$ contain certain
polymorphisms, then the relations of $\bD$ can be defined by formulas satisfying simple syntactic restrictions. 
These restrictions are all of the same kind: the relations can be defined by conjunctions of clauses with
straightforward definitions.
We start with the case that there exists an $f \in \Pol(\bD)$ 
such that $\theta_1(f)$ is an $\lele$-operation or a $\pp$-operation (Proposition~\ref{prop:decomp}). 
We then prove a stronger statement if
$\theta_1(f)$
is a $\pp$-operation
(Proposition~\ref{prop:syntax-2}),
and an even stronger result if additionally there is no binary operation $g \in \Pol(\bD)$ such that $\theta_2(g)$ is a $\lex$-operation
(Proposition~\ref{prop:factor}). 
Finally, we treat the situation that for both $i \in \{1,2\}$ there exists $f_i \in \Pol(\bD)$ such that $\theta_i(f_i)$
is an $\lele$-operation (Proposition~\ref{prop:pp-pow}). These results are collected in Section~\ref{sec:normalforms}.
The proofs of Propositions~\ref{prop:decomp} and \ref{prop:syntax-2} are based on a particular normalisation of formulas that we describe
in Section~\ref{sec:normalisation}.

\subsubsection{Normalisation}
\label{sec:normalisation}
We will now describe a normalisation process for formulas that will be extensively
used in Section~\ref{sec:normalforms}.
Let $\phi$ be a quantifier-free 
formula over $({\mathbb Q}^2;=_1,<_1,=_2,<_2)$. 
We may assume 
that
\begin{itemize}
\item[R1.] $\phi$ is in reduced CNF (as defined in Section~\ref{sec:determinedclauses}), 

\item[R2.] $\phi$ does not contain literals of the form $x \neq y$, because such literals can be replaced by $x <_1 y \vee y <_1 x \vee x <_2 y \vee y <_2 x$, and

\item[R3.] $\phi$ does not contain  literals of the form $x = y$, because such literals can be replaced by $x =_1 y \wedge x =_2 y$.

\item[R4.] $\phi$ does not contain literals of the form $\neg (x <_i y)$, for $i \in \{1,2\}$, because such literals can be replaced by $y <_i x \vee y =_i x$.
\end{itemize}

We introduce two rewriting rules R5 and R6, each of which yields a formula equivalent to the original formula $\phi$. The basis of both rules is the same and the only difference is in the relation contained in one of the affected literals.
Suppose that $\phi$ contains, for distinct $i,j \in \{1,2\}$, a clause $\alpha$ of the form $(u \circ_i v \vee x <_j y \vee \beta)$ where $u, v, x, y$ are (not necessarily distinct) variables, $\circ_i \in \{<_i,=_i,\neq_i\}$
and let $\phi'$ be the other clauses of $\phi$. If
\begin{align*} 
\phi' \wedge \neg \beta \wedge u \circ_i v & \text{ implies } x =_j y, 
\end{align*}
then we replace $\alpha$ by the two clauses 
\begin{align*}
& (x <_j y \vee x =_j y \vee \beta) \\
\text{ and } \; &  (u \circ_i v \vee x \neq_j y \vee \beta).
\end{align*}
If the relation $\circ_i$ is $<_i$, then we will refer to the rewriting rule as R5, and otherwise (that is, when  $\circ_i \in \{=_i, \neq_i\}$) we will refer to the rule as R6.
To see that the new formula is equivalent to $\phi$, let $s$ be a solution to $\phi$. If $s$ satisfies
$\beta$, then the two new clauses are satisfied. 
If $s$ does not satisfy $\beta$, then it must satisfy $u \circ_i v$ or $x <_j y$. In the first case, $s$ satisfies the second  new clause, and by assumption it also  satisfies the first new clause. In the latter case, it clearly satisfies both the first and the second clause. 
Now suppose that conversely, $s$ satisfies $\phi'$ and the two new clauses. If $s$ satisfies $\beta$
or $x <_j y$ then $\alpha$ is satisfied. Otherwise,
the first new clause implies that $x =_j y$,
and hence the second clause implies that $u \circ_i v$, and hence  $\alpha$ is satisfied.  
If the formula obtained from applying R5 or R6 is not reduced, we remove literals to make it reduced.
Note that after every application of R5 or R6 the conditions R1--R4 will still be satisfied.

\begin{ex}
We give an example of the rewriting procedure starting from a formula $\phi$ in reduced CNF.
\begin{align*}
&(x=_1 y) \wedge (u =_1 v \vee x=y) \wedge (u <_1 v \vee x <_2 y) \\
&(x=_1 y) \wedge (u =_1 v \vee (x =_1 y \wedge x =_2 y)) \wedge (u <_1 v \vee x <_2 y) &&\text{ (R3)} \\
&(x=_1 y) \wedge (u =_1 v \vee x =_1 y) \wedge (u=_1 v \vee x =_2 y) \wedge (u <_1 v \vee x <_2 y) &&\text{ (R1)} \\
&(x=_1 y) \wedge (u=_1 v \vee x =_2 y) \wedge (u <_1 v \vee x <_2 y) &&\text{ (R1)} \\
&(x=_1 y) \wedge (u=_1 v \vee x =_2 y) \wedge (x <_2 y \vee x=_2 y) \wedge (u <_1 v \vee x \neq_2 y) &&\text{ (R5)} \\
&(x=_1 y) \wedge (u=_1 v\vee x =_2 y) \wedge (x <_2 y \vee x=_2 y) \wedge (u <_1 v \vee u =_1 v) \wedge (x \neq_2 y \vee u \neq_1 v) &&\text{ (R6)}
\end{align*}
The resulting formula does not admit application of any of the rewriting rules.
\end{ex}

The reason to split the rewriting rule into two rules R5 and R6 is that only R5 terminates (i.e. can be applied only finitely many times) on every quantifier-free CNF formula (Lemma \ref{lem:R5-term}).  To prove termination of R6 (Lemma \ref{lem:R6-term}) and existence of an equivalent formula to which none of the rewriting rules above can be applied, we require the existence of a certain operation that preserves the formula. On the way we also prove a syntactic restriction on such formulas (Lemma \ref{lem:ineq-mix}). Note that Lemma \ref{lem:ineq-mix} and \ref{lem:R6-term} remain true also if $\theta_1(f)$ is the dual of an $\lele$- or $\pp$-operation; it can be proved using the version of Corollary \ref{cor:dom} based on duals.

\begin{lemma}\label{lem:R5-term}
Let $\phi$ be a quantifier-free CNF formula over $({\mathbb Q}^2;=_1,<_1,=_2,<_2)$. Then the rewriting rule R5 applied on $\phi$ terminates.

\end{lemma} 

\begin{proof}
Let $\alpha$ be an arbitrary clause over $({\mathbb Q}^2;=_1,<_1,=_2,<_2)$. 
Let $m(\alpha)$ denote the total number of $<_1$- and $<_2$-literals in $\alpha$.
Assume we apply R5  to the clause
$\alpha=(u <_i v \vee x <_j y \vee \beta)$ where we assume (without loss of generality) that
$u <_i v \not\in \beta$ and $x <_j y \not\in \beta$.
This yields two clauses $\alpha_1=(x <_j y \vee x =_j y \vee \beta)$ and $\alpha_2=(u <_i v \vee x \neq_j y \vee \beta)$. Note that $m(\alpha_1)<m(\alpha)$ and $m(\alpha_2)<m(\alpha)$ and reducing the formula cannot increase $m(\alpha)$ for any clause $\alpha$.

Now consider a quantifier-free CNF formula $\phi$ over $({\mathbb Q}^2;=_1,<_1,=_2,<_2)$. 
Let $\ell$ be the maximum clause length of $\phi$ and let $k$
be the number of variables appearing in $\phi$. Note that R5-rewriting cannot increase $\ell$ or $k$, and
for any clause $\alpha$ in $\phi$, it holds that $0 \leq m(\alpha) \leq \ell$. Let $f(\ell,k)$ denote the (finite) number of possible clauses over $({\mathbb Q}^2;=_1,<_1,=_2,<_2)$
where clause length is bounded by $\ell$ and at most $k$ variables are used.
Arbitrarily choose a clause $\alpha$ in $\phi$. If R5 is applied to $\alpha$,
then we know that $\alpha$ is replaced by at most two new clauses $\alpha_1$ and $\alpha_2$
where $m(\alpha_1) < m(\alpha)$ and $m(\alpha_2) < m(\alpha)$. Thus, the clause $\alpha$
can result in at most $2^{\ell}$ applications of rule R5.
Since there are at most $f(\ell,k)$ clauses in $\phi$, we conclude that
R5 can be applied at most $f(\ell,k) \cdot 2^{\ell}$ times to $\phi$.
\end{proof}

We continue with rewriting rule R6.
The following notation will be practical in several proofs dealing with syntactic forms. Let $g$ be a binary operation on $\Q$, $\phi$ a formula over $({\mathbb Q}^2;=_1,<_1,=_2,<_2)$ and set of variables $X$ and $s,t: X \rightarrow \Q$ assignments of $\phi$. Then $g(s,t)$ represents the assignment of $\phi$ that assigns to variable $x \in X$ the value $g(s(x), t(x))$. 

\begin{lemma}\label{lem:ineq-mix}
Let $\phi$ be a quantifier-free CNF formula over $({\mathbb Q}^2;=_1,<_1,=_2,<_2)$ such that R5 cannot be applied to $\phi$. Suppose that $\phi$ is preserved by an operation $f\in \Pol({\mathbb Q}^2;=_1,<_1,=_2,<_2)$ such that $\theta_{1}(f)$ is an $\lele$-operation or
a $\pp$-operation.
Then $\phi$ does not contain a clause that contains a $<_j$-literal for both $j\in\{1,2\}$.
\end{lemma}

\begin{proof}
By Corollary \ref{cor:dom}, there is $g\in \Pol({\mathbb Q}^2;=_1,<_1,=_2,<_2,\leq_1, \neq_1, \leq_2, \neq_2)$ that preserves $\phi$ such that $\theta_j(g)$ is dominated by $j$-th argument for both $j$.

Suppose for contradiction that $\phi$ contains a clause $\alpha$ of the form $(u <_1 v \vee x <_2 y \vee \beta)$. Since R5 cannot be applied and $\phi$ is reduced, there are satisfying assignments $s$ and $t$ of $\phi$ such that $s$ satisfies $u <_1 v$, $y <_2 x$, and falsifies $\beta$ and $t$ satisfies $v <_1 u$, $x <_2 y$, and falsifies $\beta$.  Then the tuple $g(t,s)$ 
satisfies $v <_1 u$ since $\theta_1(g)$ is
dominated by the first argument, 
and it satisfies $y <_2 x$ since
$\theta_2(g)$ is dominated by the second argument. Moreover, all other literals of $\alpha$ 
are falsified, too, 
since $g$ preserves $<_j$, $=_j$, $\leq_j$, and $\neq_j$ for $j \in \{1,2\}$. 
\end{proof}

\begin{lemma}\label{lem:R6-term}
Let $\phi$ be a quantifier-free CNF formula over $({\mathbb Q}^2;=_1,<_1,=_2,<_2)$ such that R5 cannot be applied on $\phi$. Suppose that $\phi$ is preserved by an operation $f\in \Pol({\mathbb Q}^2;=_1,<_1,=_2,<_2)$ such that $\theta_{1}(f)$ is an $\lele$-operation or a $\pp$-operation.
Then the rewriting rule R6 applied to $\phi$ terminates.
\end{lemma}

\begin{proof}
By Lemma \ref{lem:ineq-mix}, $\phi$ does not contain a clause that contains a $<_j$-literal for both $j\in\{1,2\}$. Note that by an application of the rewriting rule R6 no such clause can be created. For a clause $\alpha$ in $\phi$, let $p(\alpha)$ be the number of pairs of a $\{=_i, \neq_i\}$-literal and a $<_j$-literal, for distinct $i$ and $j$, that appears in $\alpha$.

Arbitrarily choose a clause $\alpha$ in $\phi$ that admits an application of R6. Then $\alpha=(u \circ_i v \vee x <_j y \vee \beta)$ for $\circ_i \in \{=_i, \neq_i\}$. After applying R6, $\alpha$ is replaced by two clauses $\alpha_1=(x <_j y \vee x =_j y \vee \beta)$ and $\alpha_2=(u \circ_i v \vee x \neq_j y \vee \beta)$. Observe that $p(\alpha_1)<p(\alpha)$  and $p(\alpha_2)<p(\alpha)$ since $\alpha$ does not contain $<_{3-j}$-literals. Moreover, reducing the formula does not increase $p(\gamma)$ for any clause $\gamma$ in the formula.

Let $\ell$ be the maximum clause length of $\phi$. For any clause $\gamma$ in $\phi$, it holds that $0 \leq p(\gamma) \leq \ell^2/4$. Using an argument analogous to the one in Lemma \ref{lem:R5-term}, we conclude that R6 can be applied only finitely many times to $\phi$.
\end{proof}

\noindent
If $\phi$ is a reduced formula such that none of the rewriting rules presented above
are applicable, then we call it \emph{normal}. If $\phi$ is a quantifier-free CNF formula over $({\mathbb Q}^2;=_1,<_1,=_2,<_2)$ preserved by  $f\in \Pol({\mathbb Q}^2;=_1,<_1,=_2,<_2)$ where $\theta_{1}(f)$ is an $\lele$-operation, a $\pp$-operation, or the dual of such an operation, it is possible to rewrite it to an equivalent normal formula by first applying R5 until it terminates and then applying R6 until it terminates. Observe that the resulting formula satisfies conditions R1--R4 and does not admit an application of R5 (since it does not contain a clause of the required form by Lemma \ref{lem:ineq-mix}) or R6.
Note that if $\phi$ is normal and contains a clause of the form $(u \circ_i v \; \vee \; x <_j y \; \vee \; \beta)$,
then $\phi$ has a satisfying assignment 
that satisfies $u \circ_i v$,
falsifies $\beta$, and satisfies $y <_j x$, because otherwise we could have applied R5 or R6.

The rewriting rules and Lemma~\ref{lem:R5-term} can be generalised in a straightforward fashion to formulas over $(\Q; <)^{(n)}$. To prove a generalised version of Lemma~\ref{lem:ineq-mix}, one needs to use Corollary~\ref{cor:dom} to produce a polymorphism with the particular domination property in two distinct dimensions $i$ and $j$ (see the proof of Proposition \ref{prop:syntax-n} for more details). The generalised version of the lemma then shows that there is no clause containing both $<_i$-literals and $<_j$-literals for distinct $i$ and $j$ under the assumption that for all but at most one $p \in \{1,\dots,n\}$ there is $f_p \in \Pol(\bD)$ such that $\theta_p(f_p)$ is an $\lele$-operation or a $\pp$-operation. Subsequently, a generalised version of Lemma~\ref{lem:R6-term} can be proved and normal formulas can be defined. The normalisation process for formulas over $(\Q; <)^{(n)}$ will be used in Section \ref{sect:n-dim}.

\subsubsection{Definitions via Restricted Clauses}
\label{sec:normalforms}

The following definition is central in our presentation of various syntactic normal forms.

\begin{definition}\label{def:weakly-determined}
A clause is \emph{weakly $1$-determined} if it is of the form 
$$\psi \vee \bigvee_{i \in \{1,\dots, k\}} x_i \neq_2 y_i$$ where $\psi$ is  $1$-determined and $k \geq 0$. Weakly $2$-determined clauses are defined analogously.
\end{definition}
A clause can simultaneously be weakly $1$-determined
\emph{and} weakly $2$-determined: $x \neq_1 y \vee u \neq_2 v$ is one example.
Normalised formulas in the sense of Section~\ref{sec:normalisation} play a key role in our first result
concerning logical definitions based on weakly $i$-determined clauses.

\begin{proposition}\label{prop:decomp}
Suppose that $\Pol(\bD)$ contains an operation $f$ such that $\theta_1(f)$ is an $\lele$-operation
or a $\pp$-operation.
Then, the following holds for every relation $R$ in $\bD$: if a normal formula $\phi$ is a definition of $R$ over $(\Q^2,<_1, =_1,<_2, =_2)$, then $\phi$ is a conjunction of clauses
each of which is weakly $i$-determined for some $i \in \{1,2\}$.
\end{proposition}
\begin{proof}
By Corollary~\ref{cor:dom}, there is an operation $g \in \Pol(\bD; \leq_1, \neq_1, \leq_2, \neq_2)$ such that $\theta_i(g)$ is dominated by the $i$-th argument for both $i \in \{1,2\}$. Let $\phi$ be a normal formula that defines
a relation $R \in \bD$ over $({\mathbb Q}^2;<_1,=_1,<_2,=_2)$.
Note, in particular, that $\phi$ cannot be rewritten using rule R5 or R6.
Let  
$\psi$ be a clause of $\phi$. 
Properties R1--R4 imply that $\psi$ is equal to
\[\psi_1 \vee \psi_2 \vee \bigvee_{l \in \{1,\dots, k\} \; {\rm and} \; j \in \{1,2\} } x_l \neq_j y_l\] 
where
$\psi_i$ for $i \in \{1,2\}$ only contains literals of the form  $x =_i y$ or $x <_i y$. 

We show the result by verifying that
$\psi_1$ or $\psi_2$ is the empty disjunction.
Suppose to the contrary that $\psi_1$ contains a literal $\ell_1$ and $\psi_2$ contains a literal $\ell_2$. 
Since $\phi$ is reduced, it must have a satisfying assignment $s$ that satisfies $\ell_1$ and 
falsifies all other literals of $\psi$,
and there also exists 
a satisfying assignment $t$ that satisfies $\ell_2$ and falsifies all other literals of $\psi$. Note that by Lemma \ref{lem:ineq-mix} it cannot occur that $\ell_1$ equals $u <_1 v$ and $\ell_2$ equals $x <_2 y$, since R5 cannot be applied to $\phi$. Therefore, 
we have to consider the following cases. 

\begin{enumerate}
\item Suppose that the literal $\ell_1$  is of the form $u =_1 v$
and the literal $\ell_2$ equals $x <_2 y$. 
Since $\phi$ is normal, we may assume (by R6) that $s$ satisfies $y <_2 x$. 
Then $g(t,s)$ satisfies $u \neq_1 v$ since $\theta_1(g)$ is dominated by the first argument, and it satisfies $y <_2 x$ because $\theta_2(g)$ is dominated by the second argument. Moreover, all other literals of $\psi$ are falsified. 
Hence, $g(t,s)$ does not satisfy $\phi$, which contradicts $g \in \Pol(\bD)$.
\item The case that the literal $\ell_1$ equals $u<_1 v$ and the literal $\ell_2$ equals $x =_2 y$, and the case that $\ell_1$ equals $u =_1 v$ and the literal $\ell_2$ equals $x =_2 y$ can be treated similarly.
\end{enumerate}
If $\psi_1$ is empty, then
we obtain a clause that is weakly $2$-determined. 
Likewise, if $\psi_2$ is empty, then
$\psi$ is a  weakly $1$-determined clause.
\end{proof}

Under additional conditions on polymorphisms, we can define relations by
formulas that are based on weakly $1$-determined clauses
together with (not weakly) $2$-determined clauses.

\begin{proposition}\label{prop:syntax-2}
Let $f \in \Pol(\bD)$ be 
such that $\theta_1(f)$ is a $\pp$-operation.
Then every normal conjunction $\phi$ of weakly 1-determined and weakly 2-determined clauses that is preserved by $f$
is a conjunction of weakly $1$-determined and of $2$-determined clauses. 
\end{proposition}
\begin{proof}
By Corollary~\ref{cor:dom} there is an operation $g \in \Pol(\bD; \leq_1, \neq_1, \leq_2, \neq_2)$ such that $\theta_1(g)=\pi_1^2$  and $\theta_2(g)$ is dominated by the second argument.
Suppose for contradiction that $\phi$ 
contains a clause $\psi$ with a literal $x \neq_1 y$ and a literal $\chi$ 
which is of the form $u <_2 v$ or $u =_2 v$. The formula $\phi$ is reduced so
it has a satisfying assignment $s$ which satisfies $x \neq_1 y$ and falsifies all other literals of $\psi$, 
and a satisfying assignment $t$ which satisfies
$\chi$ and falsifies all other literals in $\psi$;
in particular, $t$ satisfies $x =_1 y$. 

We first consider the case that $\chi$ is of the form $u <_2 v$ and $s$ consequently satisfies $v \leq_2 u$. 
Since $\phi$ is normal we may even suppose
that $s$ satisfies $v <_2 u$. 
Then $g(t,s)$ satisfies 
$x =_1 y$ since $\theta_1(g) = \pi^2_1$ and it satisfies $v <_2 u$ since $\theta_2(g)$ is dominated by the second argument.
Hence, it satisfies neither the literal $x \neq_1 y$ nor the literal $u<_2 v$, nor
 any of the other literals of $\psi$ since $f$ preserves $\leq_i$
and $\neq_i$ for $i \in \{1,2\}$. This is in contradiction to the assumption that $\phi$ is preserved by $g$. 

The case that $\chi$ is of the form $u =_2 v$ similarly leads to a contradiction. 
\end{proof}

\begin{rem}
Note that it is not true that every weakly $2$-determined clause of formula $\phi$ in Proposition~\ref{prop:syntax-2} is $2$-determined: For example, consider the clause $\psi$ of the form $x\neq_1 y \vee u\neq_2 v$. The clause $\psi$ is 
weakly $2$-determined, but not $2$-determined, and it satisfies the assumptions of Proposition~\ref{prop:syntax-2}: it is normal, and preserved by a map $(f_1, f_2)$ where $f_1$ is a $\pp$-operation and $f_2$ is an $\lele$-operation. To see this, let $s,t$ be two satisfying assignments of $\psi$. Either one of $s$ and $t$ satisfies $u \neq_2 v$ and hence $(f_1, f_2)(s,t)$ satisfies it as well by the injectivity of $f_2$, or both $s$ and $t$ satisfy $x \neq_1 y$ and hence $(f_1, f_2)(s,t)$ satisfies it as well since $f_1$ preserves $\neq$.
\end{rem}

In the next proof we use the notion of the orbit of a $k$-tuple $(t_1,\dots, t_k)\in \Q^k$ under $\Aut(\Q; <)$ from Section~\ref{sect:cat}. 
Observe that the homogeneity of
$({\mathbb Q}; <)$ implies that the orbit of a tuple $(t_1, \dots , t_k)$ under $\Aut({\mathbb Q}; <)$ is determined by the weak linear order induced on $(t_1, \dots , t_k)$ in $({\mathbb Q}; <)$. We need a weak linear order since some of the elements $t_1, \dots , t_k$ may be equal. 

\begin{proposition}\label{prop:syntax-3}
As in the previous proposition, let $f \in \Pol(\bD)$ be 
such that $\theta_1(f)$ is a $\pp$-operation.
Then every relation of $\bD$ can be defined by a conjunction of $2$-determined clauses
and 
weakly $1$-determined clauses of the form 
\begin{align}
u_1 \neq_2 v_1 \vee \cdots \vee u_m \neq_2 v_m \vee y_1 \neq_1 x \vee \cdots \vee y_k \neq_1 x \vee z_1 \leq_1 x \vee \cdots \vee  z_l \leq_1 x.
\label{eq:1-det-pp}
\end{align}
(In other words, if we drop the first $m$ literals in and remove subscripts we obtain a formula as described in Theorem~\ref{thm:pp}.)
\end{proposition}

\begin{proof}
Let $\phi$ be a formula over a finite set of variables $X$ that defines a relation $R$ from $\bD$. Without loss of generality, we may assume that $\phi$ is normal. Hence, by Proposition~\ref{prop:decomp} and  Proposition~\ref{prop:syntax-2}, $\phi$ is a conjunction of weakly 1-determined and of 2-determined clauses.  
Let $\phi'$ be the conjunction of all clauses of the form~\eqref{eq:1-det-pp} and of all $2$-determined clauses with variables from $X$ that are reduced and implied by $\phi$; since $|X|$ is finite, $\phi'$ is a finite formula.
We claim that $\phi'$ implies $\phi$, and consequently that $\phi'$ is a definition of $R$ of the required syntactic form.

The $2$-determined conjuncts of $\phi$ are clearly implied by $\phi'$. In the rest of the proof, we prove in two steps that every weakly $1$-determined conjunct of $\phi$ is implied by $\phi'$: Firstly, we show that we may assume that all weakly $1$-determined clauses of $\phi$ are of the form (\ref{eq:orbit}) and that they are minimal in a particular sense specified below. Secondly, we show that every such clause is implied by $\phi'$, because the assumption that $R$ is preserved by $f$ would be violated otherwise.

To proceed with the first step, we claim that every conjunction of weakly 1-determined clauses can be written as a conjunction of formulas of the form
$$\chi \vee \neg (y_1 \circ_1 y_2 \wedge \cdots \wedge y_{k} \circ_k y_{k+1})$$ 
where $\chi := \bigvee_{i=1}^m u_i \neq_2 v_i$ with $u_1, \dots, u_m$, $v_1, \dots, v_m$, $y_1, \dots, y_{k+1}$ being variables from $X$, and 
 $\circ_1,\dots,\circ_k \in \{=_1,<_1\}$. To see this, first note
 that every orbit of $k+1$-tuples in $\Aut({\mathbb Q};<)$ can be defined by a formula of the form
 $x_1 \circ_1 x_2 \wedge \cdots \wedge x_k \circ_k x_{k+1}$ where $\circ_1,\dots,\circ_k \in \{=,<\}$ if the variables are named appropriately. 
 Hence, every first-order formula in $(\Q;<)$ is equivalent to a conjunction of negations of such formulas. 
 It follows that every 1-determined clause can be written as a conjunction of formulas of the form $\neg (y_1 \circ_1 y_2 \wedge \cdots \wedge y_{k} \circ_k y_{k+1})$ for $\circ_1,\dots,\circ_k \in \{=_1,<_1\}$. Using distributivity of disjunction over conjunction we can therefore rewrite a conjunction of weakly 1-determined clauses into a conjunction of formulas of the desired form. 
 
We may henceforth assume that 
 every conjunct $\psi$
of $\phi$ that is not 
2-determined is of the form 
\begin{equation}\label{eq:orbit}
    \chi \vee y_1 \circ_1 y_2 \vee \cdots \vee y_k \circ_k y_{k+1}
\end{equation}
where $\chi$ is as above and $\circ_1,\dots,\circ_k \in \{\neq_1,\geq_1\}$.
We may additionally assume that $\phi$ contains only those clauses of the form (\ref{eq:orbit}) that are minimal in the following sense: a clause $\psi = \chi \vee y_1 \circ_1 y_2 \vee \cdots \vee y_k \circ_k y_{k+1}$ is minimal if there is no clause  $\chi \vee y_1' \circ_1 y_2' \vee \cdots \vee y_{k'}'\circ_{k'} y_{k'+1}'$ implied by $\phi$ such that it implies $\psi$, it is different from $\psi$ and $y_1',\dots, y_{k'}'$ is a subsequence of $y_1, \dots, y_k$.

We now show that the clause $\psi$ is implied by $\phi'$. If $\circ_i$ equals $\neq_1$ for every $i \in \{1,\dots,k\}$ then
$\psi$ is equivalent to $\chi \vee \neg (y_1 =_1 y_{k+1}  \wedge \cdots \wedge y_{k} =_1 y_{k+1})$
and hence is equivalent to $\chi \vee y_1 \neq_1 y_{k+1} \vee \cdots \vee y_k \neq_1 y_{k+1}$ which is of the form~\eqref{eq:1-det-pp} (for $l = 0$). But this formula is then a conjunct of $\phi'$ and there is nothing to be shown. 

Otherwise, let $j$ be smallest such that $\circ_j$ equals $\geq_1$. Then $\psi$ is equivalent to a formula of the form
$$\chi \vee y_1 \neq_1 y_j \vee \cdots \vee y_{j-1} \neq_1 y_j \vee y_{j} \geq_1 y_{j+1} \vee \eta$$
where $\eta$ is of the form $y_{j+1} \circ_{j+1} y_{j+2} \vee \cdots \vee y_k \circ_k y_{k+1}$
for $\circ_{j+1},\dots,\circ_k \in \{\neq_1,\geq_1\}$. 
The formula
\begin{align}
\chi \vee y_1 \neq_1 y_j \vee \cdots \vee y_{j-1} \neq_1 y_j \vee y_{j+1} \leq_1 y_j \vee \cdots \vee y_{k+1} \leq_1 y_j
    \label{eq:cut}
\end{align}
implies $\psi$. To see this, compare the negations of the formulas: $\neg \psi$ is equivalent to the conjunction of $\neg \chi \wedge y_1 =_1 \dots =_1 y_{j} \wedge y_j <_1 y_{j+1} $ and the fact that $y_{j+1}, \dots, y_k$ is a non-decreasing sequence, while the negation of the formula (\ref{eq:cut}) is equivalent to \[\neg \chi \wedge y_1 =_1 \dots =_1 y_{j} \wedge y_{j+1} >_1 y_j \wedge \dots \wedge y_k >_1 y_j.\] Therefore, if the formula (\ref{eq:cut}) is a conjunct of $\phi'$, then there is again nothing to be shown. 

Otherwise, there must exist an assignment $r$ that satisfies $\phi$ but not~\eqref{eq:cut}. Note that
$r(y_j) <_1 r(y_i)$ for every $i \in \{j+1,\dots,k+1\}$. 
Since $\psi$ was minimal, the clause $\chi \vee \eta$ is not implied by $\phi$ and hence there must also exist an assignment $s$ which satisfies $\phi$, but does not satisfy $\chi \vee \eta$.
Choose $\alpha \in \Aut(\bD)$ such that 
$\alpha(r(y_j)) <_1 0 <_1 \alpha(r(y_i))$ for all 
$i \in \{j+1,\dots,k+1\}$. 
Then $t := f(\alpha(r),s)$ is an assignment that does not satisfy $\psi$ by the definition of a $\pp$-operation.
This contradicts that $\phi$ defines a relation from $\bD$. 
\end{proof}

As the following lemma shows, when $\theta_1(\Pol(\bD))$ contains a $\pp$-operation and $\theta_2(\Pol(\bD))$ contains an $\lele$-operation, we may restrict to formulas of a very particular form.

\begin{proposition}\label{prop:syntax-4}
Let $f_1, f_2 \in \Pol(\bD)$ be such that $\theta_1(f_1)$ is a $\pp$-operation and $\theta_2(f_2)$ is an $\lele$-operation.
Then every relation of $\bD$ can be defined by a conjunction of 
weakly $1$-determined clauses of the form (\ref{eq:1-det-pp}) and $2$-determined clauses
\begin{align}\label{eq:ll}
x_1 \neq_2 y_1 \vee \cdots \vee x_m \neq_2 y_m \vee z_1 <_2 z_0 \vee \cdots \vee z_{\ell} <_2 z_0 \vee (z_0 =_2 z_1 =_2 \cdots =_2 z_{\ell}),
\end{align}
where it is permitted that $l = 0$ or $m=0$, and
the final disjunct may be omitted (in other words, clauses obtained from ll-Horn clauses by adding the subscript $2$ to all relation symbols).
\end{proposition}

\begin{proof}
Let $R$ be a relation of $\bD$. By Proposition~\ref{prop:syntax-3}, $R$ can be defined by a formula $\phi_1 \wedge \phi_2$, where $\phi_1$ is a conjunction of weakly $1$-determined clauses of the form (\ref{eq:1-det-pp}) and $\phi_2$ is a conjunction of $2$-determined clauses. Recall the operator $\Cr$ introduced at the end of Section~\ref{sec:determinedclauses}. Let us denote $\Cr(\phi_1 \wedge \phi_2,\{2\},\phi_2)$ by $\psi_2$; note that $\psi_2$ is a conjunction of $2$-determined clauses preserved by $\Pol(\bD)$. Moreover the formula $\phi_1 \wedge \psi_2$ still defines $R$. By Lemma~\ref{lem:pres}, $\hat{\psi_2}$ is preserved by an $\lele$-operation. By Theorem~\ref{thm:lele}, $\hat{\psi_2}$ may be taken to be a conjunction of $\lele$-Horn clauses and therefore $\psi_2$ would be a conjunction of clauses of the form (\ref{eq:ll}). This concludes the proof.
\end{proof}

We next present an even more restricted syntactic form; in this case
it is sufficient to use $i$-determined clauses, $i \in \{1,2\}$, and we do
not need weakly $i$-determined clauses at all.

\begin{proposition}\label{prop:factor}
Suppose that there exists 
$f \in \Pol(\bD)$
such that $\theta_1(f)$ is a $\pp$-operation,
but there is 
no binary $g \in \Pol(\bD)$
such that $\theta_2(g)$ is a $\lex$-operation. 
Then every relation of $\bD$ can be defined by a formula $\phi_1 \wedge \phi_2$ such that $\phi_i$ is a conjunction of $i$-determined clauses for $i=1,2$.
Moreover, for every such definition $\phi_1\wedge\phi_2$, there is a conjunction $\psi_1$ of $1$-determined clauses of the form
\[y_1 \neq_1 x \vee \cdots \vee y_k \neq_1 x \vee z_1 \leq_1 x \vee \cdots \vee z_l \leq_1 x\]
such that $\psi_1 \wedge \phi_2$ still defines the same relation.
\end{proposition}

\begin{proof}
Lemma~\ref{lem:canonise} implies that $\Pol(\bD)$ contains an operation $f'$ such that $\theta_1(f')$ is a $\pp$-operation and $\theta_2(f')$ is either a $\lex$-operation or it is essentially unary;
by assumption, it cannot be a $\lex$-operation so it must be essentially unary. 
Let $i\in \{1,2\}$ be such that $\theta_2(f')$ depends only on the $i$-th argument. Since $\theta_2(f')$ preserves $<$, 
the set $\overline{\{(\id_\Q, \alpha)f' \mid \alpha \in \Aut(\Q; <)\}} \subseteq \Pol(\bD)$ contains an operation $f''$ such that $\theta_1(f'')=\theta_1(f')$ and $\theta_2(f'')=\pi_i^2$.
Therefore, we can assume without loss of generality that $\theta_2(f')=\pi_i^2$.
By Lemma \ref{lem:dom} applied on $f'$, we obtain $g \in \Pol(\bD)$ such that $\theta_1(g)=\pi_{3-i}^2$ and $\theta_2(g)=\pi_i^2$. This implies that $(\pi_1^2, \pi_2^2)\in \Pol(\bD)$. By Lemma~\ref{lem:factors}, every relation of $\bD$ can be defined by a conjunction of $1$-determined and $2$-determined clauses.

Let $R$ be a relation of $\bD$. Let $\phi_1 \wedge \phi_2$ be a definition of $R$ such that $\phi_i$ is a conjunction of $i$-determined clauses, $i=1,2$. Recall the operator $\Cr(\cdot)$ introduced at the end of Section~\ref{sec:determinedclauses}. Let $\psi_1=\Cr(\phi_1\wedge \phi_2, \{1\}, \phi_1)$. Then the formula $\psi_1$ is a conjunction of $1$-determined clauses preserved by $\Pol(\bD)$ and $\psi_1\wedge \phi_2$ defines $R$. By Lemma~\ref{lem:pres}, $\hat{\psi_1}$ is preserved by a $\pp$-operation. Hence, by Theorem~\ref{thm:pp}, we may assume that $\psi_1$ is a conjunction of clauses of the form 
 $$ y_1 \neq_1 x \vee \cdots \vee y_k \neq_1 x \vee z_1 \leq_1 x \vee \cdots \vee z_l \leq_1 x. $$ This concludes the proof.
\end{proof}

Assume that for every $i \in \{1,2\}$, the clone $\Pol(\bD)$ contains an operation $f_i$ such that $\theta_i(f_i)$ is an $\lele$-operation or the dual of such an operation. Then, 
we may combine the information about syntactically restricted definitions of the relations of $\bD$ from Proposition~\ref{prop:decomp} with $\lele$-Horn definability from Theorem~\ref{thm:lele}, and obtain the next result (Proposition~\ref{prop:pp-pow}). We will use the following notation
for simplifying the presentation. For a formula $\phi$ over $(\Q^n;<_1,=_1, \dots, <_n,=_n)$, we introduce $n$ fresh variables $x^1, \dots, x^n$ for each variable $x$ that appears in $\phi$. Then, we let $\ve(\phi)$ (for {\em variable expansion}) denote the formula over $(\Q;<)$ resulting from $\phi$ by replacing each atomic formula of the form $x \circ_i y$ by $x^i \circ y^i$, where $\circ \in \{<,\leq, =, \neq\}$ and $i\in\{1,\dots,n\}$.

\begin{proposition}\label{prop:pp-pow}
Suppose that for every $i \in \{1,2\}$ the clone $\Pol(\bD)$ contains
an operation $f_i$ such that $\theta_i(f_i)$ is an $\lele$-operation.
Then every relation of $\bD$ has a definition
by a conjunction of clauses of the form
$$x_1 \neq_{i_1} y_1 \vee \cdots \vee x_m \neq_{i_m} y_m \vee z_1 <_j z_0 \vee \cdots \vee z_{\ell} <_j z_0 \vee (z_0 =_j z_1 =_j \cdots =_j z_{\ell})$$
for $i_1,\dots,i_m,j \in \{1,2\}$ and 
where the last disjunct may be omitted. 
Moreover, $\bD$ has a primitive positive interpretation in $\bL$. 
\end{proposition}
\begin{proof}
Let $R$ be a relation of $\bD$, and let 
$\phi$ be a definition of $R$. Without loss of generality, we may assume that $\phi$ is normal and hence a conjunction of weakly $1$-determined and weakly $2$-determined clauses (Proposition~\ref{prop:decomp}).

\medskip 

\noindent
{\bf Claim.} The
formula $\ve(\phi)$ over $({\mathbb Q};<)$ is preserved by every $\lele$ operation. 

\smallskip

\noindent
To prove the claim, let $\lele$ be an $\lele$-operation and let $r'$ and $s'$ be satisfying assignments for $\ve(\phi)$. We have to show that
$t'(x) \coloneqq \lele(r'(x),s'(x))$ satisfies $\ve(\phi)$. Let $\psi'$ be a conjunct of $\ve(\phi)$. Then $\psi'$ has been created from a conjunct $\psi$ of $\phi$ with variables $y_1,\dots,y_m$, which must be weakly $i$-determined, for some $i \in \{1,2\}$. 
We assume henceforth that $i=1$; the other case can be shown analogously. 
Note that 
the maps $r \colon x \mapsto (r'(x^1),r'(x^2))$ and $s \colon x \mapsto (s'(x^1),s'(x^2))$ are satisfying assignments to $\phi$. Since $\theta_1(f_1)$ is an $\lele$-operation, we may choose $\alpha \in \Aut(\Q;<)$ such that
for $t(x) := (\alpha,\id) f_1\big ( r(x),s(x)\big )$
we have $(t(y_1)_1,\dots,t(y_m)_1) = (t'(y^1_1),\dots,t'(y^1_m))$ (see Remark~\ref{rem:easy}). 
Therefore, we are done if one of the disjuncts of
$\psi$ of the form $x =_1 y$, $x <_1 y$, or $x \neq_1 y$ is satisfied by $t$, because then a disjunct of 
$\psi'$ of the form $x^1 = y^1$, $x^1 < y^1$, or $x^1 \neq y^1$ is satisfied by $t'$. Otherwise, since $t$ satisfies $\psi$, there must be a literal of $\psi'$ of the form $x^2 \neq y^2$. We claim that $t'$ satisfies this literal. 
As $t$ satisfies $x \neq_2 y$, we must have
$r(x) \neq_2 r(y)$ or $s(x) \neq_2 s(y)$, so
$r'(x^2) \neq r'(y^2)$ or $s'(x^2) \neq s'(y^2)$. 
Hence, $t'(x^2) \neq t'(y^2)$ by the injectivity of $\lele$. 
$\hfill 
\diamond$

\medskip

\noindent
Note that the first statement of the proposition follows from the claim by Theorem~\ref{thm:lele}. 
The claim and Theorem~\ref{thm:lele} also imply that we obtain a two-dimensional primitive positive interpretation of $\bD$ in $\bL$, which proves
the second statement of the theorem.
\end{proof}

\subsection{Classification in the 2-Dimensional Case}
\label{sect:2D-class}

 The aim of this section is to present
our algebraic dichotomy result for polymorphism clones of
first-order expansions of $(\Q^2;<_1,=_1,<_2,=_2)$
together with a corresponding complexity dichotomy (Section~\ref{sect:class})
We begin by presenting algorithmic results.

\subsubsection{Polynomial-time Algorithms}
\label{sect:tract}
In this section we present polynomial-time solvability results for $\Csp(\bD)$ when $\bD$ is a first-order expansion of $({\mathbb Q}^2;<_1,=_1,<_2,=_2)$ such that $\theta_1(\Pol(\bD))$ and $\theta_2(\Pol(\bD))$ contain sufficiently strong polymorphisms. Intuitively speaking, in every tractable case, there must be in each of the dimensions a polymorphism that guarantees tractability. By Theorem~\ref{thm:tcsps}, this means that both $\theta_1(\Pol(\bD))$ and $\theta_2(\Pol(\bD))$ have to contain one of the operations $\min_3$, $\mx_3$, $\mi_3$ or $\lele_3$. The case when $\lele_3$, equivalently any $\lele$-operation, lies in $\theta_i(\Pol(\bD))$ for both $i$ is treated in Proposition~\ref{prop:tract}. The case when $\theta_1(\Pol(\bD))$ contains $\min_3$, $\mx_3$, $\mi_3$ and $\theta_2(\Pol(\bD))$ contains $\lele_3$ is covered in Proposition~\ref{prop:tract2}; the case with the dimensions switched is symmetric. The case when both $\theta_1(\Pol(\bD))$ and $\theta_2(\Pol(\bD))$ contain one of the operations  $\min_3$, $\mx_3$, or $\mi_3$ can be easily handled using Proposition~\ref{prop:factor} and Lemma~\ref{lem:factors} and is treated directly in the proof of the classification theorem (Theorem~\ref{thm:main}).

\begin{proposition}\label{prop:tract}
Suppose that $\bD$ has a finite relational signature and for every $i \in \{1,2\}$, the clone $\Pol(\bD)$ contains an operation $f_i$ such that $\theta_i(f_i)$ is an $\lele$-operation.
Then $\Csp(\bD)$ can be solved in polynomial time. 
\end{proposition}
\begin{proof}
The structure $\bD$ has a primitive positive interpretation in $\bL$ by Proposition~\ref{prop:pp-pow}. 
The result follows from 
Lemma~\ref{lem:pp-construct-reduce} since 
$\Csp(\bL)$ can be solved in polynomial time.
\end{proof}

We will use 
some additional observations and notations in the proof of the next proposition (and also in its generalisation Proposition \ref{prop:tract-n}). 
Assume that the relational structure $\bA$ with domain ${\mathbb Q}$ has a finite signature and
there is an $\lele$-operation in $\Pol(\bA)$. We know from Theorem~\ref{thm:tcsps} that CSP$(\bA)$ is polynomial-time
solvable.
Assume now that $\bA'$ is a solvable instance of $\Csp(\bA)$ with variable set $X$.
The polynomial-time algorithm for $\Csp(\bA)$ by Bodirsky and Kára~\cite[Section 4]{ll} computes additional information in the form of an 
{\em equality set}: a set $E \subseteq X^2$ such that for every $(x,x') \in E$ and every solution $s$, it holds that $s(x)=s(x')$.
Furthermore, there exists a solution $t$ such that $t(x) \neq t(x')$ for every $(x,x') \not\in E$.

With this in mind, we introduce
the following notation. Let $X$ be a set of variables,
$i \in \{1,2\}$, and
$E\subseteq \{(x^i,y^i)\mid x,y \in X\}$
where $x^i,y^i$ denote fresh variables.
Assume that $\phi$ is a formula over $(\Q^2,<_1,=_1, \dots, <_n,=_n)$ that only uses variables from $X$. If $\phi$ equals
$\phi_1 \wedge \phi_2$, where $\phi_i$ is a conjunction of $i$-determined clauses and $\phi_{3-i}$ is a conjunction of weakly $(3-i)$-determined clauses,
then we let $\cm(\phi, E)$ (for {\em clause modification}) denote the formula resulting from $\phi$ by performing the following procedure:
 \begin{itemize}
     \item If $(x^i,y^i)\in E$ and $\phi_{3-i}$ contains a clause with the literal $x \neq_i y$, then remove this literal from the clause and add the conjunct $x =_i y$ to $\phi$.
     \item If $(x^i,y^i)\not\in E$ and $\phi_{3-i}$ contains a clause with the literal $x \neq_i y$, then delete all literals in the clause but this one.
 \end{itemize}
For example, if $i=1$, $\phi = (x=_1 y \vee x <_1 z) \wedge (x \neq_1 z \vee u <_2 v) \wedge (u \neq_1 z \vee x =_2 y)$ and $E=\{(x^1, z^1)\}$, then $\cm(\phi,E) = (x=_1 y \vee x <_1 z) \wedge  (x =_1 z) \wedge (u <_2 v) \wedge (u \neq_1 z)$.
 
 The following proposition and its proof describe a polynomial time algorithm for $\Csp(\bD)$, where $\Pol(\bD)$ satisfies certain conditions (see Algorithm~\ref{alg:main}). A generalised version of this algorithm for first-order expansions of $(\Q, <)^{(n)}$ with $n\geq 2$ will be presented
 in Proposition~\ref{prop:tract-n} and Algorithm~\ref{alg:main2}.
 There are no profound differences between the algorithms but the case when $n=2$ is easier to present since
 we can keep the formal machinery at a minimum.

\begin{proposition} \label{prop:tract2}
Suppose that $\bD$ has a finite relational signature and that $\Pol(\bD)$ contains $f_1,f_2$ such that 
$\theta_1(f_1)$ equals 
$\min_3$, $\mx_3$, or $\mi_3$, 
and $\theta_2(f_2)$ is an $\lele$-operation.
Then $\Csp(\bD)$ can be solved in polynomial time. 
\end{proposition}
\begin{proof}
Apply Lemma \ref{lem:canoni} to the operation $f_1$ for dimension $i=2$. Then, there is an operation $f_1'\in \Pol(\bD)$ such that
$\theta_2(f_1')$ is canonical over $\Aut(\Q; <)$ and 
$m := \theta_1(f_1')$ equals $\min_3$, $\mx_3$, or $\mi_3$. By Lemma \ref{lem:lex}, $f_1'$ preserves $\neq_2$. Since $\theta_2(f_2)$ is an $\lele$-operation, $f_2$ preserves $\neq_2$ as well. Therefore we may assume without loss of generality that $\bD$ contains the  relation $\neq_2$.
Since $\theta_1(\Pol(\bD))$ is closed (by Proposition~\ref{prop:closed}),
it follows from Proposition~\ref{prop:pp}
that $\theta_1(\Pol(\bD))$ contains a $\pp$-operation. 

Let $\tau$ be the signature of $\bD$ and let $\bA$ be an instance of $\Csp(\bD)$. For every $R \in \tau$ of arity $k$ and  $\bar{a}=(a_1,\dots,a_k) \in R^{\bA}$, 
let $\phi_{R,\bar a}$ be the first-order definition of $R$ in $({\mathbb Q}^2;<_1,=_1,<_2,=_2)$, using the elements $a_1,\dots,a_k$ as free variables. 
Since $\theta_1(\Pol(\bD))$ contains a $\pp$-operation and $\theta_2(\Pol(\bD))$ contains an $\lele$-operation, we may assume that $\phi_{R,\bar{a}}$ has the form described in Proposition~\ref{prop:syntax-4}.
Let $\Phi = \{\phi_{R,\bar a} \; | \; R \in \tau \text{ and } \bar a \in R^{\bA}\}$.
The set $\Phi$ can be computed in polynomial time since $\bD$ has a finite signature.
It is clear that $\bA$ is a yes-instance of $\Csp(\bD)$ if and only if $\Phi$ is satisfiable.
We will now present a polynomial-time algorithm for checking the satisfiability of $\Phi$; it is outlined in Algorithm~\ref{alg:main}. The basic idea is to compute
two sets $\Psi_1$ and $\Psi_2$ of logical formulas that are simultaneously satisfiable if and only if $\Phi$
is satisfiable. The sets $\Psi_1$ and $\Psi_2$ are, in a sense that will be clarified below,
connected to formulas in $\Phi$ that contain weakly 1-determined and 2-determined clauses, respectively.

 Every $\phi \in \Phi$ is of the form $\phi_1 \wedge \psi_2$, where $\phi_1$ is a conjunction of weakly $1$-determined clauses and $\psi_2$ is a conjunction of clauses of the form
\[x_1 \neq_2 y_1 \vee \cdots \vee x_m \neq_2 y_m \vee z_1 <_2 z_0 \vee \cdots \vee z_{\ell} <_2 z_0 \vee (z_0 =_2 z_1 =_2 \cdots =_2 z_{\ell}).\]
Therefore $\hat{\psi_2}$ is a conjunction of $\lele$-Horn clauses and by Theorem~\ref{thm:lele} preserved by an $\lele$-operation.
We let $\Psi_2$ denote the set of all formulas $\hat \psi_2$ obtained in this way from the
 members of $\Phi$. Note that $\Psi_2$ can be computed in polynomial time since $\bD$ has a finite relational signature.
 
Since $\Psi_2$ is preserved by an $\lele$-operation, we can check its satisfiability in polynomial time
by Theorem~\ref{thm:lele} combined with Theorem~\ref{thm:tcsps}.
If $\Psi_2$ is not satisfiable, then we reject the input $\bA$. 
Otherwise, we let $E'$ denote the equality set of the instance.
 For each $(x,y) \in E'$, we put a pair $(x^2,y^2)$ in the set $E$; this set will later on be used as an argument to the clause modification operator $\cm(\cdot)$
 that was introduced in connection with
 Proposition~\ref{prop:tract2}.
 
 Every $\phi$ in $\Phi$ is equivalent to a formula $\phi_1 \wedge \psi_2$ as described above.
 Note that the formula $\phi_1 \wedge \psi_2$ admits application of the clause modification operator $\cm(\cdot)$. Hence, we define $\Phi'$ to be the set of the formulas  $\cm(\phi_1 \wedge \psi_2, E)$ obtained from formulas $\phi \in \Phi$.
 Note that each formula in $\Phi'$ still defines a relation that has a primitive positive definition in $\bD$ (since $\bD$ contains the relations $=_2$ and $\neq_2$). Further note that, up to renaming variables, only finitely many different formulas may appear in $\Phi'$: there are only finitely many inequivalent ways to remove $\neq_2$-literals or clauses with such literals and to add $=_2$- or $\neq_2$-conjuncts to the formulas $\phi_1 \wedge \psi_2$ defining one of the finitely many relations in $\bD$.
 
Every weakly $1$-determined clause of a formula in $\Phi'$ that does not contain a literal of the form $x \neq_2 y$ is
 $1$-determined. Every $\phi' \in \Phi'$ can thus be written as $\phi_1' \wedge \phi_2'$, where $\phi_i'$ is a conjunction of $i$-determined clauses.
It follows that $\phi'$ is equivalent to $\psi_1' \wedge \phi_2'$ where $\psi_1'=\Cr(\phi',\{1\}, \phi_1')$.
 Note that $\psi_1'$ is preserved by $\Pol(\bD)$.
 Since $\bD$ has finite relational signature, there are only finitely many inequivalent formulas that can arise in this way so the formulas $\psi_1'$ can be computed in polynomial time: they can simply be stored in a fixed-size database that is computed off-line.  
 By Lemma~\ref{lem:pres}, $\hat \psi_1'$ is  
 preserved by the operation $m \in \theta_1(\Pol(\bD))$. 
 Let $\Psi_1$ be the set of all formulas $\hat \psi_1'$ obtained from $\Phi'$ in this way.
 We may use the algorithmic part of Theorem~\ref{thm:tcsps}
 to decide whether $\Psi_1$ is satisfiable. 
 If $\Psi_1$ is not satisfiable, then we reject the input $\bA$ and we accept it otherwise. 
 We claim that in this case $\Phi$ is satisfiable and, consequently, that $\bA$ has a homomorphism to $\bD$.
 
Indeed, let $s \colon A \to {\mathbb Q}$ be a solution to $\Psi_1$
and let $t \colon A \to {\mathbb Q}$ be a solution to $\Psi_2$. 
Since $\bD$ is preserved by an operation $g$ such that  $\theta_2(g)$
is an $\lele$-operation, and $\lele$-operations are injective, we may assume that $t$ satisfies $x \neq_2 y$ unless this literal has been removed from $\Phi$ by the algorithm. 
Then the map $x \mapsto (s(x),t(x))$
satisfies all formulas in $\Phi$ and it follows that $\bA$ admits a homomorphism to $\bD$.  
\end{proof}

\begin{algorithm}
\caption{Solve-by-Factors}\label{alg:main}
\begin{algorithmic}
\State \textbf{Input:} an instance $\bA$ of $\Csp(\bD)$
\State $\Phi \coloneqq \{ \phi_{R,\bar{a}} \mid R\in \tau, \bar{a} \in R^{\bA} \}$

\ForAll{$\phi \in \Phi$}
    \State \parbox[t]{\dimexpr\linewidth-\algorithmicindent}{ write $\phi$ as $\phi_1 \wedge \psi_2$, where $\phi_1$ is a conjunction of weakly $1$-determined clauses and $\psi_2$ is a conjunction of $2$-determined clauses such that each clause of $\hat{\psi}_2$ is $\lele$-Horn}
\EndFor

\State $\Psi_2 \coloneqq \{\hat{\psi}_2 \mid \phi \in \Phi \}$

\blockcomment{Satisfiability of $\Psi_2$ can be checked in polynomial time by Theorems~\ref{thm:tcsps} and \ref{thm:lele}.}
\If{$\Psi_2$ is not satisfiable}
    \State \textbf{reject}
\EndIf

\State let $E'$ be the equality set of $\Psi_2$
\State $E \coloneqq \{(x^2, y^2) \mid (x,y) \in E' \}$

\ForAll{$\phi \in \Phi$}
    \State \parbox[t]{\dimexpr\linewidth-\algorithmicindent}{write $\cm(\phi,E)$ as $\bigwedge_{p\in \{1,\dots, n\}\setminus S}{\phi_1'}\wedge \psi_2 \wedge \psi_2'$, where $\phi_1'$ is the conjunction of $1$-determined clauses resulting from $\phi_1$ and $\psi_2'$ is a conjunction of the added conjuncts}
\EndFor

\State $\Phi' \coloneqq \{\cm(\phi, E) \mid \phi \in \Phi\}$

\blockcomment{Every $\phi' \in \Phi'$ defines a relation that is primitively positively definable over $\bD$.}

\State $\Psi_1 \coloneqq \{\hat{\psi_1} \mid \phi' \in \Phi', \psi_1 = \Cr(\phi',\{1\},\phi_1')\}$

\blockcomment{$\Psi_1$ is preserved by $\min_3$, $\mx_3$, or $\mi_3$ and its satisfiability can be checked in polynomial time by Theorems~\ref{thm:pwnu} and \ref{thm:tcsps}.}
\If{$\Psi_1$ is not satisfiable}
    \State \textbf{reject}
\EndIf

\State \textbf{accept}
\end{algorithmic}
\end{algorithm}

\subsubsection{Classification Result}
\label{sect:class}
The known results about first-order expansions of $({\mathbb Q};<)$
from Section~\ref{sec:tempcons}
combined with
the results from Section~\ref{sect:decomp} imply an  algebraic dichotomy for polymorphism clones of first-order expansions of $(\Q^2;<_1,=_1,<_2,=_2)$; the algebraic dichotomy implies a complexity dichotomy, using the results from Section~\ref{sect:tract}.

\begin{theorem}\label{thm:main}
Let $\bD$ be a first-order expansion of $(\Q; <) \boxtimes (\Q;<)$.
Exactly one of the following two cases applies.
\begin{itemize}
\item For each $i \in \{1,2\}$
 we have that $\theta_i(\Pol(\bD))$ contains  
$\min_3$, $\mx_3$, $\mi_3$, or $\lele_3$, or one of their duals. Furthermore, $\bD$ has a pwnu polymorphism and if $\bD$ has a finite relational signature, then $\Csp(\bD)$ is in P.
\item $\Pol(\bD)$ has a 
uniformly continuous minor-preserving map
to $\Pol(K_3)$. 
In this case, $\bD$ has a finite-signature reduct whose
$\Csp$ is NP-complete.
\end{itemize}
\end{theorem}

\begin{proof}
For $i \in \{1,2\}$, let
${\mathscr C}_i = \theta_i(\Pol(\bD))$. 
If ${\mathscr C}_i$, for some $i \in \{1,2\}$, has a uniformly continuous minor-preserving map from ${\mathscr C}_i$ to $\Pol(K_3)$, then by composing uniformly continuous minor-preserving maps there is also a uniformly continuous minor-preserving map from $\Pol(\bD)$ to $\Pol(K_3)$, which implies that $\bD$ has a finite signature reduct whose $\Csp$ is NP-hard by Corollary~\ref{cor:hard}.
Assume henceforth that there is no uniformly continuous minor-preserving map from $\mathscr C_i$ to $\Pol(K_3)$. 

Let $\bD_i$ be the structure with domain ${\mathbb Q}$ that contains all relations that are preserved by ${\mathscr C}_i$;
note that $\Pol(\bD_i) = {\mathscr C}_i$ by Proposition~\ref{prop:closed}. Clearly, ${\mathscr C}_i$ contains $\Aut({\mathbb Q};<)$ and preserves $<$, 
so $\bD_i$ is a first-order expansion of $({\mathbb Q};<)$.
By Theorem~\ref{thm:tcsps}, $\mathscr C_i$ contains $\min_3$, $\mx_3$, $\mi_3$, or $\lele_3$, or the dual of one of these operations. By Lemma~\ref{lem:flip}, we may assume that $\mathscr C_i$ contains an operation $f_i \in \{\min_3, \mx_3, \mi_3, \lele_3$\}; we assume without loss of generality that $f_i = \lele_3$ whenever $\mathscr C_i$ contains $\lele_3$.
If ${\mathscr C}_i$, for some $i \in \{1,2\}$,  contains a $\lex$-operation, then
${\mathscr C}_i$ contains  $\lele_3$: otherwise, 
$\theta_i(\Pol(\bD))$ would have to contain $\mi_3$, $\mx_3$, or $\mi_3$, and by Proposition~\ref{prop:pp} a $\pp$-operation. It then follows from the final statement in Theorem~\ref{thm:lele} that ${\mathscr C}_i$ contains $\lele_3$ and this contradicts our assumptions.

Assume that ${\mathscr C}_{2}$ contains $\lele_3$, and hence an $\lele$-operation by Theorem~\ref{thm:lele}. 
If ${\mathscr C}_1$ contains 
$\min_3$, $\mi_3$, or $\mx_3$,  
then the polynomial-time  tractability of $\Csp(\bD)$ follows from
Proposition~\ref{prop:tract2}.
Otherwise, 
${\mathscr C}_{1}$ contains 
$\lele_3$ (and hence an
$\lele$-operation by Theorem~\ref{thm:lele}) and the polynomial-time  tractability of $\Csp(\bD)$ follows from Proposition~\ref{prop:tract}.
The case when $\mathscr C_1$ contains $\lele_3$ follows from the same argument with the roles of the two dimensions exchanged.

Suppose in the following that neither ${\mathscr C}_1$ nor ${\mathscr C}_2$ contains a $\lex$-operation. Remark~\ref{rem:ll-lex} combined with Theorem~\ref{thm:lele} imply that they do not contain the operation $\lele_3$.
Then,
both 
${\mathscr C}_1$ and ${\mathscr C}_2$ contain 
$\min_3$, $\mi_3$, or $\mx_3$, 
and 
Proposition~\ref{prop:pp} imply that 
both ${\mathscr C}_1$ and ${\mathscr C}_2$ contain a $\pp$-operation. By 
Proposition~\ref{prop:factor}, 
every relation of $\bD$ can be defined by a conjunction of $1$-determined clauses and of $2$-determined clauses. 
Thus, Lemma~\ref{lem:factors} 
implies
that  $\Pol(\bD)$ contains
${\mathscr C}_1 \times {\mathscr C}_2$.
Then the polynomial-time tractability of $\Csp(\bD)$ follows from Corollary~\ref{cor:prod-alg} 
applied to $\bD_1,\bD_2$, and $\bD$.

We continue by proving that $\bD$ admits a pwnu polymorphism.
Let $f$ be the ternary operation such that $\theta_i(f)$ equals $f_i$ for every $i \in \{1, 2\}$. We claim that $f$ preserves $\bD$. Let $\phi$ be a formula that defines a relation from $\bD$. By Proposition~\ref{prop:decomp}, we may assume that $\phi$ is a normal conjunction of clauses each of which is weakly $i$-determined for some $i$. Let $a,b,c$ be tuples that satisfy $\phi$. Let $\psi$ be a clause of $\phi$. We may assume that $\psi$ is weakly $2$-determined, since the case where it is weakly $1$-determined can be treated analogously. Then $\psi$ is of the form $\psi'\vee \psi''$, where $\psi'$ is a $2$-determined clause and $\psi''$ is a disjunction of $\neq_1$-literals.
We show that $f(a,b,c)$ satisfies $\psi$.

Note that it follows from the discussion above that $\mathscr C_i$ contains for each $i\in\{1,2\}$ a $\pp$-operation or an $\lele$-operation.
If $\psi$ contains a literal $x\neq_1 y$, then it is not $2$-determined and it follows from Proposition~\ref{prop:factor} that either $\mathscr C_1$ does not contain a $\pp$-operation or $\mathscr C_2$ contains a $\lex$-operation.
In the first case, $\mathscr C_1$ does not contain $\min_3$, $\mx_3$ or $\mi_3$ by Proposition~\ref{prop:pp} so $f_1=\lele_3$. 
In the second case, Proposition~\ref{prop:syntax-2} implies that $\psi$ is weakly $1$-determined. 
We see that $f_2=\lele_3$
by Theorem~\ref{thm:lele} since either (1) $\mathscr C_2$ simultaneously contains a $\lex$-operation and a $\pp$-operation or (2) $\mathscr C_2$ contains an $\lele$-operation. We may therefore assume that $f_1=\lele_3$ since, otherwise, we can treat $\psi$ as a weakly $1$-determined clause.

If one of $a,b,c$ satisfies the literal $x \neq_1 y$, then $f(a,b,c)$ satisfies the literal as well since $\theta_1(f)=f_1$ is injective. So suppose that none of $a,b,c$ satisfies such literals. We show that $f(a,b,c)$ satisfies $\psi'$.
Since $f_2 \in \mathscr C_2$, there is $f_2'\in \Pol(\bD)$ such that $\theta_2(f_2')=f_2$. Since $f_2'$ preserves $\phi$ and the relation $=_1$, the tuple $f_2'(a,b,c)$ must satisfy the $2$-determined clause $\psi'$ and hence $f_2(a,b,c)$ satisfies $\hat{\psi'}$ by Lemma~\ref{lem:pres}. Another application of Lemma~\ref{lem:pres} shows that $f(a,b,c)$ satisfies $\psi'$ as well.

Finally, we prove that $f$ is indeed a pwnu polymorphism of $\bD$. If $e^i_1, e^i_2, e^i_3$ show that $f_i$ is a pwnu polymorphism of $\bD_i$, then $e_j := (e^1_j, e^2_j)$, for $j \in \{1,2,3\}$, are the endomorphisms of $\bD$ that show that $f$ is a pwnu polymorphism of $\bD$.

Since $\overline{\Aut(\Q;<)}=\End(\Q;<)$, it follows from Lemma~\ref{lem:prod-mc-cores} that $\overline{\Aut(\bD)}=\End(\bD)$. Therefore, Lemma~\ref{lem:excl} implies that the two cases in the statement are mutually exclusive.
\end{proof}

\subsection{Classification in the $n$-Dimensional Case}
\label{sect:n-dim}
The approach in the previous section can be generalised to first-order expansions of $(\mathbb Q;<)^{(n)}  = (\mathbb Q^n;<_1,=_1,\dots,<_n,=_n)$.
However, this requires some work both on the algebraic and the algorithmic side. In order to make the transition as smooth as possible, we generalize weakly $i$-determined clauses into {\em $S$-weakly $i$-determined} clauses. This allows us to systematically transfer both our 2-dimensional algebraic results (Propositions~\ref{prop:syntax-2}-\ref{prop:pp-pow}) and algorithmic results (Propositions~\ref{prop:tract} and~\ref{prop:tract2}
into the $n$-dimensional setting.
From now on and for the remainder of Section~\ref{sect:products}, we let the symbol $\bD$ denote a first-order expansions of $(\mathbb Q;<)^{(n)}  = (\mathbb Q^n;<_1,=_1,\dots,<_n,=_n)$.
We begin by generalising Definition~\ref{def:weakly-determined}.

\begin{definition}
Let $S \subseteq \{1,\dots,n\}$ and $p \in \{1,\dots,n\}$. 
A clause is called \emph{$S$-weakly $p$-determined} if it is
of the form 
$$\psi \vee \bigvee_{i \in \{1,\dots, k\} \; {\rm and} \; j_i \in S} x_i \neq_{j_i} y_i$$ 
where $\psi$ is  $p$-determined and $k \geq 0$.
A clause is called \emph{weakly $p$-determined} if it is $\{1,\dots,n\} \setminus \{p\}$-weakly $p$-determined (note that this is consistent with the notion of weakly $p$-determined for $n=2$ from Definition~\ref{def:weakly-determined}).
\end{definition}

Next, we connect conjunctions of $S$-weakly $p$-determined clauses with first-order expansions of $({\mathbb Q}; <)^{(n)}$ that admit certain
polymorphisms.
Recall the normal forms of formulas over $(\Q; <)^{(n)}$ defined at the end of Section~\ref{sec:normalisation}; we will use them in similar fashion as in the case $n=2$.

\begin{proposition}\label{prop:syntax-n}
Suppose that for every $p \in \{1,\dots,n\}$  there is an operation $f \in \Pol(\bD)$ such that $\theta_p(f)$ is an $\lele$-operation or a $\pp$-operation. 
Then for every relation $R$ of $\bD$, if $\phi$ is a first-order definition of $R$ over $(\mathbb Q^n;<_1,=_1,\dots,<_n,=_n)$ that is normal, then $\phi$ is 
a conjunction of clauses 
each of which is weakly $i$-determined for some $i \in \{1,\dots,n\}$. 
\end{proposition}
\begin{proof}
The proof is a generalisation of the proof of Proposition~\ref{prop:decomp}; the key step is to use the generalisation of Corollary~\ref{cor:dom} which states that for $i, j\in \{1, \dots, n\}$ and $f\in \Pol(\bD)$ such that $\theta_i(f)$ is an $\lele$- or a $\pp$-operation there is an operation $g \in \Pol(\bD;\leq_1,\neq_1,\dots,\leq_n,\neq_n)$ such that \begin{itemize}
    \item $\theta_i(g)$ is dominated by the first argument (or even equal to $\pi_1^2$ if $\theta_i(f)$ is a $\pp$-operation) and 
    \item $\theta_j(g)$ is dominated by the second argument.
\end{itemize}
To prove this, note first that we may without loss of generality assume that $f$ preserves $\leq_k$ and $\neq_k$ for all $k$; for $k=i$ this follows from the assumption and for $k \neq i$ we may repeatedly apply 
the $n$-dimensional generalisation of Lemma \ref{lem:canonise}
(as discussed immediately after Lemma~\ref{lem:canoni}) 
to obtain an operation that preserves $\leq_k$ and $\neq_k$. By applying Lemma \ref{lem:canonise} to canonise the operation in the $j$-th position and subsequent application of the $n$-dimensional generalisation of Lemma \ref{lem:dom} to modify the $i$-th position, we can prove the statement analogously to the proof of Corollary \ref{cor:dom}.

To prove the proposition, one can proceed as in the proof of Proposition~\ref{prop:decomp}; the only difference is that the choice of the polymorphism $g$ depends on the considered pair of literals, because it needs to have the domination property in the right dimensions.
In fact, the generalisation of that proof yields the conclusion under the weaker assumption that for all but at most one $p \in \{1,\dots,n\}$ there exists $f \in \Pol(\bD)$ such that $\theta_p(f)$ is an $\lele$-operation, a $\pp$-operation, or the dual of such an operation. 
\end{proof}

The following proposition gives a concrete syntactic description for defining formulas of relations of $\bD$ based on the polymorphisms of $\bD$.
It can be viewed as
a generalisation of Propositions~\ref{prop:syntax-2}-\ref{prop:pp-pow}.

\begin{proposition}\label{prop:strong-syntax-n}
Let $S \subseteq \{1,\dots,n\}$
be such that 
\begin{itemize}
    \item for every $p \in S$ there exists $f_p \in \Pol(\bD)$ such that 
$\theta_p(f_p)$ is an $\lele$-operation, and
    \item for every $p \in \{1,\dots,n\} \setminus S$ 
    there exists $f_p \in \Pol(\bD)$ 
    such that $\theta_p(f_p)$ is a $\pp$-operation, 
    but there is
    no $g \in \Pol(\bD)$ such that $\theta_p(g)$ is
    a $\lex$-operation. 
\end{itemize}
Then, the following hold:
\begin{enumerate}
\item  every relation of 
$\bD$ can be defined by a conjunction of clauses each of which is an $S$-weakly $p$-determined clause for some $p \in \{1,\dots,n\}$,
\item if $p \in \{1,\dots,n\}\setminus S$, then the $S$-weakly $p$-determined clauses can be chosen to be of the form
\begin{equation} \label{eq:p-det-pp}
    u_1 \neq_{i_1} v_1 \vee \cdots u_m \neq_{i_m} v_m \vee y_1 \neq_p x \vee \cdots \vee y_k \neq_p x \vee z_1 \leq_p x \vee \cdots \vee  z_l \leq_p x,
\end{equation}
where $i_1, \dots, i_m\in S$, and
\item if $p \in S$,
then the $S$-weakly $p$-determined clauses can be chosen to be $\lele$-Horn clauses of the form
$$x_1 \neq_{i_1} y_1 \vee \cdots \vee x_m \neq_{i_m} y_m \vee z_1 <_p z_0 \vee \cdots \vee z_{\ell} <_p z_0 \vee (z_0 =_p z_1 =_p \cdots =_p z_{\ell})$$
for $i_1,\dots,i_m \in S$
(and where the last disjunct may not appear).
\end{enumerate}
\end{proposition}

\begin{proof}
Every relation of $\bD$ has a definition by a normal formula and, by  Proposition~\ref{prop:syntax-n}, it can be defined by a conjunction of clauses each of which is weakly $i$-determined for some $i \in \{1,\dots,n\}$. 
Let $R$ be a relation of $\bD$ and let $\phi$ be such a definition. We show step by step that the statements in items $1-3$ hold true for $R$.

\smallskip

\noindent\textit{Proof of item 1.}
Let $\psi$ be a weakly $i$-determined clause of $\phi$ for some $i\in \{1,\dots, n\}$.
When $i\in S$ and $j\in \{1,\dots,n\}\setminus S$, then we can proceed as in the proof of Proposition~\ref{prop:syntax-2}: we use the generalised version of Corollary~\ref{cor:dom} described in detail in the previous proof and we rule out the possibility that $\psi$ simultaneously contains a $\{<_i, =_i\}$-literal and a $\neq_j$-literal. Therefore, $\psi$ is an $S$-weakly $i$-determined clause or an $S$-weakly $j$-determined clause in this case.

If $i,j \in \{1,\dots, n\}\setminus S$ are distinct,
then there is an operation $g \in \Pol(\bD; \leq_1, \neq_1, \dots, \leq_n, \neq_n)$ with $\theta_i(g)=\pi_1^2$ and $\theta_j(g)=\pi_2^2$.
To see this, we first assume that $g$ preserves $\leq_k$ and $\neq_k$ for all $k$ --- if this is not the case, then we repeatedly apply Lemma~\ref{lem:canonise} in all but the $i$-th dimension. The existence of $g$ can now be proved similarly as in the proof of Proposition~\ref{prop:factor} (the proof uses the fact that there is no $g \in \Pol(\bD)$ such that $\theta_i(g)$ or $\theta_j(g)$ is a $\lex$-operation).
As in the proof of Lemma~\ref{lem:factors}, $g$ can be used to prove that $\psi$ cannot contain a $\{<_i, =_i, \neq_i\}$-literal and $\neq_j$-literal at the same time. Hence, the clause $\psi$ is an $S$-weakly $i$-determined clause. It follows that the normal formula $\phi$ is in fact a conjunction of clauses each of which is $S$-weakly $p$-determined for some $p$.

\smallskip
\noindent\textit{Proof of item 2.}
We now prove that we may choose the clauses that are $S$-weakly $p$-determined for $p\not\in S$ to have the syntactic form (\ref{eq:p-det-pp}). We will proceed analogously to the proof of Proposition~\ref{prop:syntax-3}. Let $p\in \{1,\dots,n\}\setminus S$ and $\phi=\phi_p \wedge \phi_0$, where $\phi_p$ is the conjunction of all $S$-weakly $p$-determined clauses of $\phi$ and $\phi_0$ is the conjunction of the remaining clauses. Let $\phi_p'$ be the conjunction of all clauses of the form (\ref{eq:p-det-pp}) that are reduced and implied by $\phi$. We will show that $\phi'=\phi_p' \wedge \phi_0$ implies $\phi$ and hence defines $R$. Applying the same procedure for all $p \in \{1,\dots,n\}\setminus S$ concludes the proof of item 2.

To see that $\phi'$ implies $\phi$, we use the same orbit argument as in the proof of Proposition~\ref{prop:syntax-3}: $\phi_p$ is equivalent to a conjunction of $S$-weakly $p$-determined clauses of the form
\[\chi \vee y_1 \circ_1 y_2 \vee \cdots \vee y_k \circ_k y_{k+1},\]
where $\chi = \bigvee_{j=1}^m u_j \neq_{i_j} v_j$, $i_j \in S$, $j=1,\dots,m$, and $\circ_1,\dots,\circ_k \in \{\neq_p,\geq_p\}$. We assume that these clauses are minimal in the same sense as in the proof of Proposition~\ref{prop:syntax-3}. Let $\psi$ be such a clause of $\phi$. The argument in the rest of the proof of Proposition~\ref{prop:syntax-3} is not dependent on the indices $i_j$ in the literals $u_j \neq_{i_j} v_j$ in $\chi$. Thus, it is applicable also in this case and shows that $\psi$ is implied by $\phi'$. This proves that $\phi'$ implies $\phi$ since $\psi$ was chosen arbitrarily.

\smallskip
\noindent\textit{Proof of item 3.}
By items 1 and 2, we may assume without loss of generality that $\phi$ satisfies the following condition: for every $p\in \{1,\dots,n\}\setminus S$, every $S$-weakly $p$-determined clause of $\phi$ is of the form (\ref{eq:p-det-pp}). Let $\phi_1$ be a conjunction of all $S$-weakly $p$-determined clauses of $\phi$ where $p\in S$ and let $\phi_2$ be the conjuction of the remaining clauses of $\phi$. We will now use Corollary~\ref{cor:extract-n} and the conjunction replacement operator $\Cr(\cdot)$. We note that $\phi_1$ is a conjunction of $S$-determined clauses 
and hence $\phi$ is equivalent to a formula $\psi_1\wedge \phi_2$ where $\psi_1 = \Cr(\phi,S,\phi_1)$. 
Without loss of generality, we may assume that $\psi_1$ is normal. By Proposition \ref{prop:syntax-n}, $\psi_1$ is in fact a conjunction of clauses each of which is weakly $p$-determined for some $p$ and hence $S$-weakly $p$-determined for some $p\in S$.

Recall the variable expansion operator $\ve(\cdot)$ that we defined just before Proposition~\ref{prop:pp-pow}. 
Since for every $p\in S$, there is an operation $f_p \in \Pol(\bD)$ such that $\theta_p(f_p)$ is an $\lele$-operation and $\psi_1$ is $S$-determined, it can be shown analogously to the claim in the proof of Proposition~\ref{prop:pp-pow} that $\ve(\psi_1)$ is preserved by every $\lele$-operation. By Theorem~\ref{thm:lele}, $\ve(\psi_1)$ is equivalent to a conjunction of $\lele$-Horn clauses. Since the formula $\psi_1 \wedge \phi_2$ defines $R$, item 3 follows.
\end{proof}

Note that the formula produced by Proposition~\ref{prop:strong-syntax-n} is not necessarily normal. Also note the difference between the proof for $n=2$ and general $n$: For $n=2$, there are just three cases -- $S$ is empty, $S=\{p\}$ for some $p$,  or $S = \{1, 2\}$. If $S=\{p\}$, then $S$-determined clauses are $p$-determined. If $S=\{1, 2\}$, then the formula $\phi_1$ is equal to $\phi$ and thus trivially preserved by $\Pol(\bD)$.

We continue by generalizing our algorithmic results.
Recall the clause modification operator $\cm$ from Section~\ref{sect:tract}. We generalize this operator for the case of arbitrary $n>2$. Let $X$ be a set of variables, $S\subseteq \{1, \dots n\}$ and $E\subseteq \{(x^i,y^i)\mid x,y \in X, i\in S\}$
where $x^i,y^i$ denote fresh variables.
Assume $\phi$ is a formula over $(\Q^n,<_1,=_1, \dots, <_n,=_n)$ that only uses variables from $X$. If $\phi$ equals
$\bigwedge_{p\in \{1,\dots, n\}\setminus S}{\phi_p}\wedge \phi_S$, where $\phi_p$ is a conjunction of $S$-weakly $p$-determined clauses for each $p$ and $\phi_S$ is a conjunction of $S$-determined clauses, then $\cm(\phi, E)$ denotes the formula resulting from $\phi$ by performing the following procedure for each $p\in \{1,\dots,n\} \setminus S$:
 \begin{itemize}
     \item If $(x^i,y^i)\in E$ and $\phi_p$ contains a clause with the literal $x \neq_i y$, then remove this literal from the clause and add the conjunct $x =_i y$ to $\phi$.
     \item If $(x^i,y^i)\not\in E$ and $\phi_p$ contains a clause with the literal $x \neq_i y$, then delete all literals in the clause but this one.
 \end{itemize}
 Note that the previously defined notion of the operator $\cm$ was a special case for $n=2$ and $S=\{i\}\subseteq \{1,2\}$.
 
The proof of our $n$-dimensional computational result is based on the ideas underlying
the 2-dimensional results found in Propositions~\ref{prop:tract} and~\ref{prop:tract2}, but there are also crucial differences. It may be enlightening to compare the generalized Algorithm~\ref{alg:main2} with
the two-dimensional Algorithm~\ref{alg:main} underlying Proposition~\ref{prop:tract2}. In particular, such a comparison may clarify how $S$-weakly $p$-determined and the $n$-dimensional $\cm$ operator come into play.

\begin{proposition}\label{prop:tract-n}
Suppose that $\bD$ has a finite relational signature $\tau$ and suppose that for each $p \in \{1,\dots,n\}$ there exists $f_p \in \Pol(\bD)$ such that $\theta_p(f_p)$ equals  $\lele_3$, $\min_3$, $\mx_3$, or  $\mi_3$.
Then, $\Csp(\bD)$ can be solved in polynomial time. 
\end{proposition}

\begin{proof}
Without loss of generality, suppose that for every $p\in \{1,\dots,n\}$ such that $\theta_p(\Pol(\bD))$ contains $\lele_3$, the operation $f_p$ is chosen to be such that $\theta_p(f_p) = \lele_3$. Let $S\subseteq \{1, \dots, n\}$ be the set of all such $p$.   Moreover, we may assume that $\bD$ contains relations $\neq_i$ for every $i \in S$. Otherwise we repeatedly apply Lemma~\ref{lem:canoni} on the operations $f_p$, $p\in \{1, \dots, n\}$, and obtain polymorphisms $f_p'$ such that
\begin{itemize}
    \item $\theta_i(f_p')$ is canonical over $\Aut(\Q; <)$ for every $i\neq p$ (and hence preserves $\neq_i$ by Lemma~\ref{lem:lex}),  
    \item $\theta_p(f_p)$ equals $\lele_3$, $\min_3$, $\mx$, or  $\mi_3$ (and hence preserves $\neq_p$ whenever $p\in S$).
\end{itemize}

Note that $\theta_p(\Pol(\bD))$ contains a $\pp$-operation for every $p \in \{1,\dots, n \}\setminus S$ (by Proposition~\ref{prop:closed} and Proposition~\ref{prop:pp}), 
but it does not contain a $\lex$-operation (by 
Theorem \ref{thm:lele}). 
By Theorem~\ref{thm:lele}, $\theta_p(\Pol(\bD))$ contains a $\lele$-operation for $p\in S$.
Let $\bA$ be an instance of $\Csp(\bD)$. 
For every $R \in \tau$ of arity $k$ and  $\bar{a}=(a_1,\dots,a_k) \in R^{\bA}$, let $\phi_{R,\bar{a}}$ be the first-order definition of $R$ in the structure $(\mathbb Q^n;<_1,=_1,\dots,<_n,=_n)$
using the elements $a_1,\dots,a_k$ as the free variables. 
We then may assume that $\phi_{R,\bar{a}}$ is of the form as described in Proposition~\ref{prop:strong-syntax-n}.


\begin{algorithm}
\caption{Solve-by-Factors-$n$}\label{alg:main2}
\begin{algorithmic}
\State \textbf{Input:} an instance $\bA$ of $\Csp(\bD)$
\State $\Phi \coloneqq \{ \phi_{R,\bar{a}} \mid R \in \tau, \bar{a} \in R^{\bA} \}$

\ForAll{$\phi \in \Phi$}
    \State \parbox[t]{\dimexpr\linewidth-\algorithmicindent}{write $\phi$ as $\bigwedge_{p \in \{1, \dots, n\} \setminus S} \phi_p \wedge \phi_S$, where $\phi_p$ is a conjunction of $S$-weakly $p$-determined clauses and $\phi_S$ is a conjunction of $S$-determined $\lele$-Horn clauses}
\EndFor

\State $\Psi_S \coloneqq \{\ve(\phi_S) \mid \phi \in \Phi \}$

\blockcomment{$\Psi_S$ contains conjunctions of $\lele$-Horn clauses over $(\Q,<)$ and its satisfiability can be checked in polynomial time by Theorems~\ref{thm:tcsps} and \ref{thm:lele}.}
\If{$\Psi_S$ is not satisfiable}
    \State \textbf{reject}
\EndIf

\State let $E$ denote the equality set corresponding to $\Psi_S$

\ForAll{$\phi \in \Phi$}
    \State \parbox[t]{\dimexpr\linewidth-\algorithmicindent}{write $\cm(\phi, E)$ as $\bigwedge_{p \in \{1, \dots, n\} \setminus S} \phi_p' \wedge \phi_S \wedge \phi_S'$, where $\phi_p'$ is the conjunction of $p$-determined clauses resulting from $\phi_p$ and $\phi_S'$ is a conjunction of the added conjuncts}
\EndFor

\State $\Phi' \coloneqq \{\cm(\phi, E) \mid \phi \in \Phi\}$

\ForAll{$p \in \{1, \dots, n\} \setminus S$}
    \blockcomment{Every $\phi' \in \Phi'$ defines a relation that is primitively positively definable over $\bD$.}
    \State $\Psi_p \coloneqq \{\hat{\psi_p} \mid \phi' \in \Phi', \psi_p = \Cr(\phi', \{p\}, \phi_p')\}$
    \blockcomment{$\Psi_p$ is preserved by $\min_3$, $\mx_3$, or $\mi_3$.}
    \blockcomment{Satisfiability of $\Psi_p$ can be checked in polynomial time by Theorems~\ref{thm:pwnu} and \ref{thm:tcsps}.}
    \If{$\Psi_p$ is not satisfiable}
        \State \textbf{reject}
    \EndIf
\EndFor

\State \textbf{accept}
\end{algorithmic}
\end{algorithm}

In Algorithm~\ref{alg:main2} we present the algorithm for deciding whether a given instance $\bA$ of $\Csp(\bD)$ has a homomorphism to $\bD$. The algorithm uses the same ideas as are used in the proof of Proposition~\ref{prop:tract2}. In particular, we can show by similar arguments that the sets $\Phi$, $\Psi_S$, $\Phi'$ and $\Psi_p$, $p\in \{1,\dots,n\}\setminus S$, can be computed in polynomial time. Finally, the set
$E$ can be computed in polynomial time (as was pointed out after Proposition~\ref{prop:tract}) so the whole algorithm runs in polynomial time.

It is clear that if the algorithm rejects, then there is no homomorphism from $\bA$ to $\bD$. In case that the algorithm accepts, the existence of a homomorphism from $\bA$ to $\bD$ can be proved in a similar fashion as in the proof of Proposition~\ref{prop:tract2}.
\end{proof}

We are now in the position of proving the main result of this section.

\begin{theorem}\label{thm:n}
Let $\bD$ be a first-order expansion of $(\Q; <_1, =_1, \dots, <_n, =_n)$. Exactly one of the following two cases applies. 
\begin{itemize}
\item For each $p \in \{1,\dots,n\}$ we have that $\theta_p(\Pol(\bD))$ contains $\min_3$, $\mx_3$, $\mi_3$, or $\lele_3$, or one of their duals. In this case, $\bD$ has a pwnu polymorphism. 
If $\bD$ has a finite relational signature, then $\Csp(\bD)$ is in P. 
\item $\Pol(\bD)$ has a uniformly continuous minor-preserving map to $\Pol(K_3)$. In this case $\bD$ has a finite-signature reduct whose $\Csp$ is NP-complete.
\end{itemize}
\end{theorem}
\begin{proof}
For $p \in \{1,\dots,n\}$, define 
${\mathscr C}_p \coloneqq \theta_p(\Pol(\bD))$. 
If ${\mathscr C}_p$, for some $p \in \{1,\dots,n\}$, has a uniformly continuous minor-preserving map from ${\mathscr C}_p$ to $\Pol(K_3)$,
then by composing uniformly continuous minor-preserving maps there is also a uniformly continuous minor-preserving map from $\Pol(\bD)$ to $\Pol(K_3)$. This implies that $\bD$ has a finite-signature reduct whose $\Csp$ is NP-hard by Corollary~\ref{cor:hard}.

Otherwise, 
for every $p \in \{1,\dots,n\}$ the clone $\mathscr C_p$ does not have a uniformly continuous minor-preserving map to $\Pol(K_3)$. Since ${\mathscr C}_p$ is a closed clone (by Proposition~\ref{prop:closed}) 
that contains $\Aut({\mathbb Q};<)$ and preserves $<$, there exists
a first-order expansion $\bD_p$ of $\Aut({\mathbb Q};<)$ such that $\Pol(\bD_p) = {\mathscr C}_p$. 
We may therefore apply Theorem~\ref{thm:tcsps} and conclude that 
for every $p \in \{1,\dots,n\}$ 
the clone ${\mathscr C}_p$ contains an operation $f_p$ which equals $\min_3$, $\mx_3$, $\mi_3$, or $\lele_3$, or the dual of one of these operations. 
By the version of Lemma~\ref{lem:flip} for $n$-fold algebraic products, we may assume without loss of generality that $f_p \in \{ \min_3, \mx_3, \mi_3,  \lele_3\}$
for every $p\in \{1, \dots, n\}$.
In case that $\bD$ has a finite relational signature, the polynomial-time  tractability of $\Csp(\bD)$ follows from Proposition~\ref{prop:tract-n}.

We may assume that $f_p=\lele_3$ for every $p$ such that $\mathscr C_p$ contains $\lele_3$.
Let $f$ be the ternary operation such that 
$\theta_p(f)$ equals $f_p$ for every $p \in \{1,\dots,n\}$. We claim that $f$ preserves $\bD$.
Let $\phi$ be a formula that defines a relation from $\bD$ and has the form as described in Proposition~\ref{prop:strong-syntax-n} (the argument in the proof of Proposition~\ref{prop:tract-n} implies that the assumptions are satisfied). 
We show that $f$ preserves $\phi$. 
Let $a,b,c$ be tuples that satisfy $\phi$ and let $\psi$ be a clause of $\phi$. We see that there is a $p\in \{1,\dots,n\}$ such that $\psi$ is $S$-weakly $p$-determined.
As in the proof of Theorem~\ref{thm:main}, one can show that $f(a,b,c)$ satisfies $\psi$.
It follows that $f$ preserves $\phi$ and $f \in \Pol(\bD)$.

As in the case when $n=2$ (Theorem~\ref{thm:main}), we can show that $f$ is a pwnu polymorphism, and hence it follows from Lemma~\ref{lem:prod-mc-cores} and Lemma~\ref{lem:excl} that the two cases of the statement are mutually exclusive.
\end{proof}

\subsection{Classification of Binary Relations} 
\label{sect:binary}
A relational signature is called \emph{binary} if all its relation symbols have arity two,
 and a relational structure is binary if its signature is binary. 
If $\bD$ is binary, then the results from the previous sections can be substantially strengthened. Note that an $\omega$-categorical structure has only finitely many distinct relations of arity at most two so we may assume that binary structures have a finite signature.

\begin{definition} \label{def:ord-hornclause}
A formula is called an \emph{Ord-Horn clause}
if it is of the form $$x_1 \neq y_1 \vee \cdots \vee x_m \neq y_m \vee z_1 \; \circ \; z_0$$
where $\circ \in \{<,\leq,=\}$,
it is permitted that $m = 0$, and
the final disjunct may be omitted. An \emph{Ord-Horn formula} is a conjunction of Ord-Horn clauses.
\end{definition}

Ord-Horn clauses are ll-Horn formulas and a first-order formula over $(\Q;<)$ is equivalent to an Ord-Horn formula if and only if it is preserved by an $\lele$-operation and the dual of an $\lele$-operation~\cite{ll}. 
We say that a relation has an Ord-Horn definition if it can be defined by an Ord-Horn formula.
The polynomial-time tractability of $\Csp(\bB)$ if all relations of $\bB$ have an Ord-Horn definition follows from Theorem~\ref{thm:tcsps} and Theorem~\ref{thm:lele}, but this was first shown by Nebel and B\"urckert~\cite{Nebel} using a very different approach. The reader should note that the Ord-Horn fragment does not have a characterisation in terms of equations satisfied by the polymorphism clone~\cite[Theorem 7.2]{RydvalFP}.

The following theorem strengthens the results of Theorem~\ref{thm:n} for binary first-order expansions, providing a very concrete syntactic condition for tractability based on Ord-Horn formulas.

\begin{theorem}\label{thm:cdc-binary}
Let $\bD$ be a binary first-order expansion of $(\Q; <_1, =_1, \dots, <_n, =_n)$.
Then exactly one of the following two cases applies. 
\begin{itemize}
\item Each relation in $\bD$, viewed as a relation of arity $2n$ over ${\mathbb Q}$, has an Ord-Horn definition. In this case, $\bD$ has a pwnu polymorphism and 
$\Csp(\bD)$ is in P. 
\item $\Pol(\bD)$ has a uniformly continuous minor-preserving map to $\Pol(K_3)$. In this case, 
$\Csp(\bD)$ is NP-complete.
\end{itemize} 
\end{theorem}
\begin{proof}
If the second item of the statement does not apply, then 
Theorem \ref{thm:n} implies that $\bD$ has a pwnu polymorphism and for each $p \in \{1,\dots,n\}$ we have that $\theta_p(\Pol(\bD))$ contains 
$\min_3$, $\mx_3$, $\mi_3$, or $\lele_3$, or one of their duals. 
By Lemma~\ref{lem:flip} we may focus on the situation that $\theta_p(\Pol(\bD))$ contains 
$\min_3$, $\mx_3$, $\mi_3$, or $\lele_3$ (note that the dual of a relation with an Ord-Horn definition has an Ord-Horn definition as well).
Then Proposition~\ref{prop:pp} and Theorem~\ref{thm:lele} imply that the assumptions of Proposition~\ref{prop:syntax-n} hold, and therefore every relation of $\bD$ can be defined by a normal conjunction of clauses each of which is weakly $s$-determined for some $s \in \{1,\dots,n\}$.
If such a clause contains 
two disjuncts of the form $x <_i y$ and $y <_i x$, then replace the disjuncts by $x \neq_i y$. 
If such a clause contains 
two disjuncts of the form $x =_i y$ and $x \neq_i y$, then remove the clause (since it is always true). 
If such a clause contains 
two disjuncts of the form $x <_i y$ and $x =_i y$, then replace the disjuncts by $x \leq_i y$. 
Since the relations of $\bD$ are binary, the resulting formula is Ord-Horn. 
The result follows since we know that the satisfiability of Ord-Horn formulas can be decided in polynomial time.
Lemma~\ref{lem:excl} implies that the two items cannot hold simultaneously.
\end{proof}

\section{Complexity Classification Transfer}
\label{sect:transfer}
Assume that ${\cal C}$ and ${\cal D}$ are classes of structures
and that the complexity of
CSP$(\bD)$ is known for every $\bD \in {\cal D}$.
A {\em complexity classification transfer} is
a process that systematically
uses this information for
inferring the complexity of
CSP$(\bC)$ for every $\bC \in {\cal C}$.
The particular method that we will use originally 
appeared in~\cite{Book}.
Combined with our classification for first-order expansions of $(\mathbb Q;<)^{(n)}$, this method
allows us to derive several new dichotomy results in Section~\ref{sect:applications}.

Let $\bC$ and $\bD$ denote relational
structures. Recall that an interpretation of $\bC$ in $\bD$ is a partial surjection from a finite power of $D$ to $C$.
Two interpretations $I$ and $J$ of $\bC$ in $\bD$ are called 
\emph{primitively positively homotopic}\footnote{We follow the terminology from~\cite{AhlbrandtZiegler}.} (pp-homotopic)
if the relation $\{(\bar x,\bar y) \; | \; I(\bar x) = J(\bar y) \}$
is primitively positively definable in $\bD$. 
The \emph{identity interpretation} of a $\tau$-structure $\bC$
 is the identity map on $C$, which is clearly a primitive positive interpretation. 
 We write
$I_1 \circ I_2$ for the natural composition of two interpretations $I_1$ and $I_2$. 

\begin{definition}
Let $\bC$ and $\bD$ be two mutually primitively positively interpretable structures with a primitive positive interpretation $I$ of $\bC$ in $\bD$ and a primitive positive interpretation $J$ of $\bD$ in $\bC$.
They are called \emph{primitively positively bi-interpretable} if additionally $I \circ J$ and $J \circ I$ are pp-homotopic to the identity interpretation. 
\end{definition}

Note that structures that are primitively positively bi-definable are in particular bi-interpretable (via 1-dimensional interpretations). It is known that two $\omega$-categorical structures are primitively positively bi-interpretable if and only if there is a topological clone isomorphism between their polymorphism clones \cite{Topo-Birk}.

\begin{ex}\label{ex:allen-bi}
The structure $({\mathbb Q};<)$ and the structure $({\mathbb I};{\sf m})$ from Example~\ref{ex:allen} are primitively positive bi-interpretable (Example 3.3.3 in~\cite{Book}): 
the identity map $I$ on $\mathbb I$ is a 2-dimensional primitive positive interpretation of $({\mathbb I};{\sf m})$ in $(\mathbb Q;<)$, and the projection $J \colon {\mathbb I} \to {\mathbb Q}$ to the first coordinate is a 1-dimensional 
interpretation of  $(\mathbb Q;<)$ in $({\mathbb I};{\sf m})$. 

Consider first the map $I$. The domain is $\mathbb{I} \subseteq \Q^2$, defined by a primitive positive formula $\delta(a,b) = a < b$. The relation 
\[I^{-1}( {\sf m})= \{ ((a,b), (c,d)) \mid I(a,b) \; {\sf m} \; I(c,d) \}\] has a primitive positive definition over $(\Q; <)$ by a formula
$b=c$. The relation 
\[I^{-1}(\equiv) = \{ ((a,b), (c,d)) \mid I(a,b) \; \equiv \; I(c,d) \}\]
has a primitive positive definition $(a=c) \wedge (b=d)$. It follows that $I$ is a primitive positive interpretation of $(\mathbb{I}; \sf m)$ in $(\Q; <)$.

Concerning $J$, the domain is the full set $\mathbb{I}$ and hence defined by a formula $x=x$. The relation \[J^{-1}(<) = \{ (x,y) \mid J(x) < J(y) \} = \{ ((x_1,x_2),(y_1,y_2)) \mid x_1 < y_1 \}\] has a primitive positive definition 
\[\exists u,v (u \; {\sf m} \; x) \wedge (u \; {\sf m} \; v) \wedge (v \; {\sf m} \; y).\]
The relation $J^{-1}(=)$ has a primitive positive definition $\phi(x,y)$ given by
\[\exists u (u \; {\sf m} \; x) \wedge (u \; {\sf m} \; y).\]
It follows that $J$ is a primitive positive interpretation.

Observe that $J \circ I$ is pp-homotopic to the identity interpretation over $(\Q; <)$ since the relation
\[\{ (a,b,c) \mid J \circ I (a,b) = c\}\]
has a primitive positive definition 
\[a=c.\]
Finally, we show $I \circ J$ is pp-homotopic to the identity interpretation over $(\mathbb{I}; {\sf m})$. First note that the formula $\phi(x,y) $ 
is a primitive positive definition of the relation
\[\{(x,y) \mid x = (x_1, x_2), y = (y_1,y_2), x_1=y_1 \} \]
over $(\mathbb{I}; {\sf m})$.
Observe that for $(x_1, x_2), (y_1, y_2)$ in the domain of $I \circ J$, we have $I \circ J ((x_1,x_2), (y_1,y_2)) = I (J(x_1,x_2), J(y_1,y_2)) = (z_1, z_2)$ if and only if $x_1 = z_1$ and $y_1=z_2$. Then
\[I \circ J (x, y) = z\]
is equivalent to the primitive positive formula
\[ \phi(x,z) \wedge z \; {\sf m} \; y,\]
and hence $I \circ J$ is pp-homotopic to the identity interpretation. This finishes the proof that $(\mathbb{I}; {\sf m})$ and $(\Q; <)$ are bi-interpretable.
We note that the proof of Theorem~\ref{thm:rect} generalizes this construction.
\end{ex}

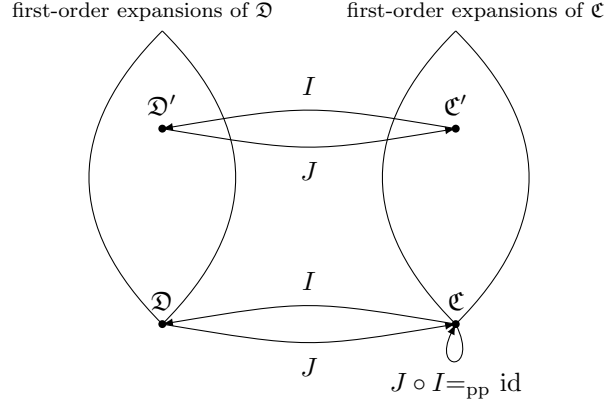
\begin{figure}

\centering
\begin{tikzpicture}[scale=1.3]
\coordinate (D) at (0,0);
\coordinate (U1) at (1cm,1cm);
\coordinate (U2) at (1cm,2cm);
\coordinate (V) at (0,3cm);
\coordinate[label=above: {\footnotesize first-order expansions of $\mathfrak D$}]  (V') at (-0.2cm,3cm);
\coordinate (W1) at (-1cm,1cm);
\coordinate (W2) at (-1cm,2cm);
\coordinate (D') at (0,2cm);
\draw (D) .. controls (U1) and (U2) .. (V);
\draw (V) .. controls (W2) and (W1) .. (D);
\node [circle,fill,inner sep=0pt,minimum size=3pt,label=above:$\mathfrak D$] at (D) {};
\node [circle,fill,inner sep=0pt,minimum size=3pt,label=above:$\mathfrak D'$] at (D') {};

\coordinate (C) at (3cm,0);
\coordinate (X1) at (4cm,1cm);
\coordinate (X2) at (4cm,2cm);
\coordinate(Y) at (3cm,3cm);
\coordinate[label=above: {\footnotesize first-order expansions of $\mathfrak C$}] (Y') at (3.2cm,3cm);
\coordinate (Z1) at (2cm,1cm);
\coordinate (Z2) at (2cm,2cm);
\coordinate (C') at (3cm,2cm);
\draw (C) .. controls (X1) and (X2) .. (Y);
\draw (Y) .. controls (Z2) and (Z1) .. (C);
\node [circle,fill,inner sep=0pt,minimum size=3pt,label=above:$\mathfrak C$] at (C) {};
\node [circle,fill,inner sep=0pt,minimum size=3pt,label=above:$\mathfrak C'$] at (C') {};

\coordinate[label=above: $I$]  (G) at (1.5cm,2.25cm);
\coordinate[label=below: $J$]  (H) at (1.5cm,1.75cm);
\coordinate[label=above: $I$]  (A) at (1.5cm,0.25cm);
\coordinate[label=below: $J$]  (B) at  (1.5cm,-0.25cm); 
\draw[-latex] (C') .. controls (G) .. (D');
\draw[-latex] (D') .. controls (H) .. (C');
\draw[-latex] (C) .. controls (A) .. (D);
\draw[-latex] (D) .. controls (B) .. (C);

\coordinate (P) at (2.75cm,-0.5cm);
\coordinate (R) at (3.25cm,-0.5cm);
\draw[-latex] (C) .. controls (R) and (P) .. (C);

\coordinate[label=below: $J\circ I {=_{\rm pp}}$ id] (Q) at (3cm,-0.4cm);

\end{tikzpicture}
\caption{Visualisation of Theorem~\ref{thm:transfer}. We use the symbol $=_{\rm pp}$ to denote that two interpretations are pp-homotopic.}
\label{fig:visual}
\end{figure}

The proof of Lemma 3.4.1 in~\cite{Book} shows in fact the following statement. 

\begin{theorem}\label{thm:transfer}
Suppose $\bD$ has a primitive positive interpretation $I$ in $\bC$,
and $\bC$ has a primitive positive interpretation $J$ in $\bD$ such that
 $J \circ I$ is pp-homotopic to the identity interpretation of $\bC$.
Then for every first-order expansion $\bC'$ of $\bC$
there is a first-order expansion $\bD'$ of $\bD$
such that $I$ is a primitive positive interpretation of $\bD'$ in $\bC'$ and $J$ is a primitive positive interpretation of $\bC'$ in $\bD'$. The theorem is described 
in Figure~\ref{fig:visual}.
\end{theorem}

In particular, if $\bC$, $\bD$, $\bC'$ and $\bD'$ are as in Theorem~\ref{thm:transfer}, 
and $\bC'$ or $\bD'$ has a finite relational signature (and hence we may assume that both have a finite signature),
then $\Csp(\bC')$ and $\Csp(\bD')$ have the same computational complexity (up to polynomial-time reductions) by Proposition~\ref{prop:pp-int-reduce}. 
Now,
let ${\cal C}$ and ${\cal D}$ denote the sets of first-order expansions of
$\bC$ and $\bD$, respectively, and
assume that the complexity of CSP$(\bD)$
is known for every $\bD \in {\cal D}$.
It follows
that we can deduce the complexity of 
CSP$(\bC)$ for every $\bC \in {\cal C}$.

\section{Applications}
\label{sect:applications}
This section demonstrates that our dichotomy result for first-order expansions of the structure $({\mathbb Q}^n;<_1,=_1,\dots,<_n,=_n)$ from Section~\ref{sect:n-dim} can be combined with the complexity classification transfer result from Section~\ref{sect:transfer}
 to obtain surprisingly strong new classification results.
We obtain full classifications for the complexity of the CSP for first-order expansions of several interesting structures;
an overview can be found in Table~\ref{tb:applications}. Our results have stronger formulations when specialised to binary languages; this yields simple new proofs of known results (Sections~\ref{sec:cdc} and~\ref{sect:allen-m}) and solves long-standing open problems from the field of temporal and spatial reasoning (Sections~\ref{sec:cdc} and~\ref{sect:rect}).
For the basic structures in Table~\ref{tb:applications}, our results show that the CSP for a first-order
expansion $\bB$ is polynomial-time solvable if and only if each relation in $\bB$ can be defined
via an Ord-Horn formula.

\begin{table}
\caption{Overview of results and methods used for studying first-order expansions of the
basic relations (unless otherwise stated). By `Interval Algebra above $\sa,\fa$' we mean a first-order expansion of $(\mathbb I; \sa,\fa)$.}
\begin{tabular}{|l|l|l|} \hline
Formalism & $({\mathbb Q}^n;<_1,=_1,\dots,<_n,=_n)$ classification & Classification transfer\\ \hline
Cardinal Direction Calculus & for $n=2$  & not used \\
Generalized CDC & for general $n$ & not used \\
Interval Algebra& not used & used \\
Interval Algebra above $\sa,\fa$ & for $n=2$ & used\\
Rectangle Algebra & for $n=2$ & used \\
$n$-dimensional Block Algebra &  for general $n$ & used \\ \hline
\end{tabular}
\label{tb:applications}
\end{table}

\subsection{Cardinal Direction Calculus}
\label{sec:cdc}
The {\em Cardinal Direction Calculus}~(CDC)~\cite{LigozatCDC}
is a formalism where the basic objects are 
the points in the plane, i.e., the domain is ${\mathbb Q}^2$.
The basic relations correspond to eight
cardinal directions (North, East, South, West and 
four intermediate ones)
together with the equality relation.
The basic relations can be viewed as pairs $(R^1, R^2)$
for all choices of $R^1, R^2 \in \{<,=,>\}$,
where each relation applies 
to the corresponding coordinate.
The connection
between cardinal directions and pairs 
$(R^1, R^2)$
is described in Table~\ref{tb:cdc}. 
Let $\bC$ denote the structure containing the basic relations of CDC.
The classical formulation of CDC contains all binary relations that are unions of relations in $\bC$. 
In the sequel, we will additionally be interested in the richer set of relations of arbitrary arity that are first-order 
definable in $\bC$.

\begin{table}[b]
    \caption{The basic relations of Cardinal Direction Calculus.}
    \centering
    \begin{tabular}{|c|c|c|c|c|c|c|c|c|}
        \hline
        {\rm =} & {\rm N} & {\rm E} & {\rm S} & {\rm W} & {\rm NE} & {\rm SE} & {\rm SW} & {\rm NW} \\
        \hline
        $(=,=)$ & $(=,>)$ & $(>,=)$ & $(=,<)$ & $(<,=)$ & $(>,>)$  & $(>,<)$  & $(<,<)$  & $(<,>)$  \\
        \hline
    \end{tabular}
    \label{tb:cdc}
\end{table}

\begin{theorem} \label{thm:cdc}
Let $\bB$ be a first-order expansion of $\bC$. 
Then exactly one of the following two cases applies.
\begin{itemize}
\item Each of $\theta_1(\Pol(\bB))$ and
$\theta_2(\Pol(\bB))$ contains 
$\mi_3$, $\min_3$, $\mx_3$, or $\lele_3$, or one of their duals.
In this case, $\Pol(\bB)$ has a pwnu polymorphism.
If the signature of $\bB$ is finite, then $\Csp(\bB)$ is in P.
\item $\Pol(\bB)$ has a 
uniformly continuous minor-preserving map
to $\Pol(K_3)$ and
$\bB$ has a finite-signature reduct whose CSP is NP-complete.
\end{itemize}
\end{theorem}
\begin{proof}
Clearly, every relation in $\bB$ is first-order definable in  $({\mathbb Q}^2;<_1,=_1,<_2,=_2)$. Moreover, 
the relations in $({\mathbb Q}^2;<_1,=_1,<_2,=_2)$ are primitively positively definable in $\bC$:
\begin{itemize}
    \item 
$x <_1 y$ is defined by $\exists z \big(x \mathrel {\rm SW} z \wedge z \mathrel {\rm NW} y \big)$, 
\item $x =_1 y$ is defined by $\exists z \big (x \mathrel {\rm S} z \wedge z \mathrel {\rm N} y \big)$, and 
\item the relations
$<_2$ and $=_2$ can be defined analogously.
\end{itemize}
Thus, the result follows immediately from
Theorem~\ref{thm:main} because every first-order expansion of $\bC$ can be viewed as a first-order expansion of
 $({\mathbb Q}^2;<_1,=_1,<_2,=_2)$.
\end{proof}

A slightly weaker version of the following result has been proved by Ligozat~\cite{LigozatCDC} using a fundamentally different approach.

\begin{corollary} \label{cor:cdc-binary-again} 
Let $\bB$ be a binary first-order expansion of ${\mathfrak C}$. Then
exactly
one of the following cases applies. 
\begin{itemize}
\item Each relation in $\bB$, viewed as a relation of arity four over ${\mathbb Q}$, has an Ord-Horn definition.
In this case, $\bB$ has a pwnu polymorphism 
and $\Csp(\bB)$ is in P. 

\item $\Pol(\bB)$ has a uniformly continuous minor-preserving map to $\Pol(K_3)$ and $\Csp(\bB)$ is NP-complete.
\end{itemize}
\end{corollary}
\begin{proof}
Immediate consequence of Theorem~\ref{thm:cdc-binary}.
\end{proof}

Ligozat~\cite{Ligozat:cosit2001} and Balbiani and Condotta~\cite{Balbiani:Condotta:appint2002} discuss the natural generalisation of CDC to the domain ${\mathbb Q}^n$: let CDC$_n$ denote this generalisation.
Balbiani and Condotta~\cite[Section 7]{Balbiani:Condotta:appint2002} claim that a particular set of relations (referred to as {\em strongly preconvex}) is a maximal tractable subclass in CDC$_3$ that contains all basic relations, and they state that they have not been able to generalise
this result to higher dimensions.
Theorem~\ref{thm:cdc} and Corollary~\ref{cor:cdc-binary-again}
can immediately be generalised to this setting by using our results for $(\mathbb Q;<)^{(n)}$.
We conclude that the Ord-Horn class is the unique maximal tractable subclass of CDC$_n$ ($n \geq 2$) that contains all basic relations.

\subsection{Allen's Interval Algebra}
\label{sect:allen}
We have already introduced Allen's Interval Algebra in Example~\ref{ex:allen} to illustrate interpretations. 
The complexity of all binary reducts of Allen's Interval Algebra have been classified in~\cite{KrokhinAllen}. However, little is known about the complexity of the CSP for first-order reducts of Allen's Interval Algebra. 
In this section we obtain classification results for first-order expansions of some reducts of Allen's Interval Algebra. Recall the relations $\sf m$ (`meets'), $\sf s$ (`starts'), and $\sf f$ (`finishes') from Table~\ref{tb:allen-basic-defs}. For reducts of Allen's Interval Algebra that contain ${\sf m}$, these classification results are an immediate consequence of the transfer result from Section~\ref{sect:transfer} (Section~\ref{sect:allen-m}). For the first-order expansions of the structure that just contains the relations ${\sf s}$ and ${\sf f}$ (Section~\ref{sect:allen-sf}), we combine classification transfer with our classification for the first-order expansions of $({\mathbb Q};<_1,=_1,<_2,=_2)$ from Section~\ref{sect:class}. 

\subsubsection{First-order Expansions of $\{{\sf m}\}$}
\label{sect:allen-m}
The following is a more explicit version
of Theorem 3.4.3 in~\cite{Book} (which only states that the CSP of a first-order expansions of $({\mathbb I};\ma)$ is polynomial-time solvable or NP-complete).

\begin{theorem}\label{thm:allen-m}
Let $\bB$ be a first-order expansion of $({\mathbb I};\ma)$.
 Then exactly one of the following cases applies. 
\begin{itemize}
    \item 
The identity map on $\mathbb I$ is
a 2-dimensional primitive positive interpretation of $\bB$ in $\bU$, $\bX$, $\bI$, or $\bL$. 
In this case, $\bB$ has a pwnu polymorphism, 
and if $\bB$ has a finite signature then $\Csp(\bB)$ is polynomial-time
solvable.
\item $\Pol(\bB)$ has a uniformly continuous minor-preserving map to $\Pol(K_3)$ (and $\bB$ has a finite-signature reduct whose CSP is NP-complete). 
\end{itemize}

\end{theorem}
\begin{proof}
From Example~\ref{ex:allen-bi} we know
that the structures $({\mathbb I};\ma)$ and $({\mathbb Q};<)$ are primitively positively  bi-interpretable via interpretations $I$ and $J$ of dimension 2 and 1, respectively, where $I$ is the identity map on ${\mathbb I}$.  Theorem~\ref{thm:transfer} 
implies that there exists a first-order expansion $\bC$ of $({\mathbb Q};<)$
such that $I$ is a primitive positive interpretation of $\bB$ in $\bC$ and
$J$ is a primitive positive interpretation of $\bC$ in $\bB$.

If $\bC$ has a primitive positive interpretation in $\bU$, $\bX$, $\bI$, $\bL$, then 
 $\bB$ has a primitive positive interpretation in one of those structures. Since each of $\bU$, $\bX$, $\bI$, and $\bL$ has a pwnu polymorphism (Theorem~\ref{thm:pwnu}) and since the existence of pwnu polymorphisms is preserved by primitive positive interpretations (as was discussed in Section~\ref{sect:cat}),
 the structure $\bC$ has a pwnu polymorphism. 
 Furthermore, if $\bB$ has a finite signature, then $\Csp(\bB)$ 
 has a polynomial-time reduction to 
 the CSP of one of those structures by Proposition~\ref{prop:pp-int-reduce}.
 The polynomial-time tractability then follows from Theorem~\ref{thm:pwnu} and Theorem~\ref{thm:tcsps}. 
 Otherwise, 
Theorem~\ref{thm:tcsps} implies
that $\Pol(\bC)$ has a uniformly continuous minor-preserving map to $\Pol(K_3)$. Since there is also a uniformly continuous minor-preserving map from $\Pol(\bB)$ to $\Pol(\bC)$ by Theorem~\ref{thm:wonderland}, we can compose maps and obtain a uniformly continous minor-preserving map from $\Pol(\bB)$ to $\Pol(K_3)$. 
By Corollary~\ref{cor:hard}, $\bB$ has a finite-signature reduct whose CSP is NP-complete.
By Corollary~\ref{cor:excl}, the two cases in the statement of the theorem are mutually exclusive.
\end{proof}

Nebel \& B\"urckert~\cite{Nebel} proved (by a computer-generated proof) that if a reduct $\bB$ of Allen's Interval Algebra only contains relations that have an Ord-Horn definition when considered as a relation of arity four over ${\mathbb Q}$, then $\Csp(\bB)$ is in P. Otherwise, and if 
it contains the relation $\sf m$, it has an NP-hard CSP. Later on, Ligozat~\cite{Ligozat:constr98} presented a mathematical proof of this result. 
We can derive a stronger variant of the results by Nebel \& B\"urckert and Ligozat as a consequence 
of Theorem~\ref{thm:allen-m}.

\begin{theorem}\label{thm:allen-binary}
Let $\bB$ be a binary first-order expansion of $({\mathbb I};\ma)$. Then exactly one of the following cases applies. 
\begin{itemize}
\item Every relation of $\bB$, viewed as a relation of arity four over ${\mathbb Q}$, has an Ord-Horn definition.
In this case, $\bB$ has a pwnu polymorphism and 
$\Csp(\bB)$ is in P. 
 
\item $\Pol(\bB)$ has a uniformly continuous minor-preserving map to $\Pol(K_3)$ and
$\Csp(\bB)$ is NP-complete. 
\end{itemize}
\end{theorem}
\begin{proof}
If $\Pol(\bB)$ has a uniformly continuous minor-preserving map to $\Pol(K_3)$, then
the statement follows from 
Corollary~\ref{cor:hard}.
Otherwise, Theorem~\ref{thm:allen-m} implies that
$\bB$ has a pwnu polymorphism and 
the identity map $I$ on ${\mathbb I}$ is a 2-dimensional primitive positive interpretation in $\bU$, $\bX$, $\bI$, or $\bL$. We claim that every relation $R$ of $\bB$, considered as a relation of arity four over ${\mathbb Q}$, has an Ord-Horn
definition. Let $\phi(u_1,u_2,v_1,v_2)$ 
be the first-order formula that defines $I^{-1}(R)$ over $({\mathbb Q};<)$. 
Note that if $(u_1, u_2)$ is in the domain of $I$, then $u_1<u_2$, so $\phi(u_1, u_2, v_1, v_2)$ implies $u_1 < u_2 \wedge v_1 < v_2$.

We first consider the case that $\bB$ has a 2-dimensional primitive positive interpretation in $\bU$, $\bX$, or $\bI$. In this case, $\phi$ is preserved by a $\pp$-operation (Proposition~\ref{prop:pp}),
and we may assume that $\phi$ has the syntactic form described in Theorem~\ref{thm:pp}. 
Since $\phi$ implies $u_1 < u_2 \wedge v_1 < v_2$, we may add these two conjuncts to $\phi$; note that the resulting formula
is still of the required form. 
We may additionally assume that $\phi$ is reduced, because every formula obtained from $\phi$ by removing literals is again of the required form. 
Each clause in $\phi$ has the form
$$y_1 \neq x \vee \cdots \vee y_k \neq x \vee z_1 \leq x \vee \cdots \vee z_l \leq x,$$
where all the variables are from $\{u_1, u_2, v_1,v_2\}$.
We want to prove that $\phi$ is equivalent to an Ord-Horn formula, which we obtain by showing that $l \leq 1$ in every such clause.

If $u_1 \in \{z_1,\dots,z_l\}$ and $u_2$ equals $x$ 
then the conjunct $u_1 < u_2$ implies that the literal $u_1 \leq u_2$ 
is true in every satisfying assignment to $\phi$, which means that the clause has no other literals by the assumption that $\phi$ is reduced, and we are done. 
If $u_1$ equals $x$ and $u_2$ equals $z_i$, for some $i \in \{1,\dots,l\}$, 
then the conjunct $u_1 < u_2$ implies that
removing 
the literal $z_i \leq x$ would result in an equivalent formula, in contradiction to the assumption that $\phi$ is reduced. 
If $u_1$ equals $z_i$ 
and $u_2$ equals $z_j$ for some $i,j \in \{1,\dots,l\}$, then the clause $u_1 < u_2$ implies that the literal $z_j \leq x$ is redundant, again in contradiction to the assumption that $\phi$ is reduced. 
We see that if one of $u_1,u_2$ is from $\{z_1,\dots,z_l\}$, then the other variable cannot be from $\{x, z_1,\dots,z_l\}$. 
The same argument applies to $v_1$ and $v_2$ and we conclude that 
$l \leq 1$.

Finally we consider the case that $\bB$ has a 2-dimensional primitive positive interpretation in $\bL$. In this case, 
Theorem~\ref{thm:lele} implies that every relation of $\bB$, considered as a relation of arity
four over ${\mathbb Q}$, has a definition $\phi(u_1,u_2,v_1,v_2)$ by a conjunction of ll-Horn clauses
$$x_1 \neq y_1 \vee \cdots \vee x_m \neq y_m \vee z_1 < z_0 \vee \cdots \vee z_{l} < z_0 \vee (z_0 = z_1 = \cdots = z_l),$$
where the final disjunct might be missing and the variables $x_i, y_i, z_i$ are from the set $\{u_1, u_2, v_1, v_2\}$.
Again we may assume that $\phi$ contains the two clauses
$u_1 < u_2$ and $v_1 < v_2$, and we may also assume that $\phi$ is reduced in the sense that whenever we remove a literal $z_i < z_0$ and remove $z_i$ from the final disjunct $z_0=z_1= \cdots = z_l$, 
or if we remove the final disjunct entirely, we obtain a formula which is not equivalent to $\phi$. 
It suffices to show that this implies that  $l \leq 1$. Again, we break into cases.
If both $u_1$ and $u_2$ are in $\{z_0,z_1,\dots,z_l\}$, then the final disjunct is never satisfied, so we may assume that it is not present. 
If $u_1 \in \{z_1,\dots,z_l\}$ and $u_2$ equals $z_0$, then the literal $z_i < z_0$ would be true in every satisfying assignment to $\phi$, which means that the clause has no other literals by the assumption that $\phi$ is reduced, and we are done. If $u_1,u_2 \in \{z_1\dots,z_l\}$ we also obtain a contradiction to the assumption that $\phi$ is reduced, since $u_2 < z_0$ implies $u_1 < z_0$. If $u_1$ equals $z_0$ and $u_2$ equals $z_i$ for some $i \in \{1,\dots,l\}$, then 
the literal $z_i < z_0$ would be false and can be removed, in contradiction to $\phi$ being reduced. Again it follows that at most one of $u_1$ and $u_2$ can appear in $\{z_0, \dots, z_l \}$. 
Analogous reasoning for $v_1$ and $v_2$ implies that $l \leq 1$. 

The disjointness of the two cases follows from the disjointness of the cases in Theorem~\ref{thm:allen-m}.
\end{proof}

\subsubsection{First-order Expansions of $\{{\sf s},{\sf f}\}$}
\label{sect:allen-sf}
Despite the obvious difference between the domains ${\mathbb I}$ and ${\mathbb Q}^2$, there is a way to use our classification of the first-order expansions of $({\mathbb Q}^2;<_1,=_1,<_2,=_2)$ to obtain classification results for 
first-order expansions of $({\mathbb I};\sa,\fa)$.
Our starting point is the following definability result.

\begin{lemma}\label{lem:bi-inter-sf}
 $({\mathbb I};\sa,\fa)$ 
and $({\mathbb Q}^2;<_1,=_1,<_2,=_2)$ are primitively positively bi-definable. 
\end{lemma}

\begin{proof}
Consider the structure $(\mathbb{I}; <_1, =_1, <_2, =_2)$, where the relations $<_i$ and $=_i$ denote restrictions of $<_i$ and $=_i$ on $\mathbb{I} \subseteq \Q^2$.
Note that the relation $\sa$ can be defined over $(\mathbb{I}; <_1, =_1, <_2, =_2)$ by the formula $a =_1 b \wedge a <_2 b$.
Similarly, the relation $\fa$ can be defined over $(\mathbb{I}; <_1, =_1, <_2, =_2)$ by the formula $a =_2 b \wedge b <_1 a$.
The relations in $({\mathbb I};<_1,=_1,<_2,=_2)$ can be primitively positively defined in $(\mathbb I; \sa, \fa)$ as follows.
\begin{align*}
=_1 & = \big\{(a,b) \in \mathbb{I}^2 \mid \exists c \big (
c \mathrel{\sa} a \wedge c \mathrel{\sa} b \big ) \big\}, \\
=_2 & = \big\{(a,b) \in \mathbb{I}^2 \mid \exists c \big (
c \mathrel{\fa} a \wedge c \mathrel{\fa} b \big ) \big\}, \\
<_1 & = \big\{(a,b) \in \mathbb{I}^2 \mid \exists c,d \big (a =_1 c \wedge d \mathrel{\fa} c \wedge d \mathrel{\sa} b \big ) \big\}, \\
<_2 & = \big\{(a,b) \in \mathbb{I}^2 \mid \exists c,d \big (b =_2 c \wedge d \mathrel{\sa} c \wedge d \mathrel{\fa} a \big ) \big\}. 
\end{align*}
It follows that the structures $(\mathbb I; \sa, \fa)$ and $({\mathbb I};<_1,=_1,<_2,=_2)$ are primitively positively interdefinable.

\medskip 

\noindent
{\bf Claim.} The structures $({\mathbb I};<_1,=_1,<_2,=_2)$ and $({\mathbb Q}^2;<_1,=_1,<_2,=_2)$ are isomorphic. 

\smallskip

\noindent
We prove the statement by a back-and-forth argument. Suppose that $i$ is an isomorphism between a finite substructure $\bA$ of $({\mathbb I};<_1,=_1,<_2,=_2)$ and a finite substructure $\bA'$ of $({\mathbb Q}^2;<_1,=_1,<_2,=_2)$. The sets $A$ and $A'$ denote (as usual) the domains of $\bA$ and $\bA'$, respectively.
Let
\begin{align*}
A_1 & \coloneqq \{ p \in {\mathbb Q} \mid \exists q \: (p,q) \in A\} & 
A'_1 & \coloneqq \{ p \in {\mathbb Q} \mid \exists q \: (p,q) \in A'\} 
\\ 
A_2 & \coloneqq \{ q \in {\mathbb Q} \mid \exists p \: (p,q) \in A\}
& A'_2 & \coloneqq \{ q \in {\mathbb Q} \mid \exists p \: (p,q) \in A'\}.
\end{align*}
Define $i_1 \colon A_1 \to A_1'$ by 
setting $i_1(p) = p'$ if there exist $q,q' \in {\mathbb Q}$ such that $i(p,q) = (p',q')$. 
Similarly, define $i_2 \colon A_2 \to A_2'$ by 
setting $i_2(q) = q'$ if there exist $p,p' \in {\mathbb Q}$ such that $i(p,q) = (p',q')$. 
Since $i$ and $i^{-1}$ are isomorphisms, $i_1$ and $i_2$ and their inverses are well-defined bijections and preserve $<$. By the homogeneity of $({\mathbb Q};<)$, there exist automorphisms $\alpha_1$ and $\alpha_2$ that extend $i_1$ and $i_2$.

For going forth, 
let $(a,b) \in {\mathbb I} \setminus A$.
Then $i$ is extended
by setting  $i(a,b) \coloneqq (\alpha_1(a),\alpha_2(b))$. Since $\alpha_1$ and $\alpha_2$ preserve $<$, the extended map $i$  preserves the relations $<_1, =_1, <_2, =_2$.
The operations $\alpha_1$ and $\alpha_2$ are automorphisms of $(\Q;<)$, hence $i$ is injective and $i^{-1}$ preserves $<_1, =_1, <_2, =_2$. Therefore, the extension of $i$ is an isomorphism.

For going back, let $(a',b') \in {\mathbb Q}^2 \setminus A'$.  Then $i$ is extended
by setting  $i(\alpha_{1}^{-1}(a'),\alpha_{2}^{-1}(b')) \coloneqq (a',b')$; to prove that the extension is an isomorphism, we may argue similarly as in the forth step.
Alternating between going back and going forth, we may thus construct an isomorphism between the two countable structures $({\mathbb I};<_1,=_1,<_2,=_2)$ and $({\mathbb Q}^2;<_1,=_1,<_2,=_2)$.
$\hfill \diamond$

\medskip 

\noindent
The claim implies that $({\mathbb Q}^2;<_1,=_1,<_2,=_2)$ is isomorphic to the structure $({\mathbb I};<_1,=_1,<_2,=_2)$, which is primitively positive interdefinable with $(\mathbb I; \sa, \fa)$, i.e.,
$({\mathbb I};\sa,\fa)$ and $({\mathbb Q}^2;<_1,=_1,<_2,=_2)$
are primitively positively bi-definable. 
\end{proof} 

With the aid of Lemma~\ref{lem:bi-inter-sf}, we can now proceed
in a way that is similar to the proof of Theorem~\ref{thm:allen-m}.

\begin{theorem}\label{thm:binary-sf}
Let $\bD$ be a first-order expansion of $({\mathbb I};\sa,\fa)$.
 Then exactly one of the following cases applies. 

\begin{itemize}
\item
$\bD$ has a pwnu polymorphism. If $\bD$ has a 
finite relational signature, then 
$\Csp(\bD)$ is in P.
\item
$\Pol(\bD)$ has a uniformly continuous minor-preserving map to $\Pol(K_3)$. In this case, $\bD$ has a finite-signature reduct whose CSP is NP-complete.
\end{itemize}
\end{theorem}

\begin{proof}
There is a first-order expansion $\bC$ of $({\mathbb Q}^2;<_1,=_1,<_2,=_2)$ such that $\bD$ has a primitive positive interpretation in $\bC$ and vice versa,
by Lemma~\ref{lem:bi-inter-sf} together with Theorem~\ref{thm:transfer}. 
If $\Pol(\bD)$ has a uniformly continuous minor-preserving map to $\Pol(K_3)$, then $\bD$ has a finite-signature reduct whose $\Csp$ is NP-complete by Corollary~\ref{cor:hard}. Otherwise, since there is a uniformly continuous minor-preserving map from $\Pol(\bD)$ to $\Pol(\bC)$ by Theorem~\ref{thm:wonderland}, $\Pol(\bC)$ does not have a uniformly continuous minor-preserving map to $\Pol(K_3)$ as well. By Theorem~\ref{thm:main}, $\bC$ has a pwnu polymorphism and $\Csp(\bC)$ is in P if $\bC$ has a finite signature. Since primitive positive interpretations preserve the existence of pwnu polymorphisms, $\bD$ has a pwnu polymorphism as well. Moreover, if $\bD$ has a finite signature, then $\bC$ can be assumed to have a finite signature and $\Csp(\bD)$ is in P by Proposition~\ref{prop:pp-int-reduce}.
By Corollary~\ref{cor:excl}, the two cases in the statement of the theorem are mutually exclusive.
\end{proof}

We note that relations $\sa$ and $\fa$ are primitively positively definable in $\{\ma\}$ but $\ma$ is not primitively positively definable in $\{\sa,\fa\}$, so
Theorem~\ref{thm:binary-sf} is incomparable to Theorem~\ref{thm:allen-m}.

\subsection{Block Algebra}
\label{sect:rect}

We will now study the $n$-{\em dimensional block algebra}
($\mathfrak{BA}_n$) by Balbiani et al.~\cite{BlockAlgebra}. 
This formalism has become widespread since it can capture directional
information in spatial reasoning, something that the 
classical RCC formalisms~\cite{RandellCuiCohn} cannot.
Let $n \geq 1$ be an integer. The
$n$-dimensional block algebra has 
the domain ${\mathbb I}^{n}$. 
For relations $R^1,\dots,R^n$ from the interval algebra, we write
\[\{((x_1,\dots,x_{n}),(y_1,\dots,y_{n})) \in ({\mathbb I}^{n})^2  \; | \; x_i \mathrel{R^i} y_i, \; 1 \leq i \leq n\}.\] 
The structure $\mathfrak{BA}_n$ contains all such relations, and we say that the relation $(R^1|R^2|\dots|R^n)$ is {\em basic} if $R^1,\dots,R^n$ are basic relations in the interval
algebra.
We note that $\mathfrak{BA}_1$ is the interval algebra and that $\mathfrak{BA}_2$ is often referred
to as the
{\em rectangle algebra} (RA)~\cite{Guesgen,MukerjeeJ90}.

We begin by studying first-order expansions of $({\mathbb I};{\sf m}) \boxtimes ({\mathbb I};{\sf m}) = ({\mathbb I}^2;{\sf m}_1,=_1,{\sf m}_2,=_2)$.
It is easy to see that the relation $=_i$ is primitively positively definable by $\ma_i$, hence it is equivalent to study the first-order expansions of the structure $({\mathbb I}^2;{\sf m}_1,{\sf m}_2)$.
Note that the relation ${\sf m}_1$ and ${\sf m}_2$ over ${\mathbb I}^2$ 
can be written as $({\sf m}|\top)$ and $(\top|{\sf m})$ respectively in the terminology of the Block Algebra. 
Also note that ${\sf m}_1$ and ${\sf m}_2$ are primitively positively definable over the basic relations of the rectangle algebra: 
for example, $\exists z(x \mathrel{({\sf m}|{\sf p})} z \wedge y \mathrel{(\equiv|{\sf p})} z)$ is equivalent to $x \mathrel{({\sf m}|\top)} y$.
The fact that every basic relation in the interval algebra has a primitive positive definition over $({\mathbb I}; {\sf m})$~\cite{Allen:Hayes:ijcai85} now immediately implies that
every RA relation has a primitive positive definition over $({\mathbb I}^2;\ma_1,\ma_2)$.
Hence, the results below imply a classification of the Rectangle Algebra containing the basic relations.  

\begin{theorem}\label{thm:rect}
Let $\bD$ be a first-order expansion of the structure 
$({\mathbb I}^2;\ma_1,\ma_2)$.
Then there exists a first-order expansion $\bC$ of $({\mathbb Q}^2;<_1,=_1,<_2,=_2)$ such that 
$\bD$ has a 2-dimensional primitive positive interpretation in $\bC$
and $\bC$ has a 1-dimensional primitive positive interpretation in $\bD$. 

Furthermore, exactly one of the following two cases applies. 
\begin{itemize}
\item $\bD$ has a pwnu polymorphism. If the signature of $\bD$ is finite, then $\Csp(\bD)$ is in P. 
\item There exists a uniformly continuous minor-preserving map from $\Pol(\bD)$ to $\Pol(K_3)$ and $\bD$ has a finite-signature reduct whose CSP is NP-complete.
\end{itemize}
\end{theorem}
\begin{proof}
For the first part of the statement we apply Theorem~\ref{thm:transfer}; 
so it suffices to prove that 
$({\mathbb I}^2;\ma_1,\ma_2)$ and $({\mathbb Q}^2;<_1,=_1,<_2,=_2)$ are primitively positively bi-interpretable via interpretations of dimension 2 and 1, respectively. This is basically the primitive positive bi-interpretation of $({\mathbb I};\ma)$ and $({\mathbb Q};<)$ from Example 3.3.3 in~\cite{Book} performed in each dimension separately. 
\begin{itemize}
\item There is a 2-dimensional interpretation 
$I$ of $({\mathbb I}^2;\ma_1,\ma_2)$ in $({\mathbb Q}^2;<_1,=_1,<_2,=_2)$
whose domain $U \subseteq (\mathbb Q^2)^2$ has the primitive positive definition
$\delta(a,b)$ given by $a <_1 b \wedge a<_2 b$. 
The interpretation $I \colon U \to {\mathbb I}^2$ is given by 
$$((a_1,a_2),(b_1,b_2)) \mapsto ((a_1,b_1),(a_2,b_2)).$$
The relation
$$I^{-1}(\ma_i) = \{((a,b), (c,d)) \mid I(a,b) \mathrel{\ma_i} I(c,d) \} \subseteq U^{2}$$
has the primitive positive definition
$b =_i c$ in $({\mathbb Q}^2;<_1,=_1,<_2,=_2)$.
The relation $$I^{-1}((\equiv|\equiv)) = \{((a,b), (c,d)) \mid I(a,b) \mathrel{(\equiv|\equiv)} I(c,d) \} \subseteq U^2$$
has the primitive positive definition $a =_1 c \wedge a =_2 c \wedge b =_1 d \wedge b =_2 d$ in $({\mathbb Q}^2;<_1,=_1,<_2,=_2)$.

\item There is a one-dimensional 
interpretation $J:\mathbb{I}^2 \rightarrow \Q^2$ of $({\mathbb Q}^2;<_1,=_1,<_2,=_2)$
in $({\mathbb I}^2;\ma_1,\ma_2)$
given by \[((p_1,p_2),(q_1,q_2)) \mapsto (p_1,q_1).\]
The relation 
\[ J^{-1}(<_i) = \{ (x,y) \mid J(x) <_i J(y)\} \subseteq (\mathbb{I}^2)^2 \]
has the primitive positive definition 
$$ \exists u,v (u \mathrel{\ma_i} x \wedge u \mathrel{\ma_i} v \wedge v \mathrel{\ma_i} y)$$
in $({\mathbb I}^2;\ma_1,\ma_2)$.
The relation $J^{-1}(=_i)$ has the primitive positive definition 
\[\phi_i (x,y) \coloneqq \exists u (u \mathrel{\ma_i} x \wedge u \mathrel{\ma_i} y)\]
in $({\mathbb I}^2;\ma_1,\ma_2)$
so $J^{-1}(=)$
has the definition $\phi_1(x,y) \wedge \phi_2(x,y)$.
\item $J \circ I$ is pp-homotopic to the identity interpretation: 
We have $$J(I(a,b)) = c \text{ if and only if } a = c.$$ 
\item $I \circ J$ is pp-homotopic to the identity interpretation:
Note that $\phi_i((x^1, x^2), (y^1, y^2))$ defines the relation
\[\{((x^1,x^2),(y^1,y^2)) \mid x^i =_1 y^i \} \subseteq (\mathbb{I}^2)^2\]
in $({\mathbb I};\ma_i)$. 
If $x=((x^1_1, x^1_2),(x^2_1, x^2_2))$ and $y=((y^1_1, y^1_2),(y^2_1, y^2_2))$ are elements of ${\mathbb I}^2$, then
\[I(J(x),J(y))=I((x^1_1, x^2_1),(y^1_1,y^2_1))= ((x^1_1, y^1_1), (x^2_1,y^2_1)).\]
Therefore we have $I(J(x),J(y)) \mathrel{(\equiv|\equiv)} z$ if and only if
\[\phi_1(x,z) \wedge z \mathrel{\ma_1} y \wedge \phi_2(x,z) \wedge z \mathrel{\ma_2} y.\]
\end{itemize}
This concludes the proof of the first statement. 

To prove the second statement, suppose that 
there is no uniformly continuous minor-preserving map from $\Pol(\bD)$ to $\Pol(K_3)$---otherwise, we are done by Corollary~\ref{cor:hard}. 
Since there is a primitive positive interpretation of $\bC$ in $\bD$, there is a uniformly continuous clone homomorphism from $\Pol(\bD)$ to $\Pol(\bC)$  by Lemma~\ref{lem:clone-homo}, and there cannot exist a uniformly continuous minor-preserving map from $\Pol(\bC)$ to $\Pol(K_3)$. It now follows 
from 
Theorem~\ref{thm:main} that $\Pol(\bC)$ has a pwnu polymorphism and if $\bC$ has a finite relational signature then $\Csp(\bC)$ is in P. 
By the first statement, there is also a primitive positive interpretation of $\bD$ in $\bC$.
Therefore, 
the polynomial-time tractability of $\Csp(\bD)$ in the finite signature case follows from 
Lemma~\ref{lem:pp-construct-reduce}.
Moreover, there is a clone homomorphism from 
$\Pol(\bC)$ to $\Pol(\bD)$, again by Lemma~\ref{lem:clone-homo}. 
Therefore, $\bD$ has a pwnu polymorphism as well. 
By Corollary~\ref{cor:excl}, the two cases in the statement of the theorem are mutually exclusive.
\end{proof}

We now consider binary first-order expansions of $({\mathbb I}^2;\ma_1,\ma_2)$.
The proof combines arguments from Theorem~\ref{thm:allen-binary} and Theorem~\ref{thm:rect}.

\begin{theorem}\label{thm:proof-of-conjecture}
Let $\bB$ be a binary first-order expansion of $({\mathbb I}^2;\ma_1,\ma_2)$. Then exactly one of the following cases applies. 
\begin{itemize}
\item Every relation of $\bB$, viewed as a relation of arity $8$ over ${\mathbb Q}$, has an Ord-Horn definition.
In this case, $\bB$ has a pwnu polymorphism and 
$\Csp(\bB)$ is in P. 
 
\item $\Pol(\bB)$ has a uniformly continuous minor-preserving map to $\Pol(K_3)$ and $\Csp(\bB)$ is NP-complete. 
\end{itemize}
\end{theorem}

\begin{proof}
Let $\bB'$ be the first-order expansion of $({\mathbb Q}^2;<_1,=_1,<_2,=_2)$ such that the structure
$\bB$ has a 2-dimensional primitive positive interpretation $I$ in $\bB'$
and $\bB'$ has a 1-dimensional primitive positive interpretation in $\bB$ which exists by Theorem~\ref{thm:rect}. By Theorem~\ref{thm:transfer}, $I$ may be taken to be the same interpretation $I \colon U \to {\mathbb I}^2 $ as in the proof of Theorem~\ref{thm:rect}, where $U = \{ (a,b) \in \Q^2 \mid a <_1 b \wedge a<_2 b \}$ and $$I: ((a_1,a_2),(b_1,b_2)) \mapsto ((a_1,b_1),(a_2,b_2)).$$

If $\Pol(\bB)$ has a uniformly continuous minor-preserving map to $\Pol(K_3)$, then $\Csp(\bB)$ is NP-hard by Corollary~\ref{cor:hard}.
Otherwise, it follows from Lemma~\ref{lem:clone-homo} that $\Pol(\bB')$ does not have a uniformly continous minor-preserving map to $\Pol(K_3)$ and Theorem~\ref{thm:main}
implies that
for each $i\in \{1,2\}$ 
there exists $f_i \in \Pol(\bB')$ such that $\theta_i(f_i)$ equals $\min_3$, $\mx_3$, $\mi_3$, $\lele_3$, or one of their duals
and that $\Csp(\bB')$ is in P. By Proposition~\ref{prop:pp-int-reduce}, $\Csp(\bB)$ is in P, too. In this case, $\bB$ has a pwnu polymorphism by Theorem~\ref{thm:rect}; the theorem also implies that the two cases in the statement are mutually exclusive.

It remains to show
that every (binary) relation of $\bB$, considered as a relation of arity $8$ over ${\mathbb Q}$, has an Ord-Horn definition.
Let $R$ be a relation of $\bB$. Observe that it is sufficient to show that the $4$-ary relation $I^{-1}(R)$ has a definition $\phi$ that is a conjunction of clauses of the form
\begin{equation}\label{eq:ord-horn}
    x_1 \neq_{i_1} y_1 \vee \cdots \vee x_m \neq_{i_m} y_m \vee z_1 \; \circ \; z_0,
\end{equation}
where $i_j \in \{1,2\}$, $\circ \in \{<_1,\leq_1,=_1, <_2, \leq_2, =_2\}$, it is permitted that $m=0$ and
the last disjunct may be omitted. With this in mind, $\ve(\phi)$ (as defined just before Proposition~\ref{prop:pp-pow}) is the desired Ord-Horn definition of $R$ viewed as a relation of arity $8$ over $\Q$.
By Lemma~\ref{lem:flip}, we may focus on the situation when $\theta_i(f_i) \in \{\min_3, \mx_3, \mi_3,\lele_3\}$ (since an Ord-Horn definition  with reversed ordering in one of the dimensions results in an Ord-Horn definition again).

Let
$\phi(u_1, u_2, v_1, v_2)$ be the first-order definition of $I^{-1}(R)$ over $(\Q^2; <_1, =_1, <_2, =_2)$. By the definition of $I$, if $I(u_1, u_2)=u$, then $u_1 <_1 u_2 \wedge u_1 <_2 u_2$ and if $I(v_1, v_2)=v$, then $v_1 <_1 v_2 \wedge v_1 <_2 v_2$. Therefore, $\phi$ implies that the four conjuncts above hold.

Suppose first that $\theta_i(f_i)=\lele_3$, $i=1,2$, then, by Proposition~\ref{prop:closed} and Theorem~\ref{thm:lele}, $\theta_i(\Pol(\bB'))$ contains an $\lele$-operation for both $i$.
By Proposition~\ref{prop:pp-pow}, we may assume that $\phi$ is a conjunction of clauses of the form
$$x_1 \neq_{i_1} y_1 \vee \cdots \vee x_m \neq_{i_m} y_m \vee z_1 <_j z_0 \vee \cdots \vee z_{\ell} <_j z_0 \vee (z_0 =_j z_1 =_j \cdots =_j z_{\ell})$$
for $i_1,\dots,i_m,j \in \{1,2\}$ and 
where the last disjunct may be omitted. Since $\phi$ implies $u_1 <_1 u_2$, $u_1 <_2 u_2$, $v_1 <_1 v_2$ and $v_1 <_2 v_2$, we may add these conjuncts to $\phi$ without loss of generality; note that the formula is still of the required form. By an analogous argument as in the proof of Theorem~\ref{thm:allen-binary}, we can assume 
that $\phi$ is of the form
(\ref{eq:ord-horn}).

Next, let $\theta_1(f_1)\in \{\min_3, \mx_3, \mi_3 \}$ and $\theta_2(f_2)=\lele_3$. As in the previous paragraph, $\theta_2(\Pol(\bB'))$ contains an $\lele$-operation. Similarly, by Proposition~\ref{prop:closed} and Proposition~\ref{prop:pp}, $\theta_1(\Pol(\bB'))$ contains a $\pp$-operation.
By Proposition~\ref{prop:syntax-4}, $R$ may be defined by a conjunction of 
weakly $1$-determined clauses of the form 
\begin{align*}
u_1 \neq_2 v_1 \vee \cdots u_m \neq_2 v_m \vee y_1 \neq_1 x \vee \cdots \vee y_k \neq_1 x \vee z_1 \leq_1 x \vee \cdots \vee  z_l \leq_1 x
\end{align*}
together with $2$-determined clauses of the form
\begin{align*}
x_1 \neq_2 y_1 \vee \cdots \vee x_m \neq_2 y_m \vee z_1 <_2 z_0 \vee \cdots \vee z_{\ell} <_2 z_0 \vee (z_0 =_2 z_1 =_2 \cdots =_2 z_{\ell}).
\end{align*}
Again, we may add the implied conjuncts  $u_1 <_1 u_2$, $u_1 <_2 u_2$, $v_1 <_1 v_2$ and $v_1 <_2 v_2$ to the defining formula. Now we may use the same arguments as in the proof of Theorem~\ref{thm:allen-binary} to prove that each of the clauses may be taken to be a clause of the form (\ref{eq:ord-horn}). 
The proof with the two dimensions exchanged is analogous.

Finally, assume that $\theta_i(f_i)\in \{\min_3, \mx_3, \mi_3 \}$ for $i \in \{1,2\}$. We can reason
analogously to the previous paragraph and conclude that $\theta_i(\Pol(\bB'))$, $i \in \{1,2\}$, contains a $\pp$-operation. Furthermore, Theorem~\ref{thm:lele}
implies that we may assume that $\theta_i(\Pol(\bB'))$, $i \in \{1,2\}$, does not contain a $\lex$-operation (otherwise it would contain an $\lele$-operation and this case would be covered by one of the previous cases). By applying Proposition \ref{prop:factor} twice (the second time with the roles of the dimensions exchanged), every relation of $\bB'$ can be defined by a formula
$\psi_1 \wedge \psi_2$ such
that $\psi_i$ is a conjunction of clauses of the form $$ y_1 \neq_i x \vee \cdots \vee y_k \neq_i x \vee z_1 \leq_i x \vee \cdots \vee z_l \leq_i x. $$
As in the previous cases, we may add the conjuncts $u_1 <_1 u_2$, $u_1 <_2 u_2$, $v_1 <_1 v_2$ and $v_1 <_2 v_2$ to the formula $\psi_1 \wedge \psi_2$ without loss of generality. Now we may proceed analogously to the proof of Theorem~\ref{thm:allen-binary} and show that if the formula is reduced, it is of the form (\ref{eq:ord-horn}).
\end{proof}

Balbiani et al.~\cite{Balbiani:etal:ijcai99} have presented a tractable subclass---consisting of the so-called {\em strongly preconvex} relations---of the rectangle algebra. They write the following on page 447.

\begin{quotation}
The subclass generated by the set of the strongly preconvex relations is now the biggest known tractable
set of RA which contains the 169 atomic relations. An
open question is: is this subclass a maximal tractable
subclass which contains the atomic relations? 
\end{quotation}

We answer their
question affirmatively.
We know that every basic RA relation has a primitive positive definition in $({\mathbb I}^2;\ma_1,\ma_2)$, and
we observed in the beginning of this section that the relation
${\sf m}_1$
is primitively positively definable with the aid of the basic relations $({\sf m}|{\sf p})$ and $(\equiv \!|{\sf p})$, and
${\sf m}_2$ is analogously
primitively positively definable from $({\sf p}|{\sf m})$ and $({\pa}| \! \equiv)$.
We have thus proved (via Theorem~\ref{thm:proof-of-conjecture}) that the Rectangle Algebra contains a single maximal subclass
that is polynomial-time solvable and contains all basic relations. The relations in this subclass are
definable via Ord-Horn formulas, and
Balbiani et al.~\cite[Section 6.2]{BlockAlgebra} have proved that strongly preconvex relations coincide with Ord-Horn-definable relations.

We continue by analysing the $n$-dimensional block algebra when $n > 2$.
The approach is similar to the approach used for the rectangle algebra:
we use complexity transfer to deduce a classification for first-order expansions of 
$({\mathbb I}^n;\ma_1,\dots,\ma_n)$
from the classification for 
first-order expansions of
$(\mathbb Q^n;<_1,=_1,\dots,<_n,=_n)$. 
The relations $\ma_1,\dots,\ma_n$ are primitively positively definable from the basic relations of the $n$-dimensional block algebra, so we obtain in particular a classification
of the complexity of the CSP for all first-order expansions of the basic relations of the $n$-dimensional block algebra. 

\begin{theorem}\label{thm:block}
Let $\bD$ be a first-order expansion of the structure 
$({\mathbb I}^n;\ma_1,\dots,\ma_n)$. 
Then 
there exists a first-order expansion $\bC$ of $({\mathbb Q}^n;<_1,=_1,\dots,<_n,=_n)$ such that 
$\bD$ has a 2-dimensional primitive positive interpretation in $\bC$
and $\bC$ has a 1-dimensional primitive positive interpretation in $\bD$. 
Furthermore, exactly one of the following two cases applies. 
\begin{itemize}
  
\item 
$\bD$ has a pwnu polymorphism. 
If the signature of $\bD$ is finite, then $\Csp(\bD)$ is in P.
  \item There exists  a uniformly continuous minor-preserving map from $\Pol(\bD)$ to $\Pol(K_3)$ and $\bD$ has a finite-signature reduct whose $\Csp$ is NP-complete.
\end{itemize}
\end{theorem}
\begin{proof}
Straightforward generalisation of Theorem~\ref{thm:rect}.
\end{proof}

It has been known for a long time that the set of
Ord-Horn-definable relations is a tractable fragment of the $n$-dimensional Block Algebra \cite{BlockAlgebra}. In that article (pp. 907--908), Balbiani et al. note the following.

\begin{quotation}
The problem of the maximality of this tractable subset [Ord-Horn] remains an open problem. Usually to prove the maximality of a fragment of a relational algebra an extensive machine-generated analysis is used. Because of the huge size ... we cannot proceed in the same way.
\end{quotation}

We answer this question in the affirmative: the subset of
relations in $\mathfrak{BA}_n$ that can be viewed as
arity-$4n$ relations with an
Ord-Horn definition,
is a maximal tractable subclass. Furthermore, it is the only maximal subclass that is tractable and contains all basic relations.
To see this, we proceed in the same way as in the analysis of the Rectangle Algebra. First of all, 
every basic relation in $\mathfrak{BA}_n$ has a
primitive positive definition in $({\mathbb I}^n;\ma_1,\dots,\ma_n)$, and
the relations $\ma_1,\dots,\ma_n$ are easily seen to be primitively positive definable in the basic relations of $\mathfrak{BA}_n$.
It follows from the corollary below that the only maximal subclass of $\mathfrak{BA}_n$
that is polynomial-time solvable and contains all basic relations is the Ord-Horn class.

\begin{corollary} \label{cor:block-binary}
Let $\bB$ be a binary first-order expansion of $(\mathbb{I}^n;\ma_1,\dots,\ma_n)$. Then exactly one of the 
following cases applies. 
\begin{itemize}

\item
Each relation in $\bB$, viewed as a relation of arity $4n$ over ${\mathbb Q}$, has an Ord-Horn definition.
In this case, $\bB$ has a pwnu polymorphism and 
$\Csp(\bB)$ is in P. 

\item $\Pol(\bB)$ has a uniformly continuous minor-preserving map to $\Pol(K_3)$ and $\bB$ has a finite-signature reduct whose $\Csp$ is NP-complete.
\end{itemize}
\end{corollary}
\begin{proof}
Recall Theorem~\ref{thm:block} and let $\bB'$ be the first-order expansion of $({\mathbb Q}^n;<_1,=_1, \dots, <_n,=_n)$ such that $\bB$ has a 2-dimensional primitive positive interpretation $I$ in $\bB'$
and $\bB'$ has a 1-dimensional primitive positive interpretation in $\bB$. By combining Corollary~\ref{cor:hard}, Lemma~\ref{lem:clone-homo} and Theorem~\ref{thm:n}, we deduce by the same argumentation as in the proof of Theorem~\ref{thm:proof-of-conjecture} that if the statement in the second item does not hold, then $\bB'$ and $\bB$ have a pwnu polymorphism and $\Csp(\bB')$ and $\Csp(\bB)$ are in P. Moreover, in this case, Theorem~\ref{thm:n} implies that for every $p\in \{1, \dots, n\}$, there is an $f_p\in\Pol(\bB')$ such that $\theta_p(f_p)$ equals $\min_3$, $\mx_3$, $\mi_3$, $\lele_3$, or one of their duals.

It remains to show that if $\Csp(\bB)$ is in P, then every (binary) relation of $\bB$, considered as a relation of arity $4n$ over ${\mathbb Q}$, has an Ord-Horn definition.
Let $R$ be a relation of $\bB$. Observe that, as in the proof of Theorem~\ref{thm:proof-of-conjecture}, it is enough to show that the $4$-ary relation $I^{-1}(R)$ has a definition $\phi(u_1, u_2, v_1, v_2)$ that is a conjunction of clauses of the form
\begin{equation}\label{eq:ord-horn-n}
    x_1 \neq_{i_1} y_1 \vee \cdots \vee x_m \neq_{i_m} y_m \vee z_1 \; \circ \; z_0,
\end{equation}
where $i_j \in \{1,\dots, n\}$, $\circ \in \{<_1,\leq_1,=_1, \dots, <_n, \leq_n, =_n\}$, it is permitted that $m=0$ and
the last disjunct may be omitted; then $\ve(\phi)$ will be the desired Ord-Horn definition of $R$ viewed as a relation of arity $4n$ over $\Q$.

By Lemma~\ref{lem:flip}, we may focus on the situation that $f_p\in\{\min_3, \mx_3, \mi_3, \lele_3\}$ for every $p\in\{1,\dots,n\}$. Therefore, by Proposition~\ref{prop:pp} and Theorem~\ref{thm:lele}, there is a set $S\subseteq \{1,\dots,n\}$ such that, for each $p \in S$, $\theta_p(\Pol(\bB'))$ contains an $\lele$-operation and, for each $p\in \{1,\dots,n\}\setminus S$, $\theta_p(\Pol(\bB'))$ contains a $\pp$-operation but not a $\lex$-operation. We may therefore assume that $\phi$ has the syntactic form described in Proposition~\ref{prop:strong-syntax-n}. Moreover, we may assume that $\phi$ contains the conjuncts $u_1 <_j u_2$ and $v_1 <_j v_2$, $j=1,\dots,n$, since these are implied by $\phi$ (see the proof of Theorem~\ref{thm:proof-of-conjecture} for more details). For each of the two types of clauses that appear in $\phi$, we may use the same case distinction as in Theorems~\ref{thm:allen-binary}
and \ref{thm:proof-of-conjecture} to show that each of the clauses is of the form (\ref{eq:ord-horn-n}). This concludes the proof.
\end{proof}

Balbiani et al. \cite[p. 908]{BlockAlgebra} also raise the following question:

\begin{quotation}
...the question also arises as to how the qualitative constraints [Block Algebra] we have been considering could be integrated
into a more general setting to include metric constraints.
\end{quotation}

If one focuses on tractable subclasses that contain all basic relations, then such an integration is indeed possible.
Since our results imply that the relations in such a tractable subclass must be definable via Ord-Horn formulas, they can immediately be embedded into the metric framework
suggested by Jonsson and Bäckström~\cite{JonssonBaeckstroem} and Koubarakis~\cite{Koubarakis}.
Under the same assumptions, this holds for the cardinal direction calculus and Allen's Interval Algebra, too.

\section{Conclusions and Open Problems}
\label{sect:conc}
We proved that the CSPs for first-order expansions of $({\mathbb Q};<)^{(n)}$ satisfy a complexity dichotomy: they are in P or NP-complete. Using a general complexity transfer method, we prove that first-order expansions of the basic relations of the
cardinal direction calculus, Allen's Interval Algebra, and 
the $n$-dimensional block algebra have a CSP complexity dichotomy. Less obviously, the complexity transfer method can also be applied to show that first-order expansions of the relations ${\sf s}$ and ${\sf f}$ of Allen's Interval Algebra have a complexity dichotomy. 
All of the results can be specialised for binary signatures, in which case we obtain new and conceptually simple proofs of results that have first been shown with the help of a computer (Theorem~\ref{thm:allen-binary}) 
or that answer several questions from the literature (Section~\ref{sect:rect}). 

Our results also imply that the so-called \emph{meta-problem} of complexity classification is decidable: given finitely many first-order formulas that define a first-order expansion $\bD$ of one of the structures for which we obtained a complexity classification, one can effectively decide whether $\Csp(\bD)$ is in P or NP-complete.
This follows from the general fact that for homogeneous finitely bounded structures that are model-complete cores and have an extremely amenable automorphism group (all of these assumptions are satisfied by our structures) the condition of the tractability conjecture (see Corollary~\ref{cor:hard}) can be decided effectively (essentially by checking exhaustively for the existence of a \emph{diagonally canonical pseudo Siggers polymorphism}); since these results are not new we refer to~\cite[Section 11.6]{Book} for details. 

One may wonder about first-order reducts of 
$({\mathbb Q};<)^{(n)}$ 
rather than just first-order expansions.
Classifying the complexity of the CSP for this class of structures will be a challenging project. It is easy to see 
that every CSP for a finite-domain structure can be formulated in this way. Also all first-order reducts of the infinite Johnson graphs $J(\omega,n)$ fall into this class  (for example, the line graph of the countably infinite clique when $n=2$).
Such a project would also
include the following interesting classification problem.

The \emph{age} of a relational structure $\bB$ is the class of all finite structures that embed into $\bB$. It follows from Fra{\"i}ss\'e's theorem (or by a direct back-and-forth argument) that two homogeneous structures with the same age are isomorphic (see, e.g.~\cite[Theorem 6.1.2]{Hodges}).
Let $\bA_1$ and $\bA_2$ be homogeneous structures with disjoint relational signatures $\tau_1$ and $\tau_2$ and without algebraicity (see~\cite{Hodges}; the structure $({\mathbb Q};<)$ is an example of such a structure without algebraicity). 
It is well known that there exists an up to isomorphism unique countable homogeneous $(\tau_1 \cup \tau_2)$-structure whose age consists of all structures whose $\tau_1$-reduct is in the age of $\bA_1$ and whose $\tau_2$-reduct is in the age of $\bA_2$;
this structure is called the \emph{generic combination} of $\bA_1$ and $\bA_2$, and will be denoted by $\bA_1 * \bA_2$. 
It can be shown by a back-and-forth argument that the $\tau_1$-reduct of $\bA_1 * \bA_2$ is isomorphic to $\bA_1$ and the $\tau_2$-reduct is isomorphic to $\bA_2$. 
The notion of generic combinations can be defined also for $\omega$-categorical structures without algebraicity~\cite{BodirskyGreinerCombinations}: 
$\bA_1 * \bA_2$ is then 
defined as the $(\tau_1 \cup \tau_2)$-reduct of the generic combination of a homogeneous expansion of $\bA_1$ and a homogeneous expansion of $\bA_2$ (it can be shown that this is well-defined).
If $\bA_1$ and $\bA_2$ are first-order reducts of $({\mathbb Q};<)$
then the complexity of $\Csp(\bA_1 * \bA_2)$ has been classified recently~\cite{BodirskyGreinerRydval}. 
A complexity classification of $\Csp(\bB)$
for first-order reducts $\bB$ of $({\mathbb Q};<) * ({\mathbb Q};<)$, however, is open. 

\begin{proposition}
For every first-order reduct $\bB$ of $({\mathbb Q};<) * ({\mathbb Q};<)$, there exists a first-order reduct $\bC$ of
$({\mathbb Q};<) \boxtimes ({\mathbb Q};<)$
such that $\bC$ is homomorphically equivalent to 
$\bB$. 
\end{proposition}
\begin{proof}
Let $\bA$ be 
the $\{<_1,<_2\}$-reduct of 
$({\mathbb Q};<) \boxtimes ({\mathbb Q};<)$. 
For $i \in \{1,2\}$ there exists an isomorphism $\alpha_i$ between the $\{<_i\}$-reduct of $({\mathbb Q};<) * ({\mathbb Q};<)$ and $({\mathbb Q};<)$.
Then $e \colon d \mapsto (\alpha_1(d),\alpha_2(d))$
is an embedding of 
$({\mathbb Q};<) * ({\mathbb Q};<)$ into $\bA$. 
Conversely, if we fix a linear extension of $<_1$ and $<_2$ in
$({\mathbb Q};<) \boxtimes ({\mathbb Q};<)$, then the $\{<_1,<_2\}$-reduct of the resulting structure embeds into $({\mathbb Q};<) * ({\mathbb Q};<)$ (see, e.g., \cite[Lemma 4.1.7]{Book}).
This shows that $\bA$ has an injective homomorphism $h$ to $({\mathbb Q};<) * ({\mathbb Q};<)$.

Let $\bB$ be a first-order reduct of 
$({\mathbb Q};<) * ({\mathbb Q};<)$. 
Every first-order formula
$\phi$ over
$({\mathbb Q};<) * ({\mathbb Q};<)$ is equivalent to a
quantifier-free formula in conjunctive normal form, because $({\mathbb Q};<) * ({\mathbb Q};<)$ is homogeneous and $\omega$-categorical. 
Replace each atomic subformula of $\phi$ of the form  
$\neg (x <_i y)$, for $i \in \{1,2\}$,  by $y <_i x \vee x=y$. Then replace each subformula of the form 
$x \neq y$ by $x <_1 y \vee y <_1 x$.
The resulting formula is equivalent over $({\mathbb Q};<) * ({\mathbb Q};<)$.
Each formula that defines a relation of $\bB$ and is written in this form can be interpreted over $({\mathbb Q};<) \boxtimes ({\mathbb Q};<)$ instead of $({\mathbb Q};<) * ({\mathbb Q};<)$; 
let $\bC$ be the obtained first-order reduct of $({\mathbb Q};<) \boxtimes ({\mathbb Q};<)$.

Since the first-order definitions of the relations of $\bB$ are quantifier-free, the embedding $e$ of 
$({\mathbb Q};<) * ({\mathbb Q};<)$ into $\bA$ 
is also an embedding of $\bB$ into $\bC$. 
We claim that  
$h$ is a homomorphism from $\bC$ to $\bB$. 
This follows from the fact that  $h$ is a  homomorphism from $({\mathbb Q};<) \boxtimes ({\mathbb Q};<)$ to $\bA$
and that the defining formulas for the relations of $\bB$ and $\bC$ do not involve negation.
\end{proof}

Another way forward is to study basic structures other than $({\mathbb Q};<)$. Here, temporal reasoning is a source of examples with applications in, for instance, AI.
An important time model used in temporal reasoning
is {\em branching time}, where for every point in time the past is linearly ordered, but
the future is partially ordered.
This motivates the so-called {\em left-linear point algebra} \cite{Duentsch,HirschAlgebraicLogic}, which is a relation
algebra with four basic relations, denoted by $=$, $<$, $>$, and $|$. Here, $x|y$ signifies that
$x$ and $y$ are incomparable in time, and `$x < y$' signifies that $x$ is earlier in time than
$y$.
The branching-time satisfiability problem can be formulated
as CSP$(\bB)$ for an $\omega$-categorical structure $\bB$~\cite{BodirskyNesetrilJLC}.
One possible concrete description of the structure $\bB$, described by 
Adeleke and Neumann~\cite{MR1388893}, is to let $\bB=(B;<,|,=)$ where $B$ is the set of finite sequences of rational
numbers. For arbitrary $a=(a_1,\dots,a_n)$ and $b=(b_1,\dots,b_m)$ in $B$ with $n \leq m$, $a < b$ holds if
one of the following conditions hold:
\begin{enumerate}
\item 
$n < m$ and $a_i=b_i$ for $1 \leq i \leq n$, or
\item
$a_i=b_i$ for $1 \leq i < n$ and $a_n < b_n$.
\end{enumerate}
The remaining relations are defined in the obvious way.

Branching time has been used, for example, in automated planning~\cite{Dean:Boddy:aij88}, as the basis for temporal logics~\cite{Emerson:Halpern:jacm86}, and as the basis for a generalisation of Allen’s Interval Algebra~\cite{Ragni:Wolfl:sc2004}.  
In particular, the complexity of the branching variant\footnote{Various ways of defining the formalism are possible~\cite{Ragni:Wolfl:sc2004}. We restrict our attention to the most well-known end-point-based formalism that contains 19 basic relations.} of Allen’s Interval Algebra has recently gained
attraction~\cite{Bertagnon:etal:ic2021,Bertagnon:etal:cilc2020,Bertagnon:etal:time2020,Gavanelli:etal:time2018,Durhan:Sciavicco:ic2018}.
Some complexity results for the branching interval algebra $\bB$ are presented in these publications, but the big picture is missing, even for first-order expansions of
the basic relations. Our classification transfer result (Theorem~\ref{thm:transfer}) is applicable to this problem, but since 
there is currently no complexity classification of
the CSPs for first-order reducts of ${\mathfrak B}$, we cannot
present a full classification for CSPs of first-order expansions of the basic branching interval relations.
Naturally, there is no polymorphism-based description of the tractable fragments either, so we cannot
analyse the complexity of CSPs for first-order expansions of the structure  $(B^n;<_1,|_1,=_1,\dots,<_n,|_n,=_n)$ in the style of 
Theorem~\ref{thm:n}.
However, given a complexity classification of the CSPs for first-order expansions of $\bB$
in place, then Theorem~\ref{thm:transfer} is immediately applicable, and we do not
see any fundamental problem that prevents us from generalising Theorem~\ref{thm:n} to expansions of 
$(B^n;<_1,|_1,=_1,\dots,<_n,|_n,=_n)$ as long as the tractable fragments can be described via polymorphisms and nice syntactic normal forms.

A related time model encountered in computer science is {\em partially ordered time} (po-time). This model has
various applications in, for instance, the analysis of concurrent and distributed systems~\cite{Anger:toplas89,Lamport:jacm86}.
In po-time, both the past and the future of a time point are partially ordered. 
This implies that time becomes a partial order
with four basic relations $=$, $<$, $>$ and $|$, 
signifying ``equal'', ``before'', ``after'' and ``unrelated'', respectively.
The satisfiability problem for po-time can convieniently be formulated with the 
{\em random partial order} $(P;<)$, and
Kompatscher and Van Pham~\cite{Kompatscher:VanPham:flap2018} have presented a full complexity 
classification of the CSP for all first-order reducts of the random partial order.
Combined with Theorem~\ref{thm:transfer}, this gives us a full classification
of the CSP for first-order expansions of the basic relations of the po-time analogue of
the interval algebra. This generalisation of the interval algebra has been studied by Zapata et al.~\cite{Zapata:etal:sc2013}.
Kompatscher and Van Pham describe the tractable fragments of the random partial order with the aid
of polymorphisms. 
Hence, it seems conceivable that their result can be generalised to a complexity classification of
the CSP for first-order expansions of $(P^n;<_1,|_1,=_1,\dots,<_n,|_n,=_n)$ by utilising the ideas
behind the proof of Theorem~\ref{thm:n}.

\section*{Acknowledgements}
We thank Johannes Greiner for his comments on an earlier version of this article and Jakub Rydval for letting us use his picture in Figure~\ref{fig:ppll}.

\bibliographystyle{siamplain} 
\bibliography{local}

\end{document}